\documentclass[a4paper,12pt]{article} %%% Version of October 27, 2008

\usepackage {amsmath, amsfonts, amssymb, amscd, amsthm, eucal}

\usepackage{amsbsy}
\usepackage[all]{xy}
\SelectTips{cm}{}

\theoremstyle{plain}{
  \newtheorem{theorem}{Theorem}[section]
  \newtheorem{definition}[theorem]{Definition}
  \newtheorem{corollary}[theorem]{Corollary}
  \newtheorem{lemma}[theorem]{Lemma}
  \newtheorem{proposition}[theorem]{Proposition}
  \newtheorem{claim}[theorem]{Claim}
  \newtheorem{notation}[theorem]{Notation}}

\theoremstyle{definition}{
  \newtheorem{remark}[theorem]{Remark}
  \newtheorem{example}[theorem]{Example}
  \newtheorem{note}[theorem]{Note}}

\newcommand{\id}{\mathrm{id}}
\newcommand{\Spec}{\mathrm{Spec}}
\newcommand{\BGL}{\mathrm{BGL}}
\newcommand{\Hom}{\operatorname{Hom}}
\newcommand{\Ker}{\operatorname{Ker}}

\newcommand{\Id}{\mathrm{Id}}
\newcommand{\Nis}{\mathrm{Nis}}
\newcommand{\fib}{\mathrm{fib}}
\newcommand{\cofrepl}{\mathrm{cof}}
\newcommand{\op}{\mathrm{op}}
\newcommand{\cm}{\mathrm{cm}}
\newcommand{\Fr}{\mathrm{Fr}}
\newcommand{\Tr}{\mathrm{Tr}}

\newcommand{\Vect}{\mathrm{Vect}}
\newcommand{\tlder}{\mathcal{L}}
\newcommand{\trder}{\mathcal{R}}
\newcommand{\Sing}{\operatorname{Sing}}
\newcommand{\Ho}{\operatorname{Ho}}
\newcommand{\Func}{\operatorname{Fun}}
\newcommand{\Th}{\mathrm{Th}}
\newcommand{\MGL}{\mathbf{MGL}}
\newcommand{\SmOp}{\mathcal Sm\mathcal Op}
\newcommand{\Sm}{\mathcal Sm}

\newcommand{\Sp}{\mathrm{Sp}}

\newcommand{\Aff}{\mathbf{A}}
\newcommand{\Pro}{\mathbf{P}}
\newcommand{\Gr}{\mathrm{Gr}}

\newcommand{\K}{\mathcal K}

\newcommand{\ZZ}{\mathbb{Z}}
\newcommand{\one}{\mathbb{I}}

\newcommand{\real}{\mathbf{R}}
\newcommand{\M}{\mathbf{M}}
\newcommand{\MS}{\mathbf{MS}}
\newcommand{\MSS}{\mathbf{MSS}}
\newcommand{\sSet}{\mathbf{sSet}}
\newcommand{\Top}{\mathbf{Top}}

\newcommand{\an}{\mathrm{an}}
\newcommand{\SH}{\mathrm{SH}}

\newcommand{\iso}{\cong}
\newcommand{\union}{\cup}

\newcommand{\CC}{\mathbb{C}}

\newcommand{\colim}{\operatornamewithlimits{colim}}

\newcommand{\intHom}{\underline{\Hom}}
\newcommand{\Cyl}{\mathrm{Cyl}}
\newcommand{\defeq}{\colon \!\! =}
\newcommand{\push}{\Box}
\newcommand{\cof}{\mathrm{cof}}

\newcommand{\xra}{\xrightarrow }
\newcommand{\ra}{\rightarrow }
\newcommand{\hra}{\hookrightarrow }
\newcommand{\xla}{\xleftarrow }

\begin{document}

\title{On Voevodsky's algebraic {$K$}-theory spectrum $\mathrm{BGL}$}

\author{I.~Panin\footnote{Steklov Institute of Mathematics at St.~Petersburg, Russia}
\footnote{Universit\"at Bielefeld, SFB 701, Bielefeld, Germany}
\and K.~Pimenov\footnotemark[1]
\and O.~R{\"o}ndigs\footnote{Institut f\"ur Mathematik, Universit\"at Osnabr\"uck, Osnabr\"uck, Germany}}

\date{October 27, 2008}

\maketitle

\begin{abstract}
Under a certain normalization assumption we prove that the
$\Pro^1$-spectrum $\mathrm{BGL}$
of Voevodsky which represents algebraic $K$-theory is unique over
$\Spec(\mathbb{Z})$.
Following an idea of Voevodsky, we equip the $\Pro^1$-spectrum
$\mathrm{BGL}$
with the structure of a commutative
$\Pro^1$-ring spectrum in the motivic
stable homotopy category. Furthermore, we prove that
under a certain normalization assumption
this ring structure is unique over
$\Spec(\mathbb{Z})$.
For an arbitrary Noetherian scheme $S$ of finite Krull
dimension we pull this structure back
to obtain a distinguished monoidal structure on
$\mathrm{BGL}$. This monoidal structure is relevant for
our proof of the motivic Conner-Floyd theorem~\cite{PPR:conner-floyd}.
It has also been used to obtain a motivic version of Snaith's theorem
\cite{Gepner-Snaith}.
\end{abstract}

\section{Acknowledgements}

The authors thank
the SFB-701 at the Universit\"at Bielefeld, the
RTN-Network HPRN-CT-2002-00287, the
RFFI-grant 03-01-00633a,
INTAS-05-1000008-8118 and the Fields Institute for
Research in Mathematical Sciences for their support.

\section{Preliminaries}
\label{WeWant}

This paper is concerned with results in motivic homotopy theory,
which was put on firm foundations by Morel and Voevodsky
in \cite{MV} and \cite{V1}.
Due to technical reasons explained below, the setup in \cite{MV},
as well as other model categories used in motivic homotopy theory,
are inconvenient for our purposes, so we decided to
pursue a slightly different approach.
We refer to the Appendix~\ref{sec:motiv-homot-theory}
for the basic terminology, notation, constructions, definitions,
and results concerning motivic homotopy theory.
For a Noetherian scheme $S$
of finite Krull dimension we write
$\M(S)$,
$\M_\bullet(S)$,
$\mathrm{H}_\bullet (S)$
and
$\SH(S)$
for
the category of motivic spaces,
the category of pointed motivic spaces,
the pointed motivic homotopy category
and
the stable motivic homotopy category over $S$.
These categories are equipped with symmetric monoidal structures.
In particular, a symmetric monoidal structure
$(\wedge,\one)$
is constructed on the motivic stable homotopy category
and its basic properties are proved.
This structure is used extensively over the present text.

Let $S$ be a regular scheme, and let
$K^0(S)$ denote the Grothendieck group of vector bundles over $S$.
Morel and Voevodsky proved in
\cite[Thm.~4.3.13]{MV}
that the Thomason-Trobaugh $K$-theory
\cite{TT}
is represented
in the pointed motivic homotopy category
$\mathrm{H}_\bullet (S)$
by the space
$\mathbb{Z} \times \mathrm{Gr}$ pointed by
$(0,x_0)$.
Here
$\mathrm{Gr}$
is the union of the finite Grassmann varieties
$\bigcup_{n=0}^{\infty} \mathrm{Gr}(n,2n)$, considered as motivic spaces.
There is a unique element
$\xi_{\infty} \in K^0(\mathbb{Z} \times \mathrm{Gr})$
which corresponds to the identity morphism
$\mathrm{id}\colon \mathbb{Z} \times \mathrm{Gr} \to \mathbb{Z} \times \mathrm{Gr}$.
It follows that there exists a unique morphism
$$
\mu_{\otimes}\colon (\mathbb{Z} \times \mathrm{Gr}) \wedge (\mathbb{Z} \times
 \mathrm{Gr}) \to \mathbb{Z} \times \mathrm{Gr}
$$
in
$\mathrm{H}_\bullet (S)$
such that the composition
$(\mathbb{Z} \times \mathrm{Gr}) \times (\mathbb{Z} \times \mathrm{Gr})
\to
(\mathbb{Z} \times \mathrm{Gr}) \wedge (\mathbb{Z} \times \mathrm{Gr})
\xra{\mu_{\otimes}}
\mathbb{Z} \times \mathrm{Gr}$
represents the element
$\xi_{\infty} \otimes \xi_{\infty}$
in
$K^0\bigl((\mathbb{Z} \times \mathrm{Gr}) \times (\mathbb{Z} \times \mathrm{Gr})\bigr)$
(see Lemma~\ref{KofXsmashY}).
Let
$e_{\otimes}\colon S^0 \to \mathbb{Z} \times \mathrm{Gr}$
be the map which corresponds to the point
$(1,x_0)\in \mathbb{Z}\times \mathrm{Gr}$.
The triple
\begin{equation}
(\mathbb{Z} \times \mathrm{Gr},\mu_{\otimes}, e_{\otimes})
\end{equation}
is a commutative monoid in
$\mathrm{H}_\bullet (S)$.

Using this fact, Voevodsky constructed in
\cite{V1}
a $\Pro^1$-spectrum
$$\mathrm{BGL}= (\K_0,\K_1,\K_2, \dots)$$
with structure maps $e_i\colon \K_i\wedge \Pro^1 \to \K_{i+1}$
such that
\begin{itemize}
  \item[(i)] there is a motivic weak
    equivalence $w\colon \mathbb{Z} \times \mathrm{Gr}\to \K_0$,
    and for all $i$ one has
    $\K_i=\K_0$ and $e_i=e_0$,
  \item[(ii)] the morphism
    \begin{equation*} \xymatrix{
      \!\ZZ\times \mathrm{Gr} \times \Pro^1\! \ar[r]^-{\mathrm{can}} &
      \!(\ZZ\times \mathrm{Gr})\wedge \Pro^1\! \ar[r]^-{w\wedge \Pro^1} &
      \K_i \wedge \Pro^1\! \ar[r]^{e_i} &\K_{i+1}\! \ar[r]^-{w^{-1}}&
      \ZZ\times \mathrm{Gr}}
    \end{equation*}
    in $\mathrm{H}_\bullet (S)$
    represents the element
    $\xi_{\infty} \otimes ([\mathcal{O}(-1)]-[\mathcal{O}]) \in
    K^0(\ZZ \times \mathrm{Gr} \times \Pro^1)$,
  \item[(iii)] and the adjoint
    $\K_i \to \Omega_{\Pro^1}(\K_{i+1})$
    of $e_i$ is a motivic weak equivalence.
\end{itemize}

With this spectrum in hand given a smooth $X$ over $S$ we may identify
$K^0(X)$ with $\mathrm{BGL}^{2i,i}(X)$ as follows
\begin{equation}
K^0(X)=\Hom_{\mathrm{H}_\bullet(S)}(X_+,\mathbb{Z} \times \mathrm{Gr})=
\Hom_{\mathrm{H}_\bullet(S)}(X_+,\K_i)=\mathrm{BGL}^{2i,i}(X)
\end{equation}

Our first aim is to recall Voevodsky's
construction to show that this spectrum
is essentially unique. This has also been obtained
in \cite{Riou}. {\em Our second and more important aim is
to give a commutative monoidal structure to the\/} $\Pro^1${\it-spectrum\/}
$\mathrm{BGL}$ {\em which respects the naive multiplicative
structure on the functor\/} $X \mapsto \mathrm{BGL}^{2\ast,\ast}(X)$.
To be more precise, we construct a product
\begin{equation}
\mu_{\mathrm{BGL}}\colon \mathrm{BGL} \wedge \mathrm{BGL} \to \mathrm{BGL}
\end{equation}
in the stable motivic homotopy category
$\SH(S)$ such that for any $X\in \Sm/S$ the diagram
\begin{equation*} \xymatrix@C=3em{
    K^0(X) \times K^0(X)  \ar[r]^-{\otimes} \ar[d]_-\iso&K^0(X) \ar[d]^-\iso \\
    \mathrm{BGL}^{2i,i}(X)\times\mathrm{BGL}^{2j,j}(X)\ar[r]^-{\mu_{\mathrm{BGL}}}
   & \mathrm{BGL}^{2(i+j),i+j}(X)}
\end{equation*}
commutes.
We show in Theorem~\ref{MonoidalStrOnBGL}  that
{\it there is a unique product\/}
\[\mu_{\mathrm{BGL}}\in \Hom_{\SH(\mathbb{Z})}
   (\mathrm{BGL} \wedge \mathrm{BGL}, \mathrm{BGL})\]
satisfying this property.
This induces a product
$\mu_{\mathrm{BGL}}\in \Hom_{\SH(S)}(\mathrm{BGL} \wedge \mathrm{BGL}, \mathrm{BGL})$
for an arbitrary regular scheme $S$
by pull-back along the structural morphism
$S \to \Spec (\mathbb{Z})$.
As well, we show that the product is associative, commutative and unital.
The resulting multiplicative structure on
the bigraded theory $\mathrm{BGL}^{\ast,\ast}$
coincides with the Waldhausen multiplicative structure
on the Thomason-Trobaugh $K$-theory.

\subsection{Recollections on motivic homotopy theory}
\label{Recallection}

The basic definitions, constructions and model structures used in the text
are given in Appendix~\ref{sec:motiv-homot-theory}. The
word ``model structure'' is used in its modern sense and thus
refers to a ``closed model structure'' as originally defined by
Quillen.
Let $S$ be a Noetherian finite-dimensional scheme.
A {\em motivic space over\/} $S$ is a simplicial presheaf on
the site $\Sm/S$ of smooth quasi-projective $S$-schemes.
A {\em pointed motivic space over\/} $S$ is a pointed simplicial presheaf on
the site $\Sm/S$.
We write
$\M_\bullet(S)$
for the category of pointed motivic spaces over $S$.
A {\it closed motivic model structure} $\M_\bullet^{\cm}(S)$
is constructed in~\ref{thm:closed-motivic-model}.
The adjective ``closed'' refers to the fact
that closed embeddings in $\Sm/S$ are forced to become cofibrations.
The resulting homotopy category
$\mathrm{H}^{\cm}_\bullet (S)$
obtained in Theorem~\ref{thm:closed-motivic-model}
is called {\it the motivic homotopy category} of $S$.
By Theorem~\ref{thm:comparison-mv} it
is equivalent to
the Morel-Voevodsky $\Aff^1$-homotopy category \cite{MV}, and
we may drop the superscript in $\mathrm{H}^\cm_\bullet(S)$ for convenience.
The closed motivic model structure
has the properties that
\begin{enumerate}
  \item\label{item:5}
    for any closed $S$-point $x_0\colon S \hra X$ in a smooth
    $S$-scheme, the pointed motivic space $(X,x_0)$ is cofibrant
    in $\M_\bullet^\cm(S)$  (Lemma~\ref{lem:closed-emb-cof}),
  \item\label{item:6} a morphism $f\colon S\to S^\prime$ of base schemes
    induces a left
    Quillen functor $f^\ast\colon \M_\bullet^\cm(S^\prime)\to \M_\bullet^\cm(S)$
    (Theorem~\ref{thm:closed-motivic-model}), and
  \item\label{item:7} taking complex points is a left Quillen functor
    $\real_\CC\colon \M_\bullet^\cm(\CC)\to \Top_\bullet$
    (Theorem \ref{thm:realization}).
\end{enumerate}
Conditions~\ref{item:6} and~\ref{item:7} do not hold for the
Morel-Voevodsky model structure, condition~\ref{item:5} fails for
the so-called projective model structure \cite[Thm.~2.12]{DRO:motivic}.
For a morphism
$f\colon A \to B$
of pointed motivic spaces
we will write
$[f]$
for the class of $f$ in
$\mathrm{H}_\bullet (S)$.

We will consider
$\Pro^1$
as a pointed motivic space over $S$ pointed by
$\infty \colon S\hra  \Pro^1$.
A $\Pro^1$-{\em spectrum\/} $E$ over $S$ consists
of a sequence $E_0,E_1,\ldots$ of pointed motivic spaces over $S$,
together with structure maps $\sigma_n\colon E_n\wedge \Pro^1 \to E_{n+1}$.
A map of $\Pro^1$-spectra is a sequence of maps of pointed motivic
spaces which is compatible with the structure maps.
Let $\SH(S)$
denote the homotopy category of
$\Pro^1$-spectra, as described in Section~\ref{sec:spectra}.
By Theorem~\ref{thm:stable-model} it is canonically equivalent
to the motivic stable homotopy category
constructed in \cite{V1} and \cite{J}.
As we will see below there exists an essentially unique
$\Pro^1$-spectrum $\mathrm{BGL}$ over
$S = \Spec (\mathbb{Z})$
satisfying properties $(i)$ and $(ii)$ from Section~\ref{WeWant}.
In the following, we will construct
$\mathrm{BGL}$
in a slightly different way than Voevodsky did originally in \cite{V1}.
In order to achieve this, we begin with a description of the known
monoidal structure on the Thomason-Trobaugh K-theory \cite{TT}.

\subsection{A construction of $\mathrm{BGL}$}
\label{ConstructionOfBGL}

Let $S$ be a regular scheme. For every
$S$-scheme $X$ consider the category
$\mathrm{Big}(X)$ of big vector bundles over $X$
(see for instance \cite{FS}
for the definition and basic properties).
The assignments
$X \mapsto \mathrm{Big}(X)$
and
$(f\colon Y \to X) \mapsto f^\ast\colon \mathrm{Big}(X) \to \mathrm{Big}(Y)$
form a functor from schemes to the category of small
categories. The reason is that there is an equality
$(f \circ g)^\ast = g^\ast \circ f^\ast$,
not just a unique natural isomorphism.
In what follows we will always consider
the Waldhausen $K$-theory space of
$X$ as the space obtained by
applying Waldhausen's $\mathcal{S}_\bullet$-construction \cite{W}
applied to the category
$\mathrm{Big}(X)$
rather than to the category
$\Vect(X)$ of usual vector bundles on $X$.
This has the advantage that
the assignment taking an $S$-scheme $X$ to
the Waldhausen $K$-theory space of
$X$
becomes a functor on the category of
$S$-schemes, and in particular a pointed motivic
space over $S$.
In what follows a category $\SmOp/S$ will be useful as well.
Its objects are pairs
$(X,U)$ with an $X \in \Sm/S$ and an open $U$ in $X$. Morphisms $(X,U)$ to $(Y,V)$
are morphisms of $X$ to $Y$ in $\Sm/S$ which take $U$ to $V$.

Let $\mathbb{K}^W$ be
the pointed motivic space defined in~\ref{ex:ktt-fibrant}.
It has the properties that it is fibrant
in $\M_\bullet^{\cm}(S)$ and that $\mathbb{K}^W(X)$ is
naturally weakly equivalent to
the Waldhausen $K$-theory space associated to
the category of big vector bundles on $X$.
For $X\in \Sm/S$ the simplicial set $\mathbb{K}^W(X)$
is thus a Kan simplicial set having the Waldhausen
$K$-theory groups
$K^W_\ast(X)$ as its homotopy groups.
These $K$-theory groups coincide with
Quillen's higher $K$-theory groups
\cite[Thm.1.11.2]{TT}.
We write $K_\ast(X)$ for $K^W_\ast(X)$.
It follows immediately from the adjunction isomorphism
\begin{equation}\label{eq:24}
  \Hom_{\mathrm{H}_\bullet(S)}(S^{p,0}\wedge X_+,\mathbb{K}^W)
   \iso \Hom_{\mathrm{H}_\bullet}\bigl(S^p,\mathbb{K}^W(X)\bigr)  = K_p(X)
\end{equation}
that $\mathbb{K}^W$,
regarded as an object in the motivic homotopy category
$\mathrm{H}_\bullet (S)$
(see~\ref{thm:closed-motivic-model})
represents the Quillen K-theory on $\Sm/S$.
Here $S^n = S^{n,0}$ denotes the $n$-fold smash
product of the constant simplicial presheaf $\Delta^1/\partial \Delta^1$
with itself.
For a pointed motivic space $A$ set
$$
K_p(A): =
\Hom_{\mathrm{H}_\bullet(S)}(S^{p,0}\wedge A,\mathbb{K}^W).
$$

For $X\in \Sm/S$ and a closed subset $Z\hra X$,
$K_n(X \ on \ Z)$ denotes the $n$-th Thomason-Trobaugh
$K$-group of perfect complexes on $X$ with support on $Z$
\cite[Defn.3.1]{TT}.
For $A=X/(X-Z)$ with an $X\in \Sm/S$ and a closed subset $Z \subset X$,
there is an isomorphism
$K_p(A) \cong K_p(X on Z)$
natural in the pair $(X, X-Z)$
(see \cite[Thm.5.1]{TT}).
It follows immediately that
$\mathbb{K}^{W}$,
regarded as an object in the motivic homotopy category
$\mathrm{H}_\bullet (S)$
(see~\ref{thm:closed-motivic-model})
represents the Thomason-Trobaugh $K(X \ on \ Z)$-theory on $\SmOp/S$.
The known monoidal structure
\cite[(3.15.4)]{TT} on the Thomason-Trobaugh $K(X \ on \ Z)$-theory
coincides with the one induced by the Waldhausen monoid
$(\mathbb{K}^W, \mu^{W}, e^{W})$
described below.

Using the notation of
Example~\ref{ex:ktt-fibrant},
consider the diagram
\begin{equation}
K^W(X) \wedge K^W(X) = \Omega^1_s(W_1(X)) \wedge \Omega^1_s(W_1(X)) \xra{m}
\Omega^2_s(W_2(X)) \xla{ad} K^W(X)
\end{equation}
with the Waldhausen multiplication $m$
and the adjunction weak equivalence $ad$
described in \cite[p.~342]{W}. The diagram defines a morphism
\begin{equation}
\mathbb{K}^W \wedge \mathbb{K}^W \xra{\mu^{W}} \mathbb{K}^W
\end{equation}
in
$\mathrm{H}_\bullet (S)$
which is the Waldhausen multiplication on
$\mathbb{K}^W$. Together with the unit morphism
$e^{W}\colon S^0 \to \mathbb{K}^W$
it forms the Waldhausen monoid
$(\mathbb{K}^W, \mu^{W}, e^{W})$.

By
\cite[Thm.~4.3.13]{MV}
there is an isomorphism
$\psi \colon \mathbb{Z} \times \mathrm{Gr} \to \mathbb{K}^{W}$
in
$\mathrm{H}_\bullet (S)$.
The pointed motivic space
$\bigl(\mathbb{Z} \times \mathrm{Gr},(0,x_0)\bigr)$
is closed cofibrant by Lemma~\ref{lem:closed-emb-cof}.
Let
$b^{W}\colon \Pro^1 \to \mathbb{K}^{W}$
be a morphism in
$\mathrm{H}_\bullet(S)$
representing the class
$[\mathcal{O}(-1)]-[\mathcal{O}]$
in the kernel of the homomorphism
$\infty^\ast\colon K_0(\Pro^1)\to K_0(k)$.

\begin{definition}
\label{BGL}
Choose a 
pointed motivic space
$\K$, together with a weak equivalence
$i\colon \mathbb{Z} \times \mathrm{Gr} \to \K$
in
$\M^\cm_\bullet(S)$,
as well as a  morphism
$\epsilon\colon \K \wedge \Pro^1 \to \K$
in $\M^\cm_\bullet(S)$
which descends to
\begin{equation*}
  \mu^{W} \circ (\id \wedge b^{W})\colon \mathbb{K}^{W} \wedge
  \Pro^1 \to \mathbb{K}^{W}
\end{equation*}
under the identification of
$\K$
with
$\mathbb{K}^{W}$
in
$\textrm{H}_\bullet(S)$
via the isomorphism
$\psi \circ [i]^{-1}$.
Define
$\mathrm{BGL}$ as the $\Pro^1$-spectrum of the form
$(\K_0,\K_1,\K_2, \dots)$
with
$\K_i=\K$
for all $i$ and with the structure maps
$e_i\colon \K_i \wedge \Pro^1 \to \K_{i+1}$
equal to the map
$\epsilon\colon \K \wedge \Pro^1 \to \K$.
\end{definition}

$\Pro^1$-spectra  as described in Definition~\ref{BGL} will be used
extensively below.

\begin{remark}
The Voevodsky spectrum $\textbf{BGL}$
is obtained if
$\K = Ex^{\Aff^1}(\mathbb{Z} \times \mathrm{Gr})$,
$i\colon \mathbb{Z} \times \mathrm{Gr} \to Ex^{\Aff^1}(\mathbb{Z} \times \mathrm{Gr})$
is the Voevodsky fibrant replacement morphism in the model structure described in
\cite[Thm.~3.7]{V1}
and the structure map
\begin{equation*} e\colon Ex^{\Aff^1}(\mathbb{Z} \times \mathrm{Gr}) \wedge \Pro^1
        \to Ex^{\Aff^1}(\mathbb{Z} \times \mathrm{Gr})
\end{equation*}
described in
\cite[Section 6.2]{V1}. By \cite[Thm.~3.6]{V1}  and
Note~\ref{note:Quillen-eq-MV}, the map
$\mathbb{Z} \times \mathrm{Gr} \to Ex^{\Aff^1}(\mathbb{Z} \times \mathrm{Gr})$
is also a weak equivalence in $\M_\bullet^\cm(S)$. In particular,
$\mathbf{BGL}$ is an example of a $\Pro^1$-spectrum as
described in Definition~\ref{BGL}.
\end{remark}

\begin{remark}
\label{MuBar}
The Waldhausen structure of a commutative monoid on
$\mathbb{K}^{W}$ in
$\mathrm{H}_\bullet(S)$
induces via the
isomorphism
$\psi \circ [i]^{-1}$
the structure of a commutative monoid
$(\K,\bar \mu, \bar e$)) on the motivic space $\K$ in
$\mathrm{H}_\bullet (S)$
such that
$\psi \circ [i]^{-1}$
is an
isomorphism of monoids.
The composition of the inclusion
$\Pro^1 =\mathrm{Gr}(1,2) \hra \lbrace 0\rbrace \times
\mathrm{Gr} \hra \mathbb{Z}\times \mathrm{Gr}$
and the weak equivalence $i$
is denoted
$b\colon \Pro^1 \to \K$.
Clearly
$[\epsilon] = {\bar \mu} \circ [(\id \wedge b)]$
in
$\mathrm{H}_\bullet(S)$.
\end{remark}

\begin{lemma}\label{CofibrantK}
  Given $\K$, $i\colon \mathbb{Z} \times \mathrm{Gr} \to \K$
  and
  $\epsilon\colon \K \wedge \Pro^1 \to \K$
  fulfilling the conditions of Definition
  \ref{BGL},
  there exist
  $\K^{\prime}$, $i^{\prime} \colon \mathbb{Z} \times \mathrm{Gr} \to \K^{\prime}$,
  $\epsilon^{\prime} \colon \K^{\prime} \wedge \Pro^1 \to \K^{\prime}$
  and
  $q \colon \K^{\prime} \to \K$
  such that
  \begin{itemize}
  \item
    $i^{\prime}$ and $\epsilon^{\prime}$
    fulfil the condition of Definition
    \ref{BGL},
  \item
    $\K^{\prime}$
    is cofibrant in
    $\M^\cm_\bullet(S)$,
  \item
    $q$ is a weak equivalence and
    $q \circ \epsilon^{\prime} = \epsilon \circ (q \wedge \id)$.
  \end{itemize}
  Thus the $\Pro^1$-spectra
  $\mathrm{BGL}^{\prime}$ and $\mathrm{BGL}$
  are weakly equivalent via the morphism given by the
  sequence of maps of pointed motivic spaces
  $q, q, q, \dotsc $.
\end{lemma}

\begin{proof}
  Decompose $i$ as $q \circ i^{\prime}$ with a trivial cofibration
  $i^{\prime} \colon \mathbb{Z} \times \mathrm{Gr} \to \K^{\prime}$
  and a fibration
  $q \colon \K^{\prime} \to \K$. Note that $q$ is a weak equivalence since so are
  $i^{\prime}$ and $i$. The pointed motivic space $\K^{\prime}$ is cofibrant in
  $\M^\cm_\bullet(S)$, because so is $\mathbb{Z}\times \mathrm{Gr}$.
  Furthermore $i^{\prime}$ is a weak equivalence. It remains to construct
  $\epsilon^{\prime}$. Consider the commutative diagram of pointed motivic spaces
  \begin{equation*}
    \xymatrix{
    \bullet   \ar[r] \ar[d]   &   \K^{\prime}  \ar[d]^- q \\
    \K^{\prime} \wedge \Pro^1 \ar[r]^-{\epsilon \circ (q \wedge \id)}  &  \K}
  \end{equation*}
  The left vertical arrow is a cofibration and the right hand side one is
  a trivial fibration. Thus there exists a map
  $\epsilon^{\prime} \colon \K^{\prime} \wedge \Pro^1 \to \K^{\prime}$
  of pointed motivic spaces making the diagram commutative.
\end{proof}

\begin{remark}
\label{PullBackBGL}
  Let $f \colon S^{\prime} \to S$ be a morphism of 
  schemes and let
  $\textrm{BGL}=(\K,\K,\K, \dotsc )$ be a $\Pro^1$-spectrum over $S$ as
  described in Definition~\ref{BGL}. Suppose further that $\K$
  is cofibrant in $\M^\cm_\bullet(S)$.
  The $\Pro^1$-spectrum
  $f^\ast(\textrm{BGL})$ over $S^\prime$ is given by
  $(\K^{\prime},\K^{\prime},\K^{\prime}, \dotsc )$, where
  $\K^\prime = f^\ast \K$ and the structure map is
  \begin{equation*}
    \epsilon^\prime\colon f^\ast \K \wedge \Pro^1_{S^\prime} \iso
    f^\ast(\K \wedge \Pro^1_S) \xrightarrow{f^\ast(\epsilon)} \K
  \end{equation*}
  Since $f^\ast \colon \M^\cm_S \to \M^\cm_{S^\prime}$ is a left Quillen
  functor by Theorem~\ref{thm:closed-motivic-model},
  $f^\ast(\mathrm{BGL})$ satisfies the conditions of Definition~\ref{BGL}
  in $M_\bullet(S^{\prime})$ and $H_\bullet(S^{\prime})$
  provided that $S^{\prime}$ is regular.
  If $S^{\prime}$ is noetherian finite dimensional, then by
  \cite[Thm.~6.9]{V1}
  and Remark
  \ref{AllBGLareIsomorphic}
  $f^\ast(\mathrm{BGL})$
  represents the homotopy invariant $K$-theory as introduced in~\cite{We}.
\end{remark}

It will be proved in Section~\ref{UniquenessOfBGL}
that in the case
$S=\Spec(\mathbb{Z})$
there is essentially just one $\Pro^1$-spectrum
$\textrm{BGL}$ in $\SH(S)$.
In the next section, we will construct
a monoidal structure on
$\mathrm{BGL}$
regarded as an object in the stable homotopy category
$\SH(S)$.
In the case of
$S=\Spec(\mathbb{Z})$
such a monoidal structure is unique.
Pulling it back via the structural morphism
$S^{\prime} \xrightarrow{f} \Spec(\mathbb{Z})$
we get a monoidal structure on
$f^\ast(\mathrm{BGL})$
in
$\mathrm{SH}(S^{\prime})$
for an arbitrary Noetherian finite-dimensional base scheme
$S^{\prime}$.

To complete this section we prove certain properties of
$\mathrm{BGL}$. It turns out that if $\K$ is fibrant
in $\M_\bullet^\cm(S)$, then
$\mathrm{BGL}$ is stably fibrant as a $\Pro^1$-spectrum.
In other words,
$\mathrm{BGL}$
is an
$\Omega_{\Pro^1}$-spectrum
which represents the
Thomason-Trobaugh K-theory on
$\Sm/S$.
For $X\in \Sm/S$ we abbreviate $\mathrm{BGL}(X_+)$ as
$\mathrm{BGL}(X)$, which forces us to write
$\mathrm{BGL}(X,x_0)$ for a pointed $S$-scheme $(X,x_0)$.

\begin{lemma}
\label{OmegaSpectrum}
Let $X\in \Sm/S$ and $n\geq 0$. The adjoint of
the structure map $\epsilon\colon \K \wedge \Pro^1 \to \K$
induces an isomorphism
\begin{equation*}
  \Hom_{\textrm{H}_\bullet(S)}(S^{n,0}\wedge X_+, \K_i)
  \to \Hom_{\textrm{H}_\bullet(S)}(S^{n,0}\wedge X_+ \wedge \Pro^1, \K_{i+1}).
\end{equation*}
In particular, if
$\K$ is fibrant
in  $\M_\bullet^\cm(S)$, then
$\mathrm{BGL}$
is stably fibrant. 
\end{lemma}

\begin{proof}
Recall that for $Y\in \Sm/S$ and a closed subset $Z\hra Y$,
$K_n(Y \ on \ Z)$ denotes the $n$-th Thomason-Trobaugh
$K$-group of perfect complexes on $Y$ with support on $Z$.
It may be obtained as the $n$-th homotopy group of the
homotopy fiber of the map $\mathbb{K}^W(Y)\to \mathbb{K}^W(Y\smallsetminus Z)$.
Abbreviate $\Hom_{\mathrm{H}_\bullet(S)}(-,-)$ as $[-,-]$.
For each smooth $X$ over $S$ the map
\begin{align*}
 K_n(X) & =  [S^{n,0}\smash X_+, \K_i] \to [S^{n,0}\smash X_+ \wedge \Pro^1, \K_{i}
 \wedge \Pro^1] \\
 & \to   [S^{n,0}\smash X_+ \wedge \Pro^1, \K_{i+1}] \iso
K_n(X \times \Pro^1 \ on \ X\times \lbrace \infty \rbrace)
\end{align*}
induced by the structure map
$e_i$ coincides with the multiplication by the class
$[\mathcal{O}(-1)]-[\mathcal{O}]$ in
$K_0(\Pro^1 \ on \ \lbrace \infty \rbrace)$.
This multiplication is known to be an isomorphism
for the Thomason-Trobaugh K-groups, by the projective bundle
theorem \cite[Thm.~4.1]{TT} for $X \times \Pro^1$.
Whence the Lemma.
\end{proof}

In the following statement, the
notation $\Sigma_{\Pro^1}^\infty A (-i)$ will be used
for the $\Pro^1$-spectrum
$\Fr_i A = (\ast,\dotsc,\ast,A,A\wedge \Pro^1,\dotsc,)$
associated to a pointed motivic space $A$ in
Example~\ref{ex:suspension-spectrum}.
Note that $\Sigma_{\Pro^1}^\infty A (-i) \iso \Sigma_{\Pro^1}^\infty A
\wedge S^{-2i,-i}$ in $\SH(S)$, as mentioned in Notation~\ref{not:bigrading}.

\begin{corollary}
\label{AdjunctionIso}
For each pointed motivic space $A$ over $S$ the adjunction map
\begin{equation*}
  \Hom_{\mathrm{H}_\bullet(S)}(A,\K_0) \to
  \Hom_{\SH(S)}(\Sigma^{\infty}_{\Pro^1}A, \mathrm{BGL})
\end{equation*}
is an isomorphism. In particular, for every smooth scheme $X$ over
$S$ and each closed subscheme $Z$ in $X$ one has
$K_p(X \ on\  Z)=\mathrm{BGL}^{-p,0}\bigl(X/(X \smallsetminus Z)\bigr)$.
The family of these isomorphisms form an isomorphism
$Ad: K_\ast \to \mathrm{BGL}^{-*,0}$
of cohomology theories
on the category $\SmOp/S$ in the sense of
\cite{PSorcoh}.
Moreover the adjunction map
$[A,\K_i] \to [\Sigma^{\infty}_{\Pro^1}(A)(-i), \mathrm{BGL}]$
is an isomorphism. In particular, for every smooth scheme $X$ over
$S$ and each closed subscheme $Z$ in $X$ one has
$K_p(X \ on \ Z)=\mathrm{BGL}^{2i-p,i}(X/(X\smallsetminus Z))$.
\end{corollary}

The family of pairings
$\K_i \wedge \K_j \xra{\mu_{ij}}  \K_{i+j}$
in
$\textrm{H}_\bullet(S)$
with
$\mu_{ij}=\bar \mu$
from Remark
\ref{MuBar}
defines a family of pairings
\begin{equation}
\label{NaiveProducts}
\cup\colon \mathrm{BGL}^{p,i}(A) \otimes \mathrm{BGL}^{q,j}(B)
\to \mathrm{BGL}^{p+i,q+j}(A \wedge B)
\end{equation}
for pointed motivic spaces $A$ and $B$. We will refer to the
latter as the
{\em naive product structure\/}
on the functor
$\mathrm{BGL}^{\ast,\ast}$
on the category
$\M_{\ast}(S)$.
It has the following property.

\begin{corollary}
\label{RingIsomorphism}
The isomorphism
$Ad\colon K_\ast \to \mathrm{BGL}^{-*,0}$
of cohomology theories
on $\SmOp/S$
is an isomorphism of ring cohomology theories
in the sense of
\cite{PSorcoh}.
\end{corollary}

\subsection{The periodicity element}
\label{sec:periodicity-element}

The aim of this Section is to construct an element
$\beta \in \mathrm{BGL}^{2,1}(S)$,
to show that it is invertible and to check that
for any pointed motivic space $A$ one has
$$
\mathrm{BGL}^{*,0}(A)[\beta, \beta^{-1}] = \mathrm{BGL}^{*,*}(A)
$$
(the Laurent polynomials over $\mathrm{BGL}^{*,0}(A)$).
We will use the naive product structure on
$\mathrm{BGL}$
described just above
Corollary~\ref{RingIsomorphism}.

\begin{definition}
Set
$\beta\colon \!\!=[S^0 \xra{\bar e} \K=\K_1] \in \mathrm{BGL}^{2,1}(S)$,
where $\bar e$ is the unit of the monoid~$\K$ (see Remark~\ref{MuBar}).
\end{definition}

\begin{lemma}
Let
$b\colon \Pro^1 \hra \K$
be the map described in Remark~\ref{MuBar}.
It represents the element
$[\mathcal{O}(-1)]-[\mathcal{O}]$
in
$\mathrm{BGL}^{0,0}(\Pro^1,\infty) = \Ker\bigl(\infty^\ast \colon K_0(\Pro^1)
\to K_0(S)\bigr)$.
There is a relation
\begin{equation}
\label{BottAndSuspensionOne}
\beta \cup \bigl([\mathcal{O}(-1)]-[\mathcal{O}]\bigr)=
\Sigma_{\Pro^1}(1) \in \mathrm{BGL}^{2,1}(\Pro^1,\infty),
\end{equation}
where
$\Sigma_{\Pro^1}$
is the suspension isomorphism and
$1 \in \mathrm{BGL}^{0,0}(S)$
is the unit. There is another relation
\begin{equation}
\label{BottAndSuspensionTwo}
\beta \cup \bigl([\mathcal{O}(1)]-[\mathcal{O}]\bigr)=
- \Sigma_{\Pro^1}(1) \in \mathrm{BGL}^{2,1}(\Pro^1,\infty).
\end{equation}
\end{lemma}

\begin{proof}
The element
$\Sigma_{\Pro^1}(1)$
is represented by the morphism
$$
S^0 \wedge \Pro^1 \xra{\bar e \wedge \id} \K_0 \wedge \Pro^1
\xra{\id \wedge b} \K_0 \wedge \K_1 \xra{\mu_{01}} \K_1,
$$
where $\mu_{ij}$ is defined just above
(\ref{NaiveProducts}), and $\bar e$ is the unit of the
monoid $\K=\K_0$.
The element
$\beta \cup \bigl([\mathcal{O}(-1)]-[\mathcal{O}]\bigr)$
is represented by the morphism
$$
S^0 \wedge \Pro^1 \xra{\bar e \wedge b} \K_1 \wedge \K_0 \xra{\mu_{10}} \K_1.
$$
Since
$\K_0=\K=\K_1$
one has
$(\id \wedge b) \circ (\bar e \wedge \id) = \bar e \wedge b$.
This implies the relation~(\ref{BottAndSuspensionOne})
since
$\mu_{10}=\mu_{01}$.
Relation~(\ref{BottAndSuspensionTwo}) follows from the first one
since
$[\mathcal{O}(1)]-[\mathcal{O}]= - [\mathcal{O}(-1)]+[\mathcal{O}]$
in
$K^0(\Pro^1)$.
\end{proof}

\begin{lemma}
\label{Invertibility}
Let
$u \in \mathrm{BGL}^{-2,-1}(S)$
be the unique element such that
$\Sigma_{\Pro^1}(u)=[\mathcal{O}(-1)]-[\mathcal{O}]$
in
$\mathrm{BGL}^{0,0}(\Pro^1,\infty)$.
Then
$\beta \cup u=1$.
\end{lemma}

\begin{proof}
Consider the commutative diagram
\begin{equation*}
  \xymatrix{
    \mathrm{BGL}^{2,1}(S) \otimes \mathrm{BGL}^{0,0}(\Pro^1,\infty)
    \ar[r]^-{\cup} &   \mathrm{BGL}^{2,1}(\Pro^1,\infty)     \\
    \mathrm{BGL}^{2,1}(S) \otimes \mathrm{BGL}^{-2,-1}(S)
    \ar[r]^-{\cup} \ar[u]^-{\id \otimes \Sigma_{\Pro^1}}& \mathrm{BGL}^{0,0}(S).
    \ar[u]_-{\Sigma_{\Pro^1}}}
\end{equation*}
Now the Lemma follows from the relation
(\ref{BottAndSuspensionOne}).
\end{proof}

\begin{definition}
\label{AlgebraicDiagonal}
For $\Pro^1$-spectra $E$ and $F$ set
$E^{\mathrm{alg}}(F)= \oplus^{+ \infty}_{- \infty}E^{2i,i}(F)$.
\end{definition}

\begin{proposition}
For every pointed motivic space $A$ the map
\begin{equation}
\label{BGLandTT}
\mathrm{BGL}^{*,0}(A) \otimes_{K_0(S)} \mathrm{BGL}^{\mathrm{alg}}(S) \to
\mathrm{BGL}^{*,*}(A)
\end{equation}
given by
$a \otimes b \mapsto a \cup b$
is a ring isomorphism and
$\mathrm{BGL}^{\mathrm{alg}}(S)= K_0(S)[\beta, \beta^{-1}]$
is the Laurent polynomial ring. One can rewrite this ring isomorphism
as
\begin{equation}
\label{BGLandTT2}
\mathrm{BGL}^{*,0}(A)[\beta,\beta^{-1}] \cong
\mathrm{BGL}^{*,*}(A)
\end{equation}
\end{proposition}

\begin{proof}
In fact,
$\mathrm{BGL}^{*,0}(A) \xra{\cup \beta} \mathrm{BGL}^{*+2,1}(A)$
is an isomorphism since $\beta$ is invertible.
Since
$\mathrm{BGL}^{0,0}(S)=K^0(S)$
the map
(\ref{BGLandTT})
is a ring isomorphism.
\end{proof}

Using
the isomorphism
$Ad\colon K_\ast \to \mathrm{BGL}^{-\ast,0}$
of ring cohomology theories
from Corollary~\ref{RingIsomorphism}
we get the following statement.

\begin{corollary}
\label{LaurantPolinomial}
For every $X\in \Sm/S$ and every closed subset $Z\hra X$ one has
\begin{equation}
\label{BGLAndTTrings1}
K_{- \ast}(X \ on \ Z)[\beta, \beta^{-1}] \cong \mathrm{BGL}^{*,*}_Z(X).
\end{equation}
The family of these isomorphisms form an isomorphism of
ring cohomology theories on $\SmOp/S$ in the sense of
\cite{PSorcoh}.
As well, there is an isomorphism
\begin{equation}
\label{BGLAndTTrings2}
K_{- \ast}(X \ on \ Z) = \mathrm{BGL}^{*,*}_Z(X)/(\beta + 1)\mathrm{BGL}^{*,*}_Z(X).
\end{equation}
The family of  these isomorphisms form an isomorphism of
ring cohomology theories on $\SmOp/S$ in the same sense.
\end{corollary}

\subsection{Uniqueness of $\mathrm{BGL}$}
\label{UniquenessOfBGL}

We prove in this Section that, at least over
$S = \Spec(\mathbb{Z})$,
a $\Pro^1$-spectrum
$\mathrm{BGL}$ as described in Definition~\ref{BGL}
is essentially unique
regarded as an object in the stable homotopy category
$\SH(S)$. This has also been obtained in \cite{Riou}.

Let
$\mathrm{BGL}$
be a
$\Pro^1$-spectrum as described in Definition
\ref{BGL}. Recall that this involves the choice of
a weak equivalence $i \colon \mathbb{Z}\times \mathrm{Gr}\to
\mathcal{K}$ and a structure map $\epsilon\colon
\mathcal{K}\wedge \Pro^1\to \mathcal{K}$.
Let $\mathrm{BGL}^\prime$ be a possibly different
$\Pro^1$-spectrum. More precisely,
take a
pointed motivic space
$\K^{\prime}$
together with a weak equivalence
$i^{\prime}\colon \mathbb{Z} \times \mathrm{Gr} \to \K^{\prime}$
and with a morphism
$\epsilon^{\prime}\colon \K^{\prime} \wedge \Pro^1 \to \K^{\prime}$
which descends to
$$\mu^{W} \circ (\id \wedge b^{W})\colon \mathbb{K}^{W} \wedge \Pro^1 \to \mathbb{K}^{W}$$
under the identification of
$\K^{\prime}$
with
$\mathbb{K}^{W}$
in
$\textrm{H}_\bullet(S)$
via the isomorphism
$\psi \circ [i^{\prime}]^{-1}$.
Let
$\mathrm{BGL}^{\prime}$
be the $\Pro^1$-spectrum of the form
$(\K^{\prime}_0,\K^{\prime}_1,\K^{\prime}_2, \dots)$
with
$\K^{\prime}_i=\K^{\prime}$
for all $i$, and with the structure maps
$\epsilon^{\prime}_i\colon \K^{\prime}_i \wedge \Pro^1 \to \K^{\prime}_{i+1}$
equal to the map
$\epsilon^{\prime}\colon  \K^{\prime} \wedge \Pro^1 \to \K^{\prime}$.

\begin{proposition}
\label{BGLandBGL}
Let $S=\Spec(\mathbb{Z})$. There exists a unique morphism
$\theta\colon \mathrm{BGL} \to \mathrm{BGL}^{\prime}$
in
$\mathrm{SH}(S)$
such that for every integer
$i \geq 0$
the diagram
\begin{equation*}
  \xymatrix{
    \Sigma^{\infty}_{\Pro^1}\K_i(-i)   \ar[r]^-{u_i}
    \ar[d]_-{\Sigma^{\infty}_{\Pro^1} \phi_i (-i)} &
    \mathrm{BGL} \ar[d]^-\theta \\
    \Sigma^{\infty}_{\Pro^1}\K^{\prime}_i(-i)\ar[r]^-{u^{\prime}_i}
    &\mathrm{BGL}^{\prime}}
\end{equation*}
commutes in
$\mathrm{SH}(S)$,
where
$\phi_i= i^{\prime} \circ i^{-1} \in [\K_i,\K^{\prime}_i]_{\mathrm{H}_\bullet (S)}$
and
$u_i$, $u^{\prime}_i$
are the canonical morphisms.
Similarly, there exists a unique morphism
$\theta^{\prime}\colon \mathrm{BGL}^{\prime} \to \mathrm{BGL}$
in
$\mathrm{SH}(S)$
such that for every integer
$i \geq 0$
the diagram
\begin{equation*}\xymatrix{
  \Sigma^{\infty}_{\Pro^1}\K_i^\prime(-i)   \ar[r]^-{u_i^\prime}
    \ar[d]_-{\Sigma^{\infty}_{\Pro^1} \phi_i^\prime(-i)} &\mathrm{BGL}^\prime
    \ar[d]^-{\theta^\prime} \\
  \Sigma^{\infty}_{\Pro^1}\K_i(-i)\ar[r]^-{u_i}&\mathrm{BGL}}
\end{equation*}
commutes in
$\mathrm{SH}(S)$, where
$\theta_i= i \circ (i^{\prime})^{-1} \in [\K^{\prime}_i,\K_i]_{\mathrm{H}_\bullet (S)}$.
\end{proposition}

\begin{proof}
Consider the exact sequence
\begin{equation*}
0 \to {\varprojlim}^{1}\mathrm{BGL}^{2i-1,i}(\K^{\prime}_i)\to
\mathrm{BGL}^{0,0}(\mathrm{BGL}^{\prime})
\to {\varprojlim}\mathrm{BGL}^{2i,i}(\K^{\prime}_i) \to 0
\end{equation*}
from Lemma~\ref{lem:filtered-colimit}. The family of elements
$(u_i \circ \Sigma^{\infty}_{\Pro^1} \theta^{\prime}_i (-i))$
is an element of the group
${\varprojlim}\mathrm{BGL}^{2i,i}(\K^{\prime}_i)$.
Thus there exists the required morphism
$\theta^{\prime}$. To prove its uniqueness, observe that the
${\varprojlim}^{1}$-group
vanishes by Proposition~\ref{LimOneBGL}.
Whence
$\mathrm{BGL}^{0,0}(\mathrm{BGL}^{\prime})=
   {\varprojlim}\mathrm{BGL}^{2i,i}(\K^{\prime}_i)$
and
$\theta^{\prime}$
is indeed unique.
By symmetry there also exists a unique
morphism $\theta$
with the required property.
\end{proof}

\begin{proposition}
\label{BGL=BGL}
Let $S=\Spec(\mathbb{Z})$. The morphism
$\theta\colon \mathrm{BGL} \to \mathrm{BGL}^{\prime}$
is the inverse of
$\theta^{\prime}\colon \mathrm{BGL}^{\prime} \to \mathrm{BGL}$
in
$\mathrm{SH}(S)$, and in particular an isomorphism.
\end{proposition}

\begin{proof}
The composite morphism
$\theta^{\prime} \circ \theta\colon \mathrm{BGL} \to \mathrm{BGL}$
has the property that
for every integer
$i \geq 0$
the diagram
\begin{equation*}\xymatrix{
  \Sigma^{\infty}_{\Pro^1}\K_i(-i)   \ar[r]^-{u_i}
    \ar[d]_-{\id} &    \mathrm{BGL} \ar[d]^-{\theta^\prime \circ \theta} \\
  \Sigma^{\infty}_{\Pro^1}\K_i(-i)\ar[r]^-{u_i}&\mathrm{BGL}}
\end{equation*}
commutes. However, the identity morphism
$\id\colon \mathrm{BGL} \to \mathrm{BGL}$
has the same property.
Thus
$\theta^{\prime} \circ \theta = \id$.
by the uniqueness in Proposition~\ref{BGLandBGL}, and similarly
$\theta \circ \theta^{\prime} = \id$.
\end{proof}

\begin{remark}
\label{BGL=BGLasMonoids}
The isomorphisms
$\theta$
and
$\theta^{\prime}$
are monoid isomorphisms provided
that
$\mathrm{BGL}$
and
$\mathrm{BGL}^{\prime}$
are equipped with the monoidal structures
given by
Theorem~\ref{MonoidalStrOnBGL}.
This follows from the fact that both $i$ and $i^\prime$ are
isomorphisms of monoids in $H_\bullet(S)$.
\end{remark}

\begin{remark}
\label{Useful}
There exists a unique morphism
$e\colon(\mathbb{Z} \times \mathrm{Gr})\wedge \Pro^1 \to \mathbb{Z} \times \mathrm{Gr}$
in
$\mathrm{H}_\bullet (S)$
such that the diagram
\begin{equation*}
  \xymatrix{
    (\mathbb{Z} \times \mathrm{Gr})\wedge \Pro^1
    \ar[r]^-{e}\ar[d]_-{\psi \wedge \id} &
    \mathbb{Z} \times \mathrm{Gr} \ar[d]^-\psi     \\
    \mathbb{K}^{W} \wedge \Pro^1   \ar[r]^-{e^{W}} &  \mathbb{K}^{W}}
\end{equation*}
commutes in
$\mathrm{H}_\bullet (S)$,
where
$e^{W}=\mu^{W} \circ (\id \wedge b^{W})$
 and $\psi$ is described right above Definition
\ref{BGL}. The diagram
\begin{equation*}
  \xymatrix{
    (\mathbb{Z} \times \mathrm{Gr})\wedge \Pro^1
    \ar[r]^-{e}\ar[d]_-{i\wedge \id} &
    \mathbb{Z} \times \mathrm{Gr} \ar[d]^-i     \\
    \K \wedge \Pro^1   \ar[r]^-{\epsilon} &  \K}
\end{equation*}
then commutes as well.
That is, $\epsilon$ descends to $e$ in $\mathrm{H}_\bullet (S)$.

We will need an observation concerning the morphism $e$.
Let $\tau_n$ be the tautological bundle over the Grassmann variety
$\textrm{Gr}(n,2n)$. By Lemma
\ref{KTTofGr}
there exists a unique element
$\xi_{\infty} \in \textrm{K}_0(\mathbb{Z} \times \mathrm{Gr})$
such that for any positive integer $n$ and any $0 \leq m \leq n$
one has
$\xi|_{\lbrace m\rbrace  \times \mathrm{Gr}(n,2n)} = [\tau_n]-n+m $.
By Lemmas
\ref{KofGrsmashGr}
and
\ref{AtensorB}
the element
$\xi_{\infty} \otimes ([\mathcal O(-1)] - [\mathcal O])$
belongs to the subgroup
$\textrm{K}_0((\mathbb{Z} \times \mathrm{Gr}) \wedge \Pro^1)$
of the group
$\textrm{K}_0(\mathbb{Z} \times \mathrm{Gr} \times \Pro^1)$.
The isomorphism $\psi\colon \mathbb{Z}\times\mathrm{Gr}\to
\mathbb{K}^{W}$ in $\mathrm{H}_\bullet(S)$ represents the element
$\xi_{\infty}$ in
$\textrm{K}^{}_0((\mathbb{Z} \times \mathrm{Gr}))$.
Now the definitions of
$e^{W}$ and $e$
show that
\begin{equation*}
  e^{\ast}(\xi_{\infty})= \xi_{\infty}  \otimes ([\mathcal{O}(-1)]-[\mathcal{O}])
\end{equation*}
in
$\textrm{K}_0((\mathbb{Z} \times \mathrm{Gr}) \wedge \Pro^1)$.
\end{remark}

\begin{remark}\label{AllBGLareIsomorphic}
  Let $S$ be a regular scheme.
  Given two triples $(\K_1,i_1,\epsilon_1)$ and $(\K_2,i_2,\epsilon_2)$
  fulfilling the conditions of Definition~\ref{BGL},
  there exists a zig-zag of weak equivalences of triples
  connecting these two. In particular, there exists a
  zig-zag of levelwise weak equivalences of $\Pro^1$-spectra
  over $S$ connecting $\mathrm{BGL}_1$ and $\mathrm{BGL}_2$.
  It follows that
  the $\Pro^1$-spectra
  $\textrm{BGL}_1$ and $\textrm{BGL}_2$ associated to these triples
  are "naturally" isomorphic in $\textrm{SH}(S)$.
  This shows that the strength of
  Proposition
  \ref{BGLandBGL}
  is in its uniqueness assertion.
\end{remark}

\subsection{Preliminary computations I}
\label{PreliminaryCompI}

In this section we prepare for the next section, in which
we show that certain
${\varprojlim}^{1}$-groups vanish.
Let
$\mathrm{BGL}$
be the $\Pro^1$-spectrum defined in~\ref{BGL}.
We will identify in this section the functors
$\mathrm{BGL}^{0,0}$
and
$\mathrm{BGL}^{2i,i}$
on the category
$\mathrm{H}_{\bullet}(S)$
via the iterated $(2,1)$-periodicity isomorphism as follows:
\begin{equation}
\label{Periodicity00}
\mathrm{BGL}^{0,0}(A)\iso\Hom_{\mathrm{H}_\bullet(S)}(A,\K_0)=
       \Hom_{\mathrm{H}_\bullet(S)}(A,\K_i)\iso \mathrm{BGL}^{2i,i}(A).
\end{equation}
Similarly,
\begin{equation}
\label{Periodicity}
\begin{split}\mathrm{BGL}^{-1,0}(A)  & \iso  \Hom_{\mathrm{H}_\bullet(S)}(S^{1,0} \wedge A,\K_0)
          \\
  &= \Hom_{\mathrm{H}_\bullet(S)}(S^{1,0} \wedge A,\K_i)\iso
\mathrm{BGL}^{2i-1,i}(A)
\end{split}
\end{equation}
These identifications respect the naive product structure
(\ref{NaiveProducts})
on the functor
$\mathrm{BGL}^{\ast,\ast}$.
In particular,
the following diagram commutes for every pointed motivic space $A$ over $S$.
\begin{equation}\label{OneDiagram}
  \xymatrix{
    \mathrm{BGL}^{-1,0}(S) \otimes \mathrm{BGL}^{0,0}(A)  \ar[r] \ar[d]_-\iso &
    \mathrm{BGL}^{-1,0}(A) \ar[d]^-\iso \\
    \mathrm{BGL}^{-1,0}(S) \otimes \mathrm{BGL}^{2i,i}(A)\ar[r]&\mathrm{BGL}^{2i-1,i}(A)}
\end{equation}

\begin{remark}
\label{21Periodicity}
The identification
(\ref{Periodicity})
of
$\mathrm{BGL}^{-1,0}(X)$
with
$\mathrm{BGL}^{2i-1,i}(X)$
coincides with the periodicity isomorphism
$\mathrm{BGL}^{-1,0}(X) \xra{\cup \beta^i} \mathrm{BGL}^{2i-1,i}(X)$.
\end{remark}

\begin{lemma}
\label{BGLofGr}
Let $S=\Spec(\mathbb{Z})$. For every integer $i$ the map
$$
\mathrm{BGL}^{-1,0}(S) \otimes_{} \mathrm{BGL}^{2i,i}(\K) \to
\mathrm{BGL}^{2i-1,i}(\K)
$$
induced by the naive product structure is an isomorphism.
The same holds if $\K \wedge \Pro^1$ replaces
$\K$.
\end{lemma}

\begin{proof}
The commutativity of the diagram
(\ref{OneDiagram})
shows that it suffices to consider the case $i=0$.
Furthermore we may replace the pointed motivic space $\K$ with
$\mathbb{Z} \times \mathrm{Gr}$
since the map
$i\colon \mathbb{Z} \times \mathrm{Gr} \to \K=\K$
is a weak equivalence. The functor isomorphism
$K^{}_\ast \to \mathrm{BGL}^{-*,0}$
is a ring cohomology isomorphism by Corollary~\ref{RingIsomorphism}.
Thus it remains to check that the map
\begin{equation*}
K^{}_1(S) \otimes_{} K^{}_0(\mathbb{Z} \times \mathrm{Gr}) \to
K^{}_1(\mathbb{Z} \times \mathrm{Gr})
\end{equation*}
is an isomorphism. For a set $M$ and a smooth $S$-scheme $X$ we will write
$M \times X$ for the disjoint union $\bigsqcup_M X$ of $M$
copies of $X$ in the category of motivic spaces over $S$.
Let $[-n,n]$
be the set of integers with absolute value $\leq n$.
By Lemma
\ref{KTTofGr}
and Lemma
\ref{KofCellularVar}
it suffices to check that the natural map
\begin{equation*} A \otimes {\varprojlim}K^{}_0 \bigl([-n,n] \times \mathrm{Gr}(n,2n)\bigr) \to
{\varprojlim} A \otimes K^{}_0 \bigl([-n,n] \times \mathrm{Gr}(n,2n)\bigr)
\end{equation*}
is an isomorphism, where $A = K^{}_1(S)$.
This is the case since
$K^{}_1(S)$
is a finitely generated abelian group (it is just
$\mathbb Z/2\mathbb Z$).
The assertion concerning
$\K \wedge \Pro^1$
is proved similarly using Lemmas
\ref{KofGrsmashGr}
and
\ref{KofSmash} instead.
\end{proof}

To state the next lemma, consider the scheme morphism
$f\colon \Spec(\mathbb{C}) \to S=\Spec(\mathbb{Z})$,
the pull-back functor
$f^\ast\colon \SH(\mathbb{Z}) \to \mathrm{SH}(\mathbb{C})$
described in Proposition~\ref{prop:base-change-spectra}, and
the topological realization functor
$\real_\CC\colon  \mathrm{SH}(\mathbb{C}) \to \SH_{\CC\Pro^1}$
described in Section~\ref{sec:stable-topol-real}. Set
$r= \real_\CC \circ f^\ast\colon  \SH(S) \to \SH_{\CC\Pro^1}$.
The functor $r$ will be called for short
the realization functor below in this Section.

\begin{lemma}
\label{SigzagWeakEquiv}
Let
$\mathbb{B}\mathrm{U}$
be the periodic complex K-theory $\CC\Pro^1$-spectrum with terms
$\mathbb{Z} \times \mathrm{BU}$. There is a zigzag
$\mathbb{B}\mathrm{U} \xla{\sim} E \xra{\sim} r\mathrm{BGL}$
of levelwise weak equivalences of $\CC\Pro^1$-spectra.
\end{lemma}

\begin{proof}
  This follows from
  Remark~\ref{Useful},
  ~\ref{example:stable-real-bgl}
  and the fact that Grassmann varieties pull back.
\end{proof}

\begin{lemma}
\label{CohOfCellularSpace}
Let $X \in \Sm/S$, where $S=\Spec(\ZZ)$, and let
$X_0 \subset X_1 \subset \dots \subset X_n = X$
be a filtration by closed subsets such that
for every integer
$i \geq 0$
the $S$-scheme
$X_i - X_{i-1}$
is isomorphic to a disjoint union of several copies of the
affine space
$\Aff^i_S$.
The map
$\mathrm{BGL}^{0,0}(X) \to
( r\mathrm{BGL})^{0}(rX)$
is an isomorphism.
\end{lemma}

\begin{proof}
  Consider the class $\mathcal{R}$ of $\Pro^1$-spectra $E$ such that
  the homomorphism $\BGL^{0,0}(E) \to r\BGL^0(rE)$ is an isomorphism. It
  contains $S^{0,0}$ because in this case we obtain the isomorphism
  $\BGL^{0,0}(S^{0,0}) \iso K^0(\ZZ)\iso \ZZ \iso K^0_\mathrm{top}(S^0)$
  which identifies the class of an algebraic resp.~complex topological
  vector bundle over $\Spec(\ZZ)$ resp.~$\bullet$ with its rank.
  The $(2,1)$-periodicity isomorphism for $\BGL$ described in
  Remark
  \ref{21Periodicity}
  and the Bott periodicity isomorphism for $r\BGL$ are
  compatible by~\ref{example:stable-real-bgl}. This implies that
  $S^{2m,m}\in \mathcal{R}$ for all $m\in \ZZ$. Finally, if
  $E\to F\to G \to S^{1,0}\wedge E$ is a distinguished triangle
  in $\SH(S)$ such that $E$ and $G$ are in $\mathcal{R}$, then so is
  $F$. 

  For $i\geq 0$ write $U^i \colon\!\!= X \smallsetminus X_i$, so that $U^i$ is an open
  subset of $U^{i-1}$. In particular we have $U^n = \emptyset$ and
  $U^{-1} = X$. The closed subscheme
  $X_i\smallsetminus X_{i-1} = X_i \cap U^{i-1} \hra U^{i-1}$
  is isomorphic to a disjoint union $m_i$ copies of affine spaces $\Aff^i$, and
  is in particular smooth over $S$. Furthermore the normal bundle is
  trivial. The homotopy purity theorem \cite[Thm.~3.2.29]{MV}
  supplies a distinguished triangle
  \begin{equation*}
    \Sigma^\infty_{\Pro^1} U^i_+ \to \Sigma^\infty_{\Pro^1}  U^{i-1}_+ \to
    \Sigma^\infty_{\Pro^1} U^{i-1}/U^i \iso \vee_{j=1}^{m_i} S^{2(n-i),(n-i)}
  \end{equation*}
  of $\Pro^1$-spectra. Since $\mathcal{R}$ contains $\Sigma^\infty_{\Pro^1} U^n =
  \bullet$
  we obtain inductively that $\mathcal{R}$ contains $\Sigma^\infty_{\Pro^1} U^{-1}
  =\Sigma^\infty_{\Pro^1} X_+$.
\end{proof}

\begin{lemma}
\label{GeomRealIsom}
Let
$S=\Spec(\mathbb{Z})$
and let
$r\colon \SH(\mathbb{Z}) \to \SH_{\CC\Pro^1}$
be the topological realization functor.
Then for every integer $i$ the realization
homomorphism
$\mathrm{BGL}^{2i,i}(\K) \to (r\mathrm{BGL})^{2i}(r\K)$
is an isomorphism.
\end{lemma}

\begin{proof}
Clearly it suffices to prove the case $i=0$.
We may replace the pointed motivic space $\K_i$
with
$\mathbb{Z} \times \mathrm{Gr}$
as in the proof of Lemma~\ref{BGLofGr}.
It remains to check that the topological realization homomorphism
$\mathrm{BGL}^{0,0}(\mathrm{Gr}) \to
( r\mathrm{BGL})^{0}(r\mathrm{Gr})$
is an isomorphism.

Since
$\mathrm{Gr}(n,2n)$
has a filtration
satisfying the condition of Lemma
\ref{CohOfCellularSpace},
we see that the map
$\mathrm{BGL}^{0,0}(\mathrm{Gr}(n,2n)) \to
( r\mathrm{BGL})^{0}(\mathrm{Gr}(n,2n))$
is an isomorphism for every $n$. To conclude the statement
for
$\mathrm{Gr}= \cup \mathrm{Gr(n,2n)}$,
use the short exact sequence
from Lemma~\ref{lem:filtered-colimit}.
In the resulting diagram
\begin{equation*}
  \xymatrix@C=0.5cm{
    \varprojlim^1 \BGL^{-1,0}\bigl(\Gr(n,2n)\bigr) \ar[r] \ar[d] &
    \BGL^{0,0}(\Gr) \ar[r] \ar[d] & \varprojlim \BGL^{0,0}\bigl(\Gr(n,2n)\bigr)
    \ar[d] \\
    \varprojlim^1 r\BGL^{-1,0}\bigl(r\Gr(n,2n)\bigr) \ar[r]  &
    r\BGL^{0,0}(r\Gr) \ar[r]  & \varprojlim r\BGL^{0,0}\bigl(r\Gr(n,2n)\bigr)}
\end{equation*}
the map on the right hand side is then an isomorphism.
Furthermore one concludes from \cite[Thm.~16.32]{Sw} that
\begin{equation*}
  {\varprojlim}^{1} r\BGL^{-1,0}\bigl(r\Gr(n,2n)\bigr) = {\varprojlim}^{1}
  K^1_\mathrm{top} \bigl(r\Gr(n,2n)\bigr)=0.
\end{equation*}
On the other hand
$\varprojlim^1 \BGL^{-1,0}\bigl(\Gr(n,2n)\bigr) =
\varprojlim^1 K_1\bigl(\Gr(n,2n)\bigr)=0$
by Lemma~\ref{LimOneVanish}.
The result follows.
\end{proof}

\begin{lemma}
\label{BUnot}
Let
$\mathbb{B}^0\mathrm{U}$
be the sub-spectrum of
$\mathbb{B}\mathrm{U}$
with the $n$-th term equal to the connected component
$\mathrm{BU}$
of the topological space
$\mathbb{Z} \times \mathrm{BU}$ containing
the basepoint $\bullet$.
The inclusion
$\mathbb{B}^0\mathrm{U}\xra{} \mathbb{B}\mathrm{U}$
is a weak equivalence of $\CC\Pro^1$-spectra.
\end{lemma}

\begin{proof}
  One has to check that the inclusion induces an isomorphism
  on stable homotopy groups. This follows because the structure
  map $(\ZZ\times BU)\wedge \CC\Pro^1 \to \ZZ\times BU$ factors
  over $\lbrace 0 \rbrace \times BU$.
\end{proof}

\begin{lemma}\label{FiniteAprox}
There exists a 
sub-spectrum $\mathbb{B}^f\mathrm{U}$ of the $\CC\Pro^1$-spectrum
$\mathbb{B}\mathrm{U}$
with the $n$-th term
$\mathrm{Gr}\bigl(b(n),2b(n)\bigr)$
such that the inclusion
$\mathbb{B}^f\mathrm{U} \xra{} \mathbb{B}\mathrm{U}$
is a stable equivalence.% of $\CC\Pro^1$-spectra.
\end{lemma}

\begin{proof}
  The sequence $b(n)$ will be constructed such that
  $b(n)\geq 2n+1$.
  Set $b(0)=1$. We may assume that the structure map
  $e_0\colon \mathrm{B}\mathrm{U} \wedge \CC\Pro^1 \to \mathrm{B}\mathrm{U}$
  is cellular. Since
  $r\mathrm{Gr}(b(0),2b(0)) \wedge \CC\Pro^1$ is a finite cell complex, it
  lands in a Grassmannian
  $r\mathrm{Gr}(b(1),2b(1))$
  for some integer $b(1)\geq 2\cdot 1+1$.
  Continuing this process produces the required sequence
  of $b(n)$'s. The inclusions induce an isomorphism
  $\colim_{n\geq 0} \mathrm{Gr}\bigl(b(n),2b(n)\bigr)\iso \mathrm{Gr}$.

  To observe that the inclusion
  $j\colon \mathbb{B}^f\mathrm{U} \xra{} \mathbb{B}\mathrm{U}$
  is then a stable equivalence, recall that the number of
  $2k$-cells in $\mathrm{Gr}(n,m)$ is given by the number of partitions
  of $k$ into at most $n$ subsets each of which has cardinality
  $\geq m-n$ \cite{Milnor-Stasheff}. In particular,
  the $2k$-skeleton of $\mathrm{BU}$ coincides with the $2k$-skeleton
  of $r\mathrm{Gr}(k,2k)$. To prove the surjectivity of $\pi_i(j)$
  choose an element $\alpha\in \pi_i \mathbb{B}\mathrm{U}$. It
  is represented by a cellular map $a\colon S^{i+2m}\to \mathrm{BU}$ for
  some $m$ with $i+2m\geq 0$. We may choose $m$ such that $m\geq i$.
  Thus $a$ lands in $r\mathrm{Gr}(b(m),2b(m))$
  and gives rise to an element in $\pi_i \mathbb{B}^f\mathrm{U}$
  mapping to $\alpha$. To prove the injectivity of $\pi_i(j)$, choose
  an element $\alpha\in \pi_i \mathbb{B}^f\mathrm{U}$ such that
  $\pi_i(j)(\alpha)= 0$. We may represent $\alpha$ by some map
  $a\colon \pi_{i+2m}(j) \mathrm{Gr}(b(m),2b(m))$ for some $m$ with
  $i+2m> 0$ and $m\geq i$. The composition
  \begin{equation*} S^{i+2m} \ra{a} \mathrm{Gr}(b(m),2b(m)) \hra \mathrm{BU} \end{equation*}
  is nullhomotopic since $\pi_{i+2m} \mathrm{BU} \iso \pi_i \mathbb{B}\mathrm{U}$
  via the homomorphism induced by the structure map.
  The nullhomotopy may be chosen to be cellular and thus lands in
  $\mathrm{Gr}(b(m),2b(m))$. This completes the proof.
\end{proof}

\subsection{Vanishing of certain groups I}
Consider the stable equivalence
$\mathrm{hocolim}_{i\geq 0}\, \Sigma^{\infty}_{\Pro^1}\K_i(-i) \cong \mathrm{BGL}$
(see (\ref{LevelAproxOne}))
and the respecting short exact sequence
$$
0 \to {\varprojlim}^{1}\mathrm{BGL}^{2i-1,i}(\K_i) \to
\mathrm{BGL}^{0,0}(\mathrm{BGL}) \to
\varprojlim \mathrm{BGL}^{2i,i}(\K_i) \to 0
$$
We prove in this section the following result
\begin{proposition}
\label{LimOneBGL}
Let $S=\Spec (\mathbb{Z})$, then
${\varprojlim}^{1}\mathrm{BGL}^{2i-1,i}(\K_i)=0$.
\end{proposition}

\begin{proof}
The connecting homomorphism in the tower
of groups for the
${\varprojlim}^{1}$-term is the composite map
\begin{equation*}
\mathrm{BGL}^{2i-1,i}(\K_i) \xla{\Sigma^{-1}_{\Pro^1}}
\mathrm{BGL}^{2i+1,i+1}(\K_i \wedge \Pro^1)
\xla{e_i^\ast} \mathrm{BGL}^{2i+1,i+1}(\K_{i+1})
\end{equation*}
where
$\Sigma^{-1}_{\Pro^1}$
is the inverse to the $\Pro^1$-suspension isomorphism and
$e_i^\ast$
is the pull-back induced by the structure map $e_i$.
Set $A=\mathrm{BGL}^{-1,0}(S)$ and consider the diagram
\begin{equation*}
  \xymatrix{
    \mathrm{BGL}^{2i-1,i}\!(\K)   &
    \mathrm{BGL}^{2i+1,i+1}\!(\K \wedge \!\Pro^1\!)
    \ar[l]_-{\Sigma^{-1}_{\Pro^1}} &
    \mathrm{BGL}^{2i+1,i+1}\!(\K) \ar[l]_-{e_i^\ast} \\
    \!A \otimes \mathrm{BGL}^{2i,i}\!(\K)  \ar[u] &
    \!A \otimes \mathrm{BGL}^{2(i+1),i+1}\!(\K \wedge\! \Pro^1\!)
    \ar[l]_-{\id \otimes \Sigma^{-1}_{\Pro^1}} \ar[u] &
    \!A \otimes \mathrm{BGL}^{2(i+1),i+1}\!(\K)
    \ar[u]\ar[l]_-{\id \otimes e_i^\ast} }
\end{equation*}
where the vertical arrows are induced by the naive product structure
on the functor
$\mathrm{BGL}^{*,*}$.
Clearly it commutes. Since
$S$ is regular, the vertical arrows are isomorphisms by Lemma~\ref{BGLofGr}.
It follows that
${\varprojlim}^{1}\mathrm{BGL}^{2i-1,i}(\K_i)=
{\varprojlim}^{1}(A \otimes \mathrm{BGL}^{2i,i}(\K_i))$
where in the last tower of groups the connecting maps are
$\id \otimes (\Sigma^{-1}_{\Pro^1} \circ e_i^\ast)$.
It remains to prove the following assertion.

\begin{claim}
  The equality ${\varprojlim}^{1}(A \otimes \mathrm{BGL}^{2i,i}(\K_i))=0$ holds.
\end{claim}

Since $S=\Spec(\ZZ)$ one obtains
$A= \mathrm{BGL}^{-1,0}(S)=K_1(\ZZ)=\mathbb{Z} /2 \mathbb{Z}$.
It follows that
$A \otimes \mathrm{BGL}^{2i,i}(\K_i)=
   \mathrm{BGL}^{2i,i}(\K_i)/m\mathrm{BGL}^{2i,i}(\K_i)$
with $m=2$ and the connecting maps in the tower are just the mod-$m$
reduction of the maps
$\Sigma^{-1}_{\Pro^1} \circ e_i^\ast$.
Now a chain of isomorphisms completes the proof of the Claim.
\begin{align*}
  {\varprojlim}^{1}\mathrm{BGL}^{2i,i}(\K_i)/m & \iso
  {\varprojlim}^{1}(r\mathrm{BGL})^{2i}(r \K_i)/m \iso
  {\varprojlim}^{1}\mathbb{B}\mathrm{U}^{2i}(r \K_i)/m \\& \iso
  {\varprojlim}^{1}\mathbb{B}\mathrm{U}^{2i}(\mathbb{Z} \times \mathrm{BU})/m \iso
  {\varprojlim}^{1}\mathbb{B}\mathrm{U}^{2i}(\mathbb{Z} \times \mathrm{BU};\mathbb{Z}/m)
  \\ & \iso K_{\mathrm{top}}^1(\mathbb{B}\mathrm{U}; \mathbb{Z}/m)  \iso
  K_{\mathrm{top}}^1(\mathbb{B}^0\mathrm{U}; \mathbb{Z}/m) \\ & \iso
  {\varprojlim}^{1}K_\mathrm{top}^{2i}\bigl(\mathrm{Gr}(b(i),2b(i));\mathbb{Z}/m\bigr)=0
\end{align*}
The first isomorphism follows from Lemma~\ref{GeomRealIsom}.
The second isomorphism is induced by the levelwise weak equivalence
$\mathbb{B}\mathrm{U} \simeq r\mathrm{BGL}$
mentioned in Lemma~\ref{SigzagWeakEquiv}.
The third isomorphism is induced by the image of the weak equivalence
$\K_i \simeq \mathbb{Z} \times \mathrm{Gr}$ under topological realization.
The forth and fifth isomorphism hold since
$\mathbb{B}\mathrm{U}^{2i+1}(\mathbb{Z} \times \mathrm{BU})=0$.
The sixth isomorphism is induced by the stable equivalence
$\mathbb{B}^0\mathrm{U} \simeq \mathbb{B}\mathrm{U}$ from Lemma~\ref{BUnot},
the seventh one is induced by the stable equivalence
$\mathbb{B}^f\mathrm{U} \simeq \mathbb{B}^0\mathrm{U}$ from Lemma~\ref{FiniteAprox}. The last one holds since all groups in the tower
are finite.
\end{proof}

\section{Smash-product, pull-backs, topological realization}
\label{sec:smash-product-pull}
In this section we construct a smash-product
$\wedge$ of $\Pro^1$-spectra,
check its basic properties, consider its behavior with respect to pull-back
and realization functors. We follow here an idea of Voevodsky
\cite[Comments to Thm.~5.6]{V1}
and use results of Jardine
\cite{J}.
In several cases we will not distinguish notationally
between a Quillen functor and its total derived functor,
$\Sigma_{\Pro^1}^\infty$ being the most prominent example.

\subsection{The smash product}

\begin{definition}
\label{Smash}
Let
$V:=\tlder v \colon \SH(S) \to \SH^\Sigma(S) $
and $U:=\trder u \colon \SH^\Sigma(S) \to \SH(S)$
be the equivalence
described in Remark~\ref{rem:monoidal-sh}.
For a pair of $\Pro^1$-spectra $E$ and $F$
set
\begin{equation*}
  E \wedge F:= U(VE \wedge VF)
\end{equation*}
as in Remark~\ref{rem:monoidal-sh}.
\end{definition}

\begin{proposition}
\label{MonoidalStr}
Let $S$ be a Noetherian finite-dimensional base scheme.
The smash-product of $\Pro^1$-spectra over $S$ induces
a closed
symmetric monoidal structure $(\wedge,\one)$ on the motivic stable homotopy category
$\mathrm{SH}(S)$
having the properties required by
Theorem 5.6 of Voevodsky's congress talk
\cite{V1}:
\begin{enumerate}
\item
There is a canonical isomorphism
$E\wedge \Sigma^{\infty}_{\Pro^1}A  \cong (A \wedge E_i ,\id \wedge e_i )$
for every pointed motivic space $A$ and every $\Pro^1$-spectrum.
\item
There is a canonical isomorphism
$(\oplus E_{\alpha}) \wedge F \cong \oplus (E_{\alpha} \wedge F)$
for $\Pro^1$-spectra $E_i,F$.
\item
Smashing with a $\Pro^1$-spectrum preserves distinguished triangles.
To be more precise, if
$E \xra{f} F \to \mathrm{cone}(f) \xra{\epsilon} E[1]$
is a distinguished triangle and $G$ is a $\Pro^1$-spectrum,
the sequence
$E\wedge G \xra{f} F\wedge G  \to \mathrm{cone}(f)\wedge G
\xra{\epsilon} E\wedge G [1]$
is a distinguished triangle, where the last morphism is the composition of
$\epsilon \wedge \id_G$
with
the canonical isomorphism
$E[1]\wedge G \to (E\wedge F)[1]$.
\end{enumerate}
\end{proposition}

\begin{proof}
  Follows from Remark~\ref{rem:monoidal-sh} and
  Theorem~\ref{thm:quillen-eq}.
\end{proof}

In the following we use that the homotopy
colimit of a sequence
$ E = E_0\to E_1 \to \dotsm $
of morphisms in the homotopy category $\SH(S)$ may be computed in three steps:
\begin{enumerate}
  \item Lift $E$ to a sequence $E^\prime = E_0^\prime \to E_1^\prime \to \dotsm$ of
    cofibrations (in fact, arbitrary maps suffice) of $\Pro^1$-spectra.
  \item Take the colimit $\colim_{i\geq 0} E^\prime_i$
    of $E^\prime$ in the category of $\Pro^1$-spectra.
  \item Consider $\colim_{i\geq 0} E^\prime_i$ as an object in $\SH(S)$.
\end{enumerate}

\begin{lemma}
\label{SmashAndHocolim}
Let $E=\mathrm{hocolim}_{i\geq 0} E_i$ be a sequential
homotopy colimit of $\Pro^1$-spectra.
For every $\Pro^1$-spectrum $F$ there is an exact sequence of abelian groups
\begin{equation}
0 \to {\varprojlim}^{1}F^{p-1,q}(E_i) \to F^{p,q}(E) \to
\varprojlim F^{p,q}(E_i) \to 0.
\end{equation}
\end{lemma}

\begin{proof}
  This is Lemma~\ref{lem:filtered-colimit}.
\end{proof}

By Lemma~\ref{lem:spectrum-colim}, any $\Pro^1$-spectrum $E$
can be expressed as the homotopy colimit
\begin{equation}
\label{LevelAproxOne}
\mathrm{hocolim}\ \Sigma^{\infty}_{\Pro^1}E_i(-i) \cong E.
\end{equation}

\begin{corollary}
\label{FiniteLevelOne}
For two $\Pro^1$-spectra $E$ and $F$ there is a canonical
short exact sequence
\begin{equation}
\label{ShortExactSequence}
0 \to {\varprojlim}^{1}F^{p+2i-1,q+i}(E_i) \to F^{p,q}(E) \to
\varprojlim F^{p+2i,q+i}(E_i) \to 0.
\end{equation}
\end{corollary}

\begin{corollary}
\label{FiniteLevelTwo}
For a pair of spectra $E$ and $F$ and each spectrum $G$
one has a canonical exact sequence of the form
\begin{equation}
0 \to {\varprojlim}^{1}G^{p+4i-1,q+2i}(E_i \wedge F_i) \to G^{p,q}(E \wedge F)
\to  \varprojlim G^{p+4i,q+2i}(E_i \wedge F_i) \to 0.
\end{equation}
\end{corollary}

\begin{proof}
  For a pair of spectra $E$ and $F$ one has a canonical isomorphism of the form
  \begin{equation}
    \label{LevelAproxTwo}
    \mathrm{hocolim}\ (\Sigma^{\infty}_{\Pro^1}(E_i \wedge F_i)(-2i)) \cong E \wedge F
  \end{equation}
  as deduced in Lemma~\ref{lem:hocolim-smash}. The result
  follows from Corollary~\ref{SmashAndHocolim}.
\end{proof}

\subsection{A monoidal structure on $\mathrm{BGL}$}
\label{MonoidOnBGL}

For a $\Pro^1$-spectrum $E$ and an integer $i\geq 0$
$u_i\colon \Sigma^{\infty}_{\Pro^1}E_i(-i) \to E$
denotes the canonical map from
\ref{ex:suspension-spectrum}.
Let
$\mathrm{BGL}$
be the $\Pro^1$-spectrum defined in~\ref{BGL}.
Recall that this involves the choice of
a weak equivalence $i \colon \mathbb{Z}\times \mathrm{Gr}\to
\mathcal{K}$ and a structure map $\epsilon\colon
\mathcal{K}\wedge \Pro^1\to \mathcal{K}$. Following Lemma
\ref{CofibrantK} we may and will assume
additionally that the pointed motivic space $\K$ is cofibrant.
The aim of this section is to prove the following statement.
\begin{theorem}
\label{MonoidalStrOnBGL}
  Assume that the pointed motivic space $\mathcal{K}$ is cofibrant.
  Consider
  the family of pairings
  $\K_i \wedge \K_j \xra{\mu_{ij}}  \K_{i+j}$
  in $H_\bullet(S)$
  with
  $\mu_{ij}=\bar \mu$
  from Remark
  \ref{MuBar}.
  For
  $S=\Spec(\mathbb{Z})$
  there is a unique
  morphism
  $\mu_{\mathrm{BGL}}\colon \mathrm{BGL} \wedge \mathrm{BGL} \xra{} \mathrm{BGL}$
  in the motivic stable homotopy category
  $\SH(S)$
  such that for every $i$ the diagram
  \begin{equation*}
    \xymatrix@C=4em{
      \Sigma^{\infty}_{\Pro^1}\K_i(-i) \wedge \Sigma^{\infty}_{\Pro^1}\K_i(-i)
      \ar[r]^-{\Sigma^{\infty}(\mu_{ii})} \ar[d]_-{u_i\wedge u_i}&
      \Sigma^{\infty}_{\Pro^1}\K_{2i}(-2i)    \ar[d]^-{u_{2i}}  \\
      \mathrm{BGL} \wedge \mathrm{BGL}\ar[r]^-{\mu_{\mathrm{BGL}}}& \mathrm{BGL}}
  \end{equation*}
  commutes.
  Let
  $e_{\mathrm{BGL}}\colon \one \to \mathrm{BGL}$
  in
  $\SH(S)$
  be adjoint to the unit
  $e_{\K}\colon S^{0,0} \to \K$.
  Then
  \begin{equation*}
    (\mathrm{BGL},\mu_{\mathrm{BGL}},e_{\mathrm{BGL}})
  \end{equation*}
  is a commutative monoid in
  $\SH(S)$.
\end{theorem}

\begin{proof}
The morphism
$\mu_{\mathrm{BGL}}$
we are looking for is an element of the group
$\mathrm{BGL}^{0,0}(\mathrm{BGL} \wedge \mathrm{BGL})$.
This group fits in the exact sequence
\begin{equation*}
0 \to {\varprojlim}^{1}\mathrm{BGL}^{4i-1,2i}(\K_i^{\wedge 2}) \to
\mathrm{BGL}^{0,0}(\mathrm{BGL} \wedge \mathrm{BGL}) \to
\varprojlim \mathrm{BGL}^{4i,2i}(\K_i^{\wedge 2}) \to 0
\end{equation*}
by Corollary~\ref{FiniteLevelTwo}.
The family of elements
$\lbrace u_{2i} \circ \Sigma^{\infty}(\mu_{ii}) \rbrace $
is an element of the
$\varprojlim $ group.
The
${\varprojlim}^{1}$
group vanishes by Proposition~\ref{LimOneBGLBGL}
below, whence there exist a unique
element $\mu_{\mathrm{BGL}}$
whose image in the
$\varprojlim $
group coincides with the element
$\lbrace u_{2i} \circ \Sigma^{\infty}(\mu_{ii}) \rbrace $.
That morphism
$\mu_{\mathrm{BGL}}$
is the required one.

In fact, the identities
$u_{2i} \circ \Sigma^{\infty}(\mu_{ii})=
\mu_{\mathrm{BGL}} \circ (u_i \wedge u_i)$
hold by the very construction of
$\mu_{\mathrm{BGL}}$.
The operation
$\mu_{\mathrm{BGL}}$
is associative because the group
${\varprojlim}^{1}\mathrm{BGL}^{8i-1,4i}(\K_i \wedge \K_i \wedge \K_i)$
vanishes by Proposition~\ref{LimOneBGLBGLBGL}.
That
$\mu_{\mathrm{BGL}}$ is commutative
follows from the vanishing of the group
${\varprojlim}^{1}\mathrm{BGL}^{4i-1,2i}(\K_i \wedge \K_i)$
(see Proposition~\ref{LimOneBGLBGL}).
The fact that
$e_{\mathrm{BGL}}$
is a two-sided unit for the multiplication
$\mu_{\mathrm{BGL}}$
follows by Proposition~\ref{LimOneBGL}, which shows that the group
${\varprojlim}^{1}\mathrm{BGL}^{2i-1,i}(\K_i)$
vanishes.
\end{proof}

\begin{definition}\label{def:BGL-over-base}
Let $S$ be a Noetherian finite-dimensional scheme, with
$f\colon  S \to \Spec(\mathbb{Z})$
being the canonical morphism.
Let
$f^\ast\colon \mathrm{SH}(\mathbb{Z}) \to \SH(S)$
be the
strict symmetric monoidal pull-back functor from~\ref{prop:base-change-spectra}.
Set
\begin{align*}
  \mu^S_{\mathrm{BGL}} & \colon\!\! = f^\ast(\mathrm{BGL}) \wedge f^\ast(\mathrm{BGL})
  \xra{\mathrm{can}} f^\ast(\mathrm{BGL} \wedge \mathrm{BGL})
  \xra{f^\ast(\mu_{\mathrm{BGL}})} f^\ast(\mathrm{BGL})\\
  e^S_{\mathrm{BGL}} & \colon\!\! = S^0 \xra{can} f^\ast(S^0)
  \xra{f^\ast(e_{\mathrm{BGL}})} f^\ast(\mathrm{BGL})
\end{align*}
and
$\mathrm{BGL}_S = f^\ast(\mathrm{BGL})$.
Then
$(\mathrm{BGL}_S, \mu^S_{\mathrm{BGL}}, e^S_{\mathrm{BGL}})$
is a commutative monoid in
$\SH(S)$.
Note that $\mathrm{BGL}^S$ satisfies the
conditions from Definition~ref{BGL} in $\MS(S)$
by Remark~\ref{PullBackBGL}.
\end{definition}

We will sometimes refer to a monoid in $\SH(S)$ as a $\Pro^1$-{\em ring spectrum}.

\begin{corollary}
  The multiplicative structure on the functor
  $\mathrm{BGL}_S^{*,*}$
  induced by the pairing
  $\mu^S_{\mathrm{BGL}}$
  and the unit
  $e^S_{\mathrm{BGL}}$
  coincides with the naive product structure
  (\ref{NaiveProducts}).
\end{corollary}

\begin{proof}
  Follows from Theorem~\ref{MonoidalStrOnBGL}.
\end{proof}

\begin{corollary}
  The functor isomorphism
  $[X,\K_0] \to [\Sigma^{\infty}_{\Pro^1}(X), \mathrm{BGL}_S]$
  respects the multiplicative structures on both sides.
  In particular, the isomorphism
  $Ad\colon K^{}_\ast \to \mathrm{BGL}^{-\ast,0}_S$
  of cohomology theories
  on $\SmOp/S$
  is an isomorphism of ring cohomology theories
  in the sense of
  \cite{PSorcoh}.
\end{corollary}

\begin{proof}
  Follows from Theorem~\ref{MonoidalStrOnBGL}.
\end{proof}

\begin{remark}
\label{Weibel}
  Let $S$ be a finite dimensional Noetherian scheme, with
  $f\colon S \to \Spec(\mathbb{Z})$ being the canonical morphism.
  Then by
  \cite[Thm.~6.9]{V1}
  the $\Pro^1$-spectrum
  $\mathrm{BGL}_S$ over $S$ as defined in~\ref{def:BGL-over-base}
  represents homotopy invariant $K$-theory as introduced in~\cite{We}.
  The triple
  $(\mathrm{BGL}_S, \mu^S_{\mathrm{BGL}}, e^S_{\mathrm{BGL}})$ is a distinguished
  monoidal structure on the the $\Pro^1$-spectrum
  $\mathrm{BGL}_S$.
\end{remark}

\subsection{Preliminary computations II}
Let
$\mathrm{BGL}$
be the $\Pro^1$-spectrum defined in~\ref{BGL}.
We will identify in this section the functors
$\mathrm{BGL}^{0,0}$
and
$\mathrm{BGL}^{2i,i}$,
$\mathrm{BGL}^{-1,0}$
and
$\mathrm{BGL}^{2i-1,i}$
on the motivic unstable category $\mathrm{H}_\bullet(S)$
as in Section~\ref{PreliminaryCompI}.

\begin{lemma}
\label{BGLofGrGr}
Let
$S=\Spec(\mathbb{Z})$.
For every integer $i$ the map
\begin{equation*}
\mathrm{BGL}^{-1,0}(S) \otimes_{}
\mathrm{BGL}^{2i,i}(\K_i \wedge \K_i) \to
\mathrm{BGL}^{2i-1,i}(\K_i \wedge \K_i)
\end{equation*}
induced by the naive product structure is an isomorphism.
The same holds if we replace
$\K_i \wedge \K_i$
by
$\K_i \wedge \K_i \wedge \Pro^1$
or by
$\K_i \wedge \K_i \wedge \Pro^1 \wedge \Pro^1$.
\end{lemma}

\begin{proof}
Since diagram
(\ref{OneDiagram})
commutes,
it suffices to consider the case $i=0$.
Furthermore we may replace the pointed motivic space $\K_i$ with
$\mathbb{Z} \times \mathrm{Gr}$
since the map
$i\colon \mathbb{Z} \times \mathrm{Gr} \to \K_i=\K$
is a motivic weak equivalence. The functor isomorphism
$\mathbb{K}^{}_\ast \to \mathrm{BGL}^{- *,0}$
is an isomorphism of ring cohomology theories.
Thus it remains to check that the map
\begin{equation*}
K^{}_1(S) \otimes_{} K^{}_0((\mathbb{Z} \times \mathrm{Gr}) \wedge (\mathbb{Z} \times \mathrm{Gr})) \to
K^{}_1((\mathbb{Z} \times \mathrm{Gr}) \wedge (\mathbb{Z} \times \mathrm{Gr}))
\end{equation*}
is an isomorphism.
This can be checked arguing as in the proof of Lemma
\ref{BGLofGr}
and using Lemma
\ref{KofGrsmashGr}
and Lemma
\ref{KofSmash}.
The cases of
$\K_i \wedge \K_i \wedge \Pro^1$
and
$\K_i \wedge \K_i \wedge \Pro^1 \wedge \Pro^1$
are proved by the same arguments.
\end{proof}

\begin{lemma}
\label{GeomRealIsom2}
Suppose that
$S=\Spec(\mathbb{Z})$,
and let
$r\colon \SH(S) \to \SH_{\CC\Pro^1}$
be the topological realization functor.
Then for every integer $i$ the
homomorphism
$\mathrm{BGL}^{2i,i}(\K \wedge \K) \to ( r\mathrm{BGL})^{2i}\bigl(r(\K \wedge \K)\bigr)
 \iso ( r\mathrm{BGL})^{2i}(r\K \wedge r\K)$
is bijective.
\end{lemma}

\begin{proof}
Since the $(2,1)$-periodicity isomorphism for $\BGL$ described in
Remark~\ref{21Periodicity}
and the Bott periodicity isomorphism for $r\BGL$ are
compatible by~\ref{example:stable-real-bgl},
it suffices to consider the case $i=0$.
We may replace the pointed motivic space $\K$
with
$\mathbb{Z} \times \mathrm{Gr}$
as in the proof of Lemma~\ref{GeomRealIsom}.
It remains to check that the realization homomorphism
\begin{eqnarray*}
  \mathrm{BGL}^{0,0}((\mathbb{Z} \times \mathrm{Gr}) \wedge (\mathbb{Z} \times \mathrm{Gr}))& \to&
( r\mathrm{BGL})^{0}(r((\mathbb{Z} \times \mathrm{Gr}) \wedge (\mathbb{Z} \times \mathrm{Gr}))) \\
& = & (r\mathrm{BGL})^{0}(r(\mathbb{Z} \times \mathrm{Gr}) \wedge r(\mathbb{Z} \times \mathrm{Gr}))
\end{eqnarray*}
is an isomorphism.
By Example~\ref{example:smash-product}
the $\Pro^1$-spectrum
$\Sigma^\infty_{\Pro^1}((\mathbb{Z} \times \mathrm{Gr})\wedge (\mathbb{Z} \times \mathrm{Gr}))$
is a retract of
$\Sigma^\infty_{\Pro^1}(\mathbb{Z} \times \mathrm{Gr}\times \mathbb{Z} \times \mathrm{Gr})$
in $\SH(S)$, whence it suffices to consider the
topological realization homomorphism for
$\mathbb{Z} \times \mathrm{Gr} \times \mathbb{Z} \times \mathrm{Gr}$.
Since the latter is an increasing union of the
cellular $S$-schemes
$[-n,n] \times \mathrm{Gr}(n,2n)\times [-m,m] \times \mathrm{Gr}(m,2m)$,
the result follows with the help of Lemma~\ref{CohOfCellularSpace}
as in the proof of Lemma~\ref{GeomRealIsom}.

\end{proof}

\subsection{Vanishing of certain groups II}
Consider the stable equivalence
$\mathrm{hocolim}\ \Sigma^{\infty}_{\Pro^1}(\K
\wedge \K)(-2i) \cong \mathrm{BGL} \wedge \mathrm{BGL}$
displayed in~(\ref{LevelAproxTwo})
and the corresponding short exact sequence
\begin{equation*}
0 \to {\varprojlim}^{1}\mathrm{BGL}^{4i-1,2i}(\K^{\wedge 2}) \to
\mathrm{BGL}^{0,0}(\mathrm{BGL} \wedge \mathrm{BGL}) \to
\varprojlim \mathrm{BGL}^{4i,2i}(\K^{\wedge 2}) \to 0
\end{equation*}
from Corollary~\ref{FiniteLevelTwo}.
We prove in this section the following result
\begin{proposition}
\label{LimOneBGLBGL}
If $S=\Spec (\mathbb{Z})$ then
${\varprojlim}^{1}\mathrm{BGL}^{4i-1,2i}(\K \wedge \K)=0$.
\end{proposition}

\begin{proof}
  For a pointed motivic space $A$ we abbreviate $A\wedge A$ as
  $A^{\wedge 2}$. The connecting homomorphism in the tower
  of groups forming the
  ${\varprojlim}^{1}$-term is the composite map
  \begin{equation*} \xymatrix@C=0.0cm{\mathrm{BGL}^{4i-1,2i}(\K \wedge \K)\!\!\! & &
           \mathrm{BGL}^{4(i+1)-1,2(i+1)}(\K^{\wedge 2})
     \ar[dl]_-{(\epsilon \wedge \epsilon)^\ast} \\
     &\mathrm{BGL}^{4(i+1)-1,2(i+1)}\bigl((\K \wedge \Pro^1)^{\wedge 2}\bigr)
     \ar[ul]_-{\ (\Sigma \circ \Sigma)^{-1}\circ \mathrm{tw}}\!\!\! &
     }
  \end{equation*}
  where
  $\Sigma$
  is the $\Pro^1$-suspension isomorphism,
  $\mathrm{tw}$ is induced by interchanging the two pointed motivic spaces in the
  middle of the four-fold smash product, and
  $\epsilon \colon \K\wedge \Pro^1\to \K$
  is the structure map of $\mathrm{BGL}$.
  Set $A=\mathrm{BGL}^{-1,0}(S)$ and write
  $\mathrm{B}$ for $\mathrm{BGL}$.
  Consider the diagram
  \begin{equation*}
    \xymatrix{
      A \otimes \mathrm{B}^{4i+4,2i+2}(\K^{\wedge 2}) \ar[r]
      \ar[d]_-{\id \otimes (\epsilon \wedge \epsilon)^\ast} &
      \mathrm{B}^{4i+3,2i+2}(\K^{\wedge 2})
      \ar[d]^-{(\epsilon \wedge \epsilon)^\ast} \\
      A \otimes \mathrm{B}^{4i+4,2i+2}\bigl((\K \wedge \Pro^1)^{\wedge 2}\bigr)
      \ar[d]_-{\id \otimes (\Sigma \circ \Sigma)^{-1} \circ \mathrm{tw}} \ar[r] &
      \mathrm{B}^{4i+3,2i+2)}\bigl((\K \wedge \Pro^1)^{\wedge 2}\bigr)
      \ar[d]^-{(\Sigma \circ \Sigma)^{-1}\circ \mathrm{tw}} \\
      A \otimes \mathrm{B}^{4i,2i}(\K^{\wedge 2}) \ar[r]&
      \mathrm{B}^{4i-1,2i}(\K^{\wedge 2})}
  \end{equation*}
where the horizontal arrows are induced by the naive product structure
on the functor
$\mathrm{BGL}^{\ast,\ast}$.
Clearly it commutes.
The horizontal arrows are isomorphisms by Lemma~\ref{BGLofGrGr}.
It follows that
${\varprojlim}^{1}\mathrm{BGL}^{4i-1,2i}(\K^{\wedge 2})=
{\varprojlim}^{1}\bigl(A \otimes \mathrm{BGL}^{4i,2i}(\K^{\wedge 2})\bigr)$
where in the last tower of groups the connecting maps are
$\id \otimes \bigl((\Sigma \circ \Sigma)^{-1}_{\Pro^1} \circ \mathrm{tw})
  \circ (\epsilon \wedge \epsilon)^\ast\bigr)$.
It remains to prove the following statement.

\begin{claim}
  One has ${\varprojlim}^{1}(A \otimes \mathrm{BGL}^{4i,2i}(\K^{\wedge 2}))=0$.
\end{claim}

Since
$A= \mathrm{BGL}^{-1,0}(S)=K^{}_1(S)=\mathbb{Z} /2 \mathbb{Z}$,
there is an isomorphism
$A \otimes \mathrm{BGL}^{4i,2i}(\K^{\wedge 2})=
\mathrm{BGL}^{4i,2i}(\K^{\wedge 2})/m\mathrm{BGL}^{4i,2i}(\K^{\wedge 2})$
with $m=2$. The connecting map in the tower are just the mod-$m$
reduction of the maps
$(\Sigma \circ \Sigma)^{-1} \circ \mathrm{tw} \circ (\epsilon \wedge \epsilon)^\ast$.
Now a chain of isomorphisms completes the proof of the Claim.
\begin{align*}
  {\varprojlim}^{1}\mathrm{BGL}^{4i,2i}(\K^{\wedge 2})/m & \iso
  {\varprojlim}^{1}(r\mathrm{BGL})^{4i}(r \K^{\wedge 2})/m \iso
  {\varprojlim}^{1}\mathbb{B}\mathrm{U}^{4i}(r \K^{\wedge 2})/m \\
  & \iso{\varprojlim}^{1}\mathbb{B}\mathrm{U}^{4i}\bigl((\mathbb{Z} \times \mathrm{BU})
  \wedge (\mathbb{Z} \times \mathrm{BU})\bigr)/m \\
  & \iso {\varprojlim}^{1}\mathbb{B}\mathrm{U}^{4i}\bigl((\mathbb{Z} \times \mathrm{BU})
  \wedge (\mathbb{Z} \times \mathrm{BU}); \mathbb{Z}/m\bigr)\\ & \iso
  K_\mathrm{top}^1(\mathbb{B}\mathrm{U} \wedge \mathbb{B}\mathrm{U}; \mathbb{Z}/m)
  \iso K_\mathrm{top}^1(\mathbb{B}^0\mathrm{U} \wedge \mathbb{B}^0\mathrm{U};
  \mathbb{Z}/m)
  \\ & \iso {\varprojlim}^{1}K_\mathrm{top}^{4i}\bigl(\mathrm{Gr}(b(i),2b(i)) \wedge
  \mathrm{Gr}(b(i),2b(i));\mathbb{Z}/m\bigr)=0
\end{align*}
The first isomorphism follows from Lemma~\ref{GeomRealIsom2}.
The second isomorphism is induced by the levelwise weak equivalence
$\mathbb{B}\mathrm{U} \simeq r\mathrm{BGL}$
mentioned in Lemma~\ref{SigzagWeakEquiv}.
The third isomorphism is induced by the image of the weak equivalence
$\K_i \simeq \mathbb{Z} \times \mathrm{Gr}$ under topological realization.
The forth and fifth isomorphism hold since
$\mathbb{B}\mathrm{U}^{4i+1}(\mathbb{Z} \times \mathrm{BU})=0$.
The sixth isomorphism is induced by the stable equivalence
$\mathbb{B}^0\mathrm{U} \simeq \mathbb{B}\mathrm{U}$ from Lemma~\ref{BUnot},
the seventh one is induced by the stable equivalence
$\mathbb{B}^f\mathrm{U}\simeq \mathbb{B}^0\mathrm{U}$ from Lemma~\ref{FiniteAprox}. The last one holds since all groups in the tower
are finite.
\end{proof}

\subsection{Vanishing of certain groups III}
Consider the stable equivalence
\begin{equation*}
\mathrm{hocolim}\ \Sigma^{\infty}_{\Pro^1}(\K
\wedge \K \wedge \K )(-3i) \cong
\mathrm{BGL} \wedge \mathrm{BGL} \wedge \mathrm{BGL}
\end{equation*}
from~(\ref{LevelAproxTwo})
and the induced short exact sequence
\begin{equation*}
0 \to {\varprojlim}^{1}\mathrm{BGL}^{8i-1,4i}(\K^{\wedge 3}) \to
\mathrm{BGL}^{0,0}(\mathrm{BGL}^{\wedge 3}) \to
\varprojlim \mathrm{BGL}^{8i,4i}(\K^{\wedge 3}) \to 0.
\end{equation*}

\begin{proposition}
\label{LimOneBGLBGLBGL}
If $S=\Spec (\mathbb{Z})$ then
${\varprojlim}^{1}\mathrm{BGL}^{8i-1,4i}(\K \wedge \K \wedge \K)=0$.
\end{proposition}

\begin{proof}
  This is proved in the same way as Proposition~\ref{LimOneBGLBGL}.
\end{proof}

\subsection{$\mathrm{BGL}$ as an oriented commutative $\Pro^1$-ring spectrum}

Following Adams and Morel, we define an orientation of a commutative
$\Pro^1$-ring spectrum. However we prefer to use a Thom class rather than
a Chern class.
Let $\Pro^\infty=\bigcup \Pro^n$ be the motivic space
pointed by $\infty\in \Pro^1 \hra \Pro^\infty$.
$\mathcal{O}(-1)$
be the tautological line bundle over
$\Pro^{\infty}$. It is also known
as the Hopf bundle. If $V\to X$ is a vector bundle over $X\in \Sm/S$, with
zero section $z\colon X\hra V$, let $\Th_X(V) = V/(V\smallsetminus z(X))$
be the Thom space of $V$, considered as a pointed motivic space over $S$.
For example $\Th_X(\Aff^n_X)\simeq S^{2n,n}$.
Define $\Th_{\Pro^\infty}\bigl(\mathcal{O}(-1)\bigr)$
as the obvious colimit of the Thom spaces $\Th_{\Pro^n}\bigl(\mathcal{O}(-1)\bigr)$.

\begin{definition}
\label{OrientationViaThom}
Let $E$ be a commutative $\Pro^1$-ring spectrum. An orientation of
$E$ is an element
$th \in E^{2,1}(\Th_{\Pro^\infty}(\mathcal{O}(-1)))=
      E^{2,1}_{\Pro^{\infty}}(\mathcal{O}(-1))$
such that its restriction to the Thom space of the fibre over the distinguished
point coincides with the element
$\Sigma_{\Pro^1}(1) \in E^{2,1}(\Th(1))= E^{2,1}(\Pro^1,\infty)$.
\end{definition}

\begin{remark}
\label{ThomAndChern}
Let $th$ be an orientation of $E$. Set
$c:=z^\ast(th) \in E^{2,1}(\Pro^{\infty})$.
Then
\cite[Prop.~6.5.1]{PY}
implies that
$c|_{\Pro^1}= - \Sigma_{\Pro^1}(1)$.
The class
$th(\mathcal{O}(-1)) \in E^{2,1}_{\Pro^{\infty}}(\mathcal{O}(-1))$
given by
(\ref{ThomClass})
coincides with the element $th$
by
\cite[Thm.~3.5]{PSorcoh}.
Thus another possible definition of an orientation of $E$
is the following.
\end{remark}

\begin{definition}
\label{OrientationViaChern}
Let $E$ be a commutative $\Pro^1$-ring spectrum. An orientation of
$E$ is an element
$c \in E^{2,1}(\Pro^{\infty})$
such that
$c|_{\Pro^1}= - \Sigma_{\Pro^1}(1)$
(of course the element $c$ should be regarded as the first
Chern class of the Hopf bundle
$\mathcal{O}(-1)$
on
$\Pro^{\infty}$).
\end{definition}

\begin{remark}
\label{ThomAndChern2}
Let $c$ be an orientation of the
commutative $\Pro^1$-ring spectrum $E$. Consider
the element
$th(\mathcal{O}(-1)) \in E^{2,1}_{\Pro^{\infty}}(\mathcal{O}(-1))$
given by~(\ref{ThomClass})
and set $th= th(\mathcal{O}(-1))$.
It is straightforward to check that
$th|_{\Th(1)}= \Sigma_{\Pro^1}(1)$.
Thus $th$ is an orientation of $E$.
Clearly
$c =z^\ast(th) \in E^{2,1}(\Pro^{\infty})$,
whence the two definitions of orientations of $E$
are equivalent.
\end{remark}

\begin{example}\label{OrientationOfMGLandK}
Set
$c^K =(- \beta) \cup \bigl([\mathcal{O}]-[\mathcal{O}(1)]\bigr)
\in \mathrm{BGL}^{2,1}(\Pro^{\infty})$.
The relation
(\ref{BottAndSuspensionTwo})
shows that $c^K$ is an orientation of
$\mathrm{BGL}$.
Consider
$th(\mathcal{O}(-1)) \in \mathrm{BGL}^{2,1}_{\Pro^{\infty}}(\mathcal{O}(-1))$
given by
(\ref{ThomClass})
and set
$th^K = th(\mathcal{O}(-1))$.
The class
$th^K$
is the same orientation of
$\mathrm{BGL}$.
\end{example}

The orientation of
$\mathrm{BGL}$
described in Example~\ref{OrientationOfMGLandK}
has the following property.
The map
(\ref{BGLAndTTrings2})
$$
\mathrm{BGL}^{*,*} \to K^{}_\ast
$$
which takes
$\beta$ to \ $-1$
is an oriented morphism of oriented cohomology
theories, provided that $K^{}_\ast$ is oriented
via the Chern structure
$L/X \mapsto [\mathcal{O}]- [L^{-1}] \in K_0(X)$.

\subsection{$\mathrm{BGL}^{*,*}$ as an oriented ring cohomology theory}

An oriented $\Pro^1$-ring spectrum $(E,c)$ defines an oriented cohomology theory on
$\SmOp$ in the sense of
\cite[Defn.~3.1]{PSorcoh}
as follows.
The restriction of the functor $E^{*,*}$ to the category
$\SmOp$ is a ring cohomology theory.
By
\cite[Thm.~3.35]{PSorcoh}
it remains to construct
a Chern structure on
$E^{*,*}|_{\SmOp}$
in the sense of
\cite[Defn.~3.2]{PSorcoh}.
The functor isomorphism
$\Hom_{\mathrm{H}_\bullet(S)}(- , \Pro^{\infty}) \to \mathrm{Pic}(-)$
on the category
$\Sm/S$
provided by
\cite[Thm.~4.3.8]{MV}
takes the class of the canonical map
$\Pro^\infty_+ \to {\Pro^{\infty}}$ to the class of the tautological line bundle
$\mathcal{O}(-1)$
over
$\Pro^{\infty}$.
Now for a line bundle $L$ over $X\in \Sm/S$
set
$c(L)=f^\ast_L(c) \in E^{2,1}(X)$,
where the morphism
$f_L\colon X_+ \to \Pro^{\infty}$
in $\mathrm{H}_\bullet(S)$ corresponds to the class $[L]$ of $L$ in the group
$\mathrm{Pic}(X)$. Clearly,
$c(\mathcal{O}(-1))=c$.
The assignment
$L/X \mapsto c(L)$
is a Chern structure on
$E^{*,*}|_{\SmOp}$
since
$c|_{\Pro^1}= - \Sigma_{\Pro^1}(1) \in E^{2,1}(\Pro^1,\infty)$.
With that Chern structure
$E^{*,*}|_{\SmOp}$
is an oriented ring cohomology theory
in the sense of
\cite{PSorcoh}.
In particular,
$(\mathrm{BGL},c^K)$
defines an oriented ring cohomology theory on
$\SmOp$.

This Chern structure induces
 a theory of Thom classes
\[ V/X \mapsto th(V) \in E^{2\mathrm{rank}(V),\mathrm{rank}(V)}\bigl(\Th_X(V)\bigr)\]
on
$E^{\ast,\ast}|_{\SmOp}$
in the sense of
\cite[Defn.~3.32]{PSorcoh} as follows.
There is a unique theory of Chern classes
$V \mapsto c_i(V) \in E^{2i,i}(X)$
such that for every line bundle $L$ on $X$ one has
$c_1(L)=c(L)$. Now for a rank $r$ vector bundle
$V$ over $X$ consider the vector bundle
$W:= {\bf 1} \oplus V$
and the associated projective vector bundle
$\Pro(W)$
of lines in $W$.
Set
\begin{equation}
\label{ThomBarClass}
\bar th(V)= c_r(p^\ast(V) \otimes \mathcal{O}_{\Pro(W)}(1)) \in E^{2r,r}(\Pro(W)).
\end{equation}
It follows from
\cite[Cor.~3.18]{PSorcoh}
that the support extension map
\begin{equation*}E^{2r,r}\bigl(\Pro(W)/(\Pro(W)\smallsetminus \Pro(\mathbf{1}))\bigr)
\to E^{2r,r}\bigl(\Pro(W)\bigr)
\end{equation*}
is injective and
$\bar th(E) \in E^{2r,r}\bigl(\Pro(W)/(\Pro(W)\smallsetminus \Pro(\mathbf{1}))\bigr) $.
Set
\begin{equation}
\label{ThomClass}
th(E)= j^\ast(\bar th(E)) \in E^{2r,r}\bigl(\Th_X(V)\bigr),
\end{equation}
where
$j\colon \Th_X(V) \to \Pro(W)/ (\Pro(W) \smallsetminus \Pro({\bf 1}))$
is the canonical motivic weak equivalence of pointed motivic spaces
induced by the open embedding $V\hra \Pro(W)$.
The assignment $V/X$ to $th(V)$ is a theory of Thom classes
on $E^{*,*}|_{\SmOp}$
(see the proof of
\cite[Thm.~3.35]{PSorcoh}). So the Thom classes are natural,
multiplicative and satisfy the following Thom isomorphism property.

\begin{theorem}
\label{ThomIsomorphism}
For a rank $r$ vector bundle
$p\colon V \to X$
on $X\in \Sm/S$ the map
\begin{equation*}
-\cup th(V)\colon  E^{*,*}(X) \to E^{*+2r,*+r}\bigl(\Th_X(V)\bigr)
\end{equation*}
is an isomorphism of the two-sided
$E^{*,*}(X)$-modules,
where
$-\cup th(V)$
is written for the composition map
$\bigl(-\cup th(V)\bigr) \circ p^\ast$.
\end{theorem}

\begin{proof}
  See \cite[Defn.~3.32.(4)]{PSorcoh}.
\end{proof}

\appendix

\section{Motivic homotopy theory}
\label{sec:motiv-homot-theory}

The aim of this section is to present details on the model
structures we use to perform homotopical calculations. Our
reference on model structures is \cite{Hovey:book}. For the
convenience of the reader who is not familiar with model structures,
we recall the basic features and purposes of the theory below,
after discussing categorical prerequisites.

\subsection{Categories of motivic spaces}
\label{sec:categ-presh}

Let $S$ be a
Noetherian separated scheme of finite Krull dimension ({\em base
scheme\/} for short). The category of smooth
quasi-projective $S$-schemes is denoted $\Sm/S$. A smooth morphism
is always of finite type. In particular, $\Sm/S$ is equivalent to
a small category.

The category
of compactly generated topological spaces is denoted $\Top$,
the category of simplicial sets is denoted $\sSet$. The
set of $n$-simplices in $K$ is $K_n$.

\begin{definition}\label{def:motivic-space}
  A motivic space over $S$ is a functor
  $A\colon \Sm/S^\op\to \sSet$. The category of
  motivic spaces over $S$ is denoted $\M(S)$.
\end{definition}

For $X\in \Sm/S$ the motivic space sending $Y\in \Sm/S$ to the
discrete simplicial set $\Hom_{\Sm/S}(Y,X)$ is denoted $X$ as
well. More generally, any scheme $X$ over $S$ defines a motivic
space $X$ over $S$. Any simplicial set $K$ defines a constant
motivic space $K$. A pointed motivic space is a pair $(A,a_0)$,
where $a_0\colon S\to A$.
Usually the basepoint will be omitted from the
notation. The resulting category is denoted
$\M_\bullet(S)$.

\begin{definition}\label{def:base-change}
  A morphism $f\colon S\to S^\prime$ of base schemes
  defines the functor
  \begin{equation*}
    f_\ast\colon M_\bullet(S)\to M_\bullet(S^\prime)
  \end{equation*}
  sending $A$ to $(Y\to S^\prime)\mapsto A(S\times_{S^\prime} Y)$.
  Left Kan extension produces a left adjoint
  $f^\ast\colon M_\bullet(S^\prime)\to M_\bullet(S)$ of $f_\ast$.
\end{definition}

If $A$ is a motivic space, let $A_+$ denote the pointed motivic
space $(A\coprod S,i)$, where $i\colon S\to A\coprod S$ is the
canonical inclusion. The category $\M_\bullet(S)$ is closed symmetric
monoidal, with smash product $A\wedge B$ defined by
the sectionwise smash product
\begin{equation}\label{eq:1}
  \bigl(A\wedge B\bigr)(X) \defeq A(X)\wedge B(X)
\end{equation}
and with internal hom $\intHom_{M_\bullet(S)}(A,B)$ defined by
\begin{equation}\label{eq:2}
  \intHom_{M_\bullet(S)}(A,B)(x\colon X\to S)_n\defeq
  \Hom_{\M_\bullet(S)}(A\wedge \Delta^n_+,x_\ast x^\ast B).
\end{equation}
In particular, $M_\bullet(S)$ is also enriched over the category of
pointed simplicial sets, with enrichment $\sSet_\bullet(A,B)\defeq
\intHom_{M_\bullet(S)}(A,B)(S)$. The {\em mapping cylinder\/} of a
map $f\colon A\to B$ is the pushout of the diagram
\begin{equation}\label{eq:10}
\xymatrix{
  A\wedge \partial \Delta^1_+  \ar[d]^-{\rho} \ar[r]^-{\cong} &
     A\coprod A \ar[r]^-{\id_A\coprod f} & A\coprod B \ar[d] \\
  A\wedge \Delta^1_+ \ar[rr] && \Cyl(f)}.
\end{equation}
The composition of the canonical maps $A\hra \Cyl(f) \to B$ is
$f$.

The {\em pushout product\/} of two maps $f\colon A\to C$ and
$g\colon B\to D$ of motivic spaces over $S$ is the map $f\push
g\colon A\wedge D\cup_{A\wedge B} C\wedge B \to C\wedge D$ induced
by the commutative diagram
\begin{equation}\label{eq:pushout-product}
  \xymatrix{
  A\wedge B \ar[r] \ar[d]& A\wedge D \ar[d]\\
  C\wedge B \ar[r] & C\wedge D.}
\end{equation}
The functor $f^\ast\colon M_\bullet(S^\prime) \to M_\bullet(S)$
induced by $f\colon S \to S^\prime$ is strict
symmetric monoidal in the sense that
there are isomorphisms
\begin{equation}\label{eq:strict}
    \xymatrix{f^\ast(A)\wedge f^\ast(B) \ar[r]^-\cong  & f^\ast(A\wedge B)
    \quad \mathrm{and}
   \quad f^\ast(S^\prime_+) \ar[r]^-\cong &  S_+}
\end{equation}
which are natural in $A$ and $B$. The isomorphisms~(\ref{eq:strict})
are induced by
the corresponding isomorphisms for the strict symmetric monoidal
pullback functor sending $X\in
\Sm_{S^\prime}$ to $S\times_{S^\prime} X\in \Sm/S$. This ends the categorical
considerations.

\subsection{Model categories}
\label{sec:model-categories}

The basic purpose of a model structure is to give a framework
for the construction of a homotopy category. Suppose
$w\mathcal{C}$ is a class of morphisms in
a category $\mathcal{C}$ one
wants to make invertible. Call them weak equivalences.
One can define the homotopy ``category''
of the pair $(\mathcal{C},w\mathcal{C})$ to be the target of the
universal
``functor'' $\Gamma\colon \mathcal{C}\to \Ho(\mathcal{C},w\mathcal{C})$
such that every weak equivalence is mapped to an
isomorphism. In general, this homotopy ``category'' may not
be a category: it has hom-classes, but not necessarily
hom-sets. If one requires the existence of two auxiliary classes
of morphisms $f\mathcal{C}$ (the fibrations) and $c\mathcal{C}$
(the cofibrations), together with certain compatibility axioms,
one does get a homotopy category $\Ho(\mathcal{C},w\mathcal{C})$
and an explicit description of
the hom-sets in it.

\begin{theorem}[Quillen]\label{thm:quillen-hocat}
  Let $(w\mathcal{C},f\mathcal{C},c\mathcal{C})$ be a model
  structure on a bicomplete category $\mathcal{C}$.
  Then the universal functor to the
  homotopy category $\Gamma\colon \mathcal{C}\to
  \Ho(\mathcal{C},w\mathcal{C})$
  exists and is the identity on objects.
  The set of morphisms in $\Ho(\mathcal{C},w\mathcal{C})$
  from $\Gamma A$ to $\Gamma B$ is the set of morphisms in
  $\mathcal{C}$ from $A$ to $B$ modulo a homotopy
  equivalence relation, provided that $\emptyset \to A$
  is a cofibration and $B\to \ast$ is a fibration.
\end{theorem}

Here $\emptyset$ is the initial object and $\ast$
is the terminal object in $\mathcal{C}$. An object
$A$ resp.~$B$ as in Theorem~\ref{thm:quillen-hocat}
is called {\em cofibrant\/} resp.~{\em fibrant}. Every object
$\Gamma A$ in the homotopy category is isomorphic to
an object $\Gamma C$, where $C$ is both fibrant and cofibrant.
A (co)fibration which is also a weak equivalence is usually
called a {\em trivial\/} or {\em acyclic\/} (co)fibration.

To describe the standard way to construct model structures on a bicomplete
category, one needs a definition.
\begin{definition}\label{defn:rlp}
  Let $f\colon A\to B$ and $g\colon C\to D$
  be morphisms in $\mathcal{C}$. If every commutative
  diagram
  \begin{equation*}\xymatrix{
    A  \ar[r] \ar[d]_f & C \ar[d]^g \\
    B  \ar[r] & D}
  \end{equation*}
  admits a morphism $h\colon B\to C$ such that the resulting diagram
  \begin{equation*}\xymatrix{
    A \ar[r] \ar[d]_f  & C \ar[d]^g\\
    B \ar[r] \ar[ru]^h & D}
  \end{equation*}
  commutes (a {\em lift\/} for short), then $f$ has the left
  lifting property with respect to $g$, and $g$ has the
  right lifting property with respect to $f$.
\end{definition}

Here is the standard way of constructing a model structure
on a given bicomplete category.
Choose the class of weak equivalences
such that it contains all identities, is closed under retracts
and satisfies the two-out-of-three axiom. Pick a set $I$
(the {\em generating cofibrations\/}) and
define a cofibration to be a morphism which is a retract of
a transfinite composition of cobase changes of morphisms in $I$.
Pick a set $J$ (the {\em generating acyclic cofibrations\/})
of weak equivalences which are also cofibrations
and define the fibrations to be those morphisms which have
the right lifting property with respect to every morphism in $J$.
Some technical conditions have to be fulfilled in order to
conclude that this indeed is a model structure, which is
then called {\em cofibrantly generated}.
See \cite[Thm.~2.1.19]{Hovey:book}.

\begin{example}\label{ex:model-structure}
  In $\Top$, let the weak equivalences be the
  weak homotopy equivalences, and set
  \begin{equation*}
    I = \lbrace \partial D^n \hookrightarrow D^n \rbrace _{n\geq 0} \quad
    J = \lbrace D^n \times \lbrace 0\rbrace  \hookrightarrow
    D^n \times I \rbrace _{n\geq 0}.
  \end{equation*}
  Then the fibrations are precisely the Serre fibrations,
  and the cofibrations are retracts of generalized cell complexes
  (``generalized'' refers to the fact that cells do not have to
  be attached in order of dimension).
  In $\sSet$, let the weak equivalences be those maps
  which map to (weak) homotopy equivalences under
  geometric realization. Set
  \begin{equation*}
    I = \lbrace \partial \Delta^n \hookrightarrow \Delta^n \rbrace _{n\geq 0} \quad
  J = \lbrace \Lambda^n_j \hookrightarrow \Delta^n  \rbrace _{n\geq 1,0\leq j\leq n}
  \end{equation*}
  where $\Lambda^n_j$ is the sub-simplicial set of $\partial \Delta^n$
  obtained by removing the $j$-th face. Then the fibrations are
  precisely the Kan fibrations, and the cofibrations are
  the inclusions.
\end{example}

\begin{example}\label{ex:presheaves-model}
  For the purpose of this paper, model structures on presheaf
  categories $\Func(\mathcal{C}^\op,\sSet)$
  with values in simplicial sets are relevant.
  There is a canonical one, due to Quillen,
  which is usually referred to as the
  {\em projective\/} model structure. It has as weak equivalences
  those morphisms $f\colon A\to B$ such that
  $f(c)\colon A(c)\to B(c)$ is a weak equivalence
  for every $c\in \mathrm{Ob}\mathcal{C}$ (the {\em objectwise\/}
  or {\em sectionwise\/} weak equivalences). Set
  \begin{eqnarray*}
    I &= &\lbrace \Hom_{\mathcal{S}}(-,c)\times
        (\partial \Delta^n \hookrightarrow \Delta^n)\rbrace _{n\geq 0,\, c \in \mathrm{Ob}\mathcal{C}} \\
     J& = &\lbrace \Hom_{\mathcal{S}}(-,c)\times
        (\Lambda^n_j \hookrightarrow \Delta^n)  \rbrace _{n\geq 1,\,0\leq j\leq n,\,c \in \mathrm{Ob}\mathcal{C}}
  \end{eqnarray*}
  so that by adjointness, the fibrations are precisely the
  sectionwise Kan fibrations. There is another cofibrantly
  generated model structure with the same
  weak equivalences, due to
  Heller \cite{Heller}, such that the cofibrations are precisely
  the injective morphisms (whence the name {\em injective\/} model
  structure). The description of $J$ involves the
  cardinality of the set of morphisms in $\mathcal{C}$ and is not
  explicit. Neither is the characterization of the fibrations.
\end{example}

The morphisms of model categories are called Quillen functors.
A Quillen functor of model categories $\mathcal{M}\to \mathcal{N}$
is an adjoint pair $(F,G)\colon \mathcal{M}\to \mathcal{N}$
such that $F$ preserves cofibrations and $G$ preserves fibrations.
This condition ensures that $(F,G)$ induces an adjoint pair on
homotopy categories $(\mathcal{L}F,\mathcal{R}G)$, where
$\mathcal{L}F$ is the total left derived functor of $F$. A Quillen
functor is a Quillen equivalence if the total left derived is an
equivalence. For example, geometric realization is a
strict symmetric monoidal left Quillen equivalence $|-|\colon
\sSet \to \Top$, and similarly in the pointed setting.

If a model category has a closed symmetric monoidal structure as
well, one has the following statement, proven in \cite[Thm.~4.3.2]{Hovey:book}.

\begin{theorem}[Quillen]\label{thm:monoidal-model}
  Let $\mathcal{C}$ be a bicomplete category with a model
  structure. Suppose that $(\mathcal{C},\otimes, \one)$
  is closed symmetric monoidal. Suppose further that
  these structures are compatible in the following sense:
  \begin{itemize}
    \item The pushout product of two cofibrations is a cofibration, and
    \item the pushout product of an acyclic cofibration with a cofibration
      is an acyclic cofibration.
  \end{itemize}
  Then $A\otimes -$ is a left Quillen functor
  for all cofibrant objects $A\in \mathcal{C}$.
  In particular, there is an induced (total derived) closed
  symmetric monoidal structure on $\Ho(\mathcal{C},w\mathcal{C})$.
\end{theorem}

One abbreviates the hypotheses of Theorem~\ref{thm:monoidal-model} by
saying that $\mathcal{C}$ is a symmetric monoidal model category.
This ends our introduction to model category theory.

\subsection{Model structures for motivic spaces}
\label{sec:model-struct-motiv}

To equip $\M_\bullet(S)$ with a model structure suitable for the
various requirements (compatibility with base change, taking
complex points, finiteness conditions, having the correct motivic
homotopy category), we construct a preliminary model structure first.
Start with the following
construction, which is a special case of
the considerations in~\cite{Isaksen:flasque}. Choose any
$X\in Sm_S$ and a finite set
\begin{equation*}
  \lbrace i^j\colon Z_j \rightarrowtail X\rbrace _{j=1}^m
\end{equation*}
of closed embeddings in $\Sm/S$. Regarding $i^j$ as a monomorphism
of motivic spaces, one may form the categorical union (not the
categorical coproduct!) $i \colon \union_{j=1}^m
Z_j\hra X$. That is, $\union_{j=1}^m Z_j$ is the coequalizer in
the category of motivic spaces of the diagram
\begin{equation}\label{eq:5}
  \coprod_{j,j'} Z_j\times_{X} Z_{j'} \rightrightarrows  \coprod_{j=1}^m Z_j
\end{equation}
Call the resulting monomorphism $i\colon
\union_{j=1}^m Z_j\hra X$ {\em acceptable}. The closed embedding
$\emptyset \rightarrowtail X$ is acceptable as well. Consider the set
$\mathrm{Ace}$ of acceptable monomorphisms. Let $I_S^c$ be the set
of pushout product maps
\begin{equation}\label{eq:3}
  \lbrace i_+\push (\partial \Delta^n \hookrightarrow \Delta^n)_+
    \rbrace _{i\in\mathrm{Ace},\, n\geq 0}
\end{equation}
and let $J_S^c$ be the set of pushout product maps
\begin{equation}\label{eq:4}
  \lbrace i_+\push (\Lambda^n_j \hookrightarrow \Delta^n)_+\rbrace _{i\in\mathrm{Ace},\,
    n\geq 1,\, 0\leq j\leq n}
\end{equation}
defined via diagram~(\ref{eq:pushout-product}).

\begin{definition}\label{def:closed-schemewise}
  A map $f\colon A\to B$ in $\M_\bullet(S)$
  is a {\em schemewise weak equivalence\/} if $f\colon A(X) \to B(X)$
  is a weak equivalence of simplicial sets for all $X\in \Sm/S$.
  It is a {\em closed schemewise fibration\/} if $f\colon A\to B$
  has the right lifting property with respect to $J_S^c$.
  It is a {\em closed cofibration\/} if it has the left
  lifting property with respect to all acyclic closed schemewise fibrations
  (closed schemewise fibrations
  which are also schemewise weak equivalences).
\end{definition}

\begin{theorem}\label{thm:closed-schemewise}
  The classes described in Definition~\ref{def:closed-schemewise}
  are a closed symmetric monoidal
  model structure on $\M_\bullet(S)$, denoted
  $\M_\bullet^{cs}(S)$.
  A morphism $f\colon S\to T$ of base schemes
  induces a strict symmetric monoidal
  left Quillen functor $f^\ast\colon \M^{cs}_\bullet(T)\to \M^{cs}_\bullet(S)$.
\end{theorem}

\begin{proof}
  The existence of the model structure
  follows from \cite{Isaksen:flasque}. The pushout product
  axiom follows, because the pushout product of
  two acceptable mono\-morphism is again acceptable.
  To conclude the last statement, it suffices to
  check that $f^\ast$ maps any map in $I_T^c$ resp.~$J_T^c$
  to a closed cofibration resp.~schemewise weak equivalence.
  In fact, if $i\colon \union_{j=1}^m Z_j\hra X$
  is an acceptable monomorphism over $T$, then
  $f^\ast(i)$ is the acceptable monomorphism obtained from
  the closed embeddings
  \begin{equation*}
    \lbrace  T\times_S Z_j \rightarrowtail T\times_S X. \rbrace
  \end{equation*}
  Because $f^\ast$ is strict symmetric monoidal
  and a left adjoint, it preserves the
  pushout product. Hence $f^\ast$ even maps the set
  $I_T^c$ to the set $I_S^c$, and likewise for $J_T^c$.
  The result follows.
\end{proof}

The resulting homotopy category is equivalent -- via the identity
functor -- to the usual homotopy category of the diagram category
$\M_\bullet(S)$ (obtained via the projective model structure from
Example~\ref{ex:presheaves-model}), since
the weak equivalences are just the objectwise ones. The model
structure $\M_\bullet^{cs}(S)$ has the following advantage over
the projective model structure.

\begin{lemma}\label{lem:closed-emb-cof}
  Let $i\colon Z \rightarrowtail X$ be a closed embedding
  in $\Sm/S$. Then the induced map $i_+\colon Z_+ \to X_+$
  is a closed cofibration in $\M_\bullet(S)$. In particular,
  for any closed
  $S$-point $x_0\colon S \rightarrowtail X$ in a smooth
  $S$-scheme, the pointed motivic space
  $(X,x_0)$ is closed cofibrant.
\end{lemma}

\begin{proof}
  The first statement follows, because
  $i_+ = i_+\push (\partial \Delta^0_+ \hookrightarrow \Delta^0_+)$
  is contained in the set of generating closed cofibrations.
  The second statement follows, because cofibrations are closed
  under cobase change.
\end{proof}

Not all pointed motivic spaces are
closed cofibrant. Let $(-)^{cs} \to \Id_{\M_\bullet(S)}$ denote a
cofibrant replacement functor, for example the one obtained from
applying the small object argument to $I_S^c$. That is, the map
$A^{cs}\to A$ is a natural closed schemewise fibration and a
schemewise weak equivalence, and $A^{cs}$ is closed cofibrant.
Dually, let $\Id_{\M_\bullet(S)}\to (-)^{cf}$ denote the
fibrant replacement functor obtained by applying the small
object argument to $J_S^c$. The
closed schemewise fibrations may be characterized explicitly.

\begin{lemma}\label{lem:char-cl-sch-fib}
  A map $f\colon A\to B$ is a closed schemewise fibration if and only if
  the following two conditions hold.
  \begin{enumerate}
    \item \label{item:1} $f(X)\colon A(X) \to B(X)$
      is a Kan fibration for every $X\in \Sm/S$,
      and
    \item \label{item:2} for every finite set
      $\lbrace  Z_j \rightarrowtail X\rbrace _{j=1}^m $ of closed
      embeddings in $\Sm/S$, the induced map
      \begin{equation*}
    A(X) \to B(X)\times_{\sSet_\bullet(\union_{j=1}^m Z_j,B)}
         \sSet_\bullet(\union_{j=1}^m Z_j,A)
      \end{equation*}
     is a Kan fibration.
  \end{enumerate}
\end{lemma}

\begin{proof}
  Follows by adjointness from the definition.
  Note that condition~\ref{item:1} is a special case of condition~\ref{item:2}
  by taking the empty family.
\end{proof}

To obtain a motivic model structure, one localizes
$\M_\bullet^{cs}(S)$ as follows. Recall that an elementary
distinguished square (or simply {\em Nisnevich square\/}) is a
pullback diagram
\begin{equation*}\xymatrix{
  V \ar[r] \ar[d] &   Y  \ar[d]^-p \\
  U \ar[r]^-j &  X}
\end{equation*}
in $\Sm/S$, where $j$ is an open embedding and
$p$ is an \'etale morphism inducing an
isomorphism $Y \smallsetminus V \cong X \smallsetminus U$ of reduced closed
subschemes. Say that a pointed motivic space $C$ is {\em closed
motivic fibrant\/} if it is closed schemewise fibrant, the map
\begin{equation*}
  C\bigl(X\times_S \mathbb{A}^1_S \xra{\mathrm{pr}} X\bigr)
\end{equation*}
is a weak equivalence of simplicial sets for every $X\in \Sm/S$,
the square
\begin{equation*}\xymatrix{
  C(V)  &   C(Y)  \ar[l] \\
  C(U) \ar[u] &  C(X) \ar[u] \ar[l]}
\end{equation*}
is a homotopy pullback square of simplicial sets for every
Nisnevich square in $\Sm/S$ and $C(\emptyset)$ is contractible.

\begin{example}\label{ex:ktt-fibrant}
  Let $X\mapsto K^W(X)$ be the pointed motivic space sending
  $X\in \Sm/S$ to the loop space of the first term of the Waldhausen $K$-theory
  spectrum $W(X)$ associated to the category of big vector bundles over
  $X$ \cite{FS}, with isomorphisms as weak equivalences \cite{W}.
  That is,
  \begin{equation*}
    K^W(X) = \Omega_s W_1(X) = \Omega_s \Sing | w\mathcal{S}_\bullet
    (\mathrm{Vect}^\mathrm{big}(X),\mathrm{iso})|
  \end{equation*}
  By
  \cite[Thm. 1.11.7, Prop.~3.10]{TT}
  the space $K^W(X)$ has
  the same homotopy type as the zeroth space
  (see \cite[1.5.3]{TT}) of
  the Thomason-Trobaugh $K$-theory spectrum $K^{naive}(X)$ of $X$
  as it is defined in
  \cite[Defn.~3.2]{TT}.

  Since in our case $S$ is regular,
  then so is $X$ and thus $X$ has an ample family of line bundles by
  \cite[Examples 2.1.2.]{TT}. It follows that the zeroth space of
  the Thomason-Trobaugh $K$-theory spectrum
  $K^{\mathrm{naive}}(X)$ has
  the same homotopy type as the zeroth space
  (see \cite[1.5.3]{TT}) of
  the Thomason-Trobaugh $K$-theory spectrum $K(X \ on \ X)$ of $X$
  \cite[Thm.~1.11.7, Cor.~3.9]{TT}
  as it is defined in
  \cite[Defn.~3.1]{TT}.

  Thus for a regular $S$ and $X\in \Sm/S$ the following results hold.
  The projection induces a weak equivalene
  $K^W(X) \to K^W(X\times_S \Aff^1_S)$ \cite[Prop.~6.8]{TT}.
  By \cite[Thm.~10.8]{TT} the square
  \begin{equation*}
    \xymatrix{
    K^W(X) \ar[r] \ar[d]_-{K^W(p)} & K^W(U) \ar[d]\\
    K^W(Y) \ar[r] & K^W(V)}
  \end{equation*}
  associated to a Nisnevich square in $\Sm/S$ is a homotopy pullback square.
  Hence $K^W$ is fibrant in the projective motivic model structure
  on $\M_\bullet(S)$. However, if $i\colon Z\hra X$ is a closed embedding
  in $\Sm/S$, the map
  $$K^W(i)\colon K^W(X) \to K^W(Z)$$ is
  not necessarily a Kan fibration.
  In particular, $K^W$ is not closed schemewise fibrant. Choose
  a closed cofibration which is also a schemewise weak equivalence
  $K^W \to \mathbb{K}^W$ such that $\mathbb{K}^W$ is closed schemewise
  fibrant. It follows immediately that $\mathbb{K}^W$ is closed motivic
  fibrant, and that $\mathbb{K}^W(X)$ has the homotopy type of
  the zero term of
  the Waldhausen $K$-theory spectrum of $X$.
\end{example}

\begin{definition}\label{defn:motivic-model}
  A map $f\colon A\to B$ is a {\em motivic weak equivalence\/} if
  the map
  \begin{equation*}
    \sSet_\bullet(f^{cs},C)\colon \sSet_\bullet(B^{cs},C)\to
    \sSet_\bullet(A^{cs},C)
  \end{equation*}
  is a weak equivalence of simplicial sets for every closed
  motivic fibrant $C$. It is a {\em closed motivic fibration\/}
  if it has the right lifting property with respect to all
  acyclic closed cofibrations (closed cofibrations
  which are also motivic equivalences).
\end{definition}

\begin{example}\label{ex:motivic-weq}
  Suppose that $f\colon A\to B$ is a map in $\M_\bullet(S)$
  inducing weak equivalences
  $x^\ast f\colon x^\ast A\to x^\ast B$ of simplicial sets on
  all Nisnevich stalks $x^\ast\colon \M_\bullet(S)\to \sSet$.
  Then $f$ is a motivic weak equivalence. If $f\colon A\to B$
  is an $\Aff^1$-homotopy equivalence (for example, the
  projection of a vector bundle), then it is a motivic
  weak equivalence.
\end{example}

\begin{example}\label{ex:pione-suspension}
  The canonical covering of $\Pro^1$ shows that
  it is motivic weakly equivalent as a pointed
  motivic space to the suspension
  $S^1\wedge (\Aff^1-\lbrace 0\rbrace ,1)$, where $S^1 =\Delta^1/\partial \Delta^1$.
  Set $S^{1,0}\defeq S^1$ and $S^{1,1}\defeq (\Aff^1-\lbrace 0\rbrace ,1)$, and
  define
    \begin{equation*} S^{p,q}\defeq \bigl(S^{1,0}\bigr)^{\wedge p-q} \wedge
    \bigl(S^{1,1}\bigr)^{\wedge q} \quad \mathrm{for}\quad p\geq q\geq 0 \end{equation*}
  To generalize the example of $\Pro^1$, one can show that
  if $\Pro^{n-1} \hra \Pro^n$ is a linear embedding, then
  $\Pro^n/\Pro^{n-1}$ is motivic weakly equivalent to $S^{2n,n}$.
\end{example}

To prove that the classes from Definition~\ref{defn:motivic-model}
are part of a model structure, it is
helpful to characterize the closed motivic fibrant objects via a
lifting property. Let $J_S^{\cm}$ be the union of the set $J_S^c$
from~(\ref{eq:4}) and the set $J_S^m$ of pushout products of maps
$(\partial \Delta^n \hra \Delta^n)_+$ with maps of the form
\begin{equation}
\label{eq:6}
\xymatrix@R=0.5cm{
    X_+ \ar[r]^{\mathrm{zero}_+} & (\Aff^1_S\times_S X)_+ \\
    U_+\cup_{V_+}\Cyl(h_+)  \ar[r] &
       \Cyl\bigl( U_+\cup_{V_+}\Cyl(h_+) \to X_+\bigr) \\
    \ast \ar[r] & \emptyset_+}
\end{equation}
where $h$ is the open embedding appearing on top of a Nisnevich
square
\begin{equation*}\xymatrix{
  V \ar[r]^-h \ar[d] &   Y  \ar[d]^-p \\
  U \ar[r]^-j &  X}
\end{equation*}
in $\Sm/S$.

\begin{lemma}\label{lem:char-closed-mot-fib}
  A pointed motivic space $C$
  is closed motivic fibrant if and only if
  the map $C\to \ast$ has the right lifting property
  with respect to the set $J_S^{\cm}$.
\end{lemma}

\begin{proof}
  This follows from adjointness, the Yoneda lemma and the construction
  of $J_S^{\cm}$.
\end{proof}

\begin{theorem}\label{thm:closed-motivic-model}
  The classes of motivic weak equivalences, closed motivic fibrations
  and closed cofibrations constitute a symmetric monoidal
  model structure on $\M_\bullet(S)$, denoted $\M_\bullet^{\cm}(S)$.
  The resulting homotopy category is denoted $\mathrm{H}^\cm_\bullet(S)$
  and called the pointed motivic unstable homotopy category of $S$.
  A morphism $f\colon S\to S^\prime$ of base schemes
  induces a strict symmetric monoidal
  left Quillen functor $f^\ast\colon M_\bullet^{\cm}(S^\prime)\to
  M_\bullet^{\cm}(S)$.
\end{theorem}

\begin{proof}
  The existence of the model structure follows
  by standard Bousfield localization
  techniques. Here are some details. The problem is that
  $J^{\cm}_S$ might be to small in order to characterize all
  closed motivic fibrations. Let $\kappa$ be
  a regular cardinal strictly bigger than the cardinality
  of the set of morphisms in $\Sm/S$.
  A motivic space $A$ is $\kappa$-bounded if the union
  \begin{equation*}
     \coprod_{n\geq 0, X\in \Sm/S} A(X)_n
  \end{equation*}
  has cardinality
  $\leq \kappa$. Let $J^\kappa_S$ be
  a set of isomorphism classes of acyclic monomorphisms
  whose target is $\kappa$-bounded. One may show that given an acyclic
  monomorphism $j\colon A\hookrightarrow B$ and a $\kappa$-bounded
  subobject $C\subseteq B$, there exists a $\kappa$-bounded
  subobject $C'\subseteq B$ containing $C$ such that
  $j^{-1}(C')\hookrightarrow C'$ is an acyclic monomorphism.
  Via Zorn's lemma, one then gets that a map $f\in \M_\bullet(S)$
  has the right lifting property with respect to all acyclic
  monomorphisms if (and only if) it has the right lifting property
  with respect to the set $J^\kappa_S$. Such a map is in particular
  a closed motivic fibration. Any given map $f\colon A\to B$ can
  now be factored (via the small object
  argument) as an acyclic monomorphism $j\colon A\hra C$
  followed by a closed motivic fibration. Factoring $j$ as a closed
  cofibration followed by an acyclic closed schemewise fibration
  in the model structure of Theorem~\ref{thm:closed-schemewise}
  implies the existence of the model structure.

  To prove that the model structure is symmetric monoidal, it suffices
  -- by the corresponding statement~\ref{thm:closed-schemewise}
  for $\M_\bullet^{cs}(S)$ -- to check that the pushout product of a
  generating closed
  cofibration and an acyclic closed cofibration
  is again a motivic equivalence. However, from the fact that
  the injective motivic model structure is symmetric monoidal,
  one knows that motivic equivalences are closed under
  smashing with arbitrary motivic spaces \cite[Lemma~2.20]{DRO:motivic}.
  The first
  sentence is now proven.

  Concerning the third sentence, Theorem~\ref{thm:closed-schemewise}
  already implies that $f^\ast$ preserves closed cofibrations.
  To prove that $f^\ast$ is a left Quillen functor, it
  suffices (by Dugger's lemma \cite[Cor.~A2]{Dugger})
  to check that it maps the set $J_{S^\prime}^{\cm}$ to motivic
  weak equivalences in $\M_\bullet(S)$. One may calculate that
  $f^\ast(J_{S^\prime}^{\cm}) = J_S^{\cm}$, whence the statement.
\end{proof}

The closed motivic model structure is cofibrantly generated.
As remarked in the proof of
Theorem~\ref{thm:closed-motivic-model},
the set $J_S^{\cm}$ is perhaps not big enough to yield a full
set of generating trivial cofibrations.  Still the following Lemma,
whose analog for the projective motivic model structure
is~\cite[Cor.~2.16]{DRO:motivic}, is valid.

\begin{lemma}\label{lem:fib-filtered}
  Motivic equivalences and closed motivic fibrations with closed
  motivic fibrant codomain are closed under filtered colimits.
\end{lemma}

\begin{proof}
  By localization theory~\cite[3.3.16]{Hirschhorn} and
  Lemma~\ref{lem:char-closed-mot-fib}, a map with closed motivic
  fibrant codomain is a closed motivic fibration if and only if it
  has the right lifting property with respect to $J_S^{\cm}$.
  The domains and codomains of the maps in $J_S^{\cm}$ are
  finitely presentable, which implies the statement about
  closed motivic fibrations with closed motivic fibrant codomain.
  The statement about motivic equivalences follows, because also
  the domains of the generating closed cofibrations in
  $I_S^c$ are finitely presentable. See \cite[Lemma 3.5]{DRO:enriched}
  for details.
\end{proof}

Let $\mathcal{M}_\bullet(S)$ be the category of
simplicial objects in the category of pointed Nisnevich sheaves
on $\Sm/S$. The functor mapping a (pointed) presheaf to its
associated (pointed) Nisnevich sheaf determines by degreewise
application a functor
$a_{Nis}\colon \M_\bullet(S) \to \mathcal{M}_\bullet(S)$.
Let $\mathcal{M}_\bullet(S) \xra{i} \M_\bullet(S)$
be the inclusion functor, the right adjoint of $a_\Nis$.

\begin{theorem}\label{thm:comparison-mv}
  The pair $(a_\Nis,i)$ is a Quillen equivalence
  to the Morel-Voevodsky model structure. The
  functor $a_\Nis$ is strict symmetric monoidal.
  In particular, the total left derived functor
  of $a_\Nis$ is a strict symmetric monoidal equivalence
  \begin{equation*}
    \mathrm{H}^\cm_\bullet (S) \to \mathrm{H}_\bullet(S)
  \end{equation*}
  to the unstable pointed $\Aff^1$-homotopy category
  from \cite{MV}.
\end{theorem}

\begin{proof}
  Recall that the cofibrations in the Morel-Voevodsky
  model structure are precisely the monomorphisms.
  Since every closed cofibration is a monomorphism and
  Nisnevich sheafification preserves these, $a_\Nis$
  preserves cofibrations. The unit $A\to i\bigl(a_\Nis(A)\bigr)$
  of the adjunction is an isomorphism on all Nisnevich stalks,
  hence a motivic weak equivalence by Example~\ref{ex:motivic-weq}
  for every motivic space $A$. In particular, $a_\Nis$ maps schemewise
  weak equivalences as well as the maps in $J_S^m$
  described in~(\ref{eq:6}) to weak equivalences. Let
  $\Id_{M_\bullet(S)}\to (-)^\fib$ be the fibrant
  replacement functor in $\M_\bullet^{\cm}(S)$ obtained from
  the small object argument applied to $J_S^{\cm}$. Hence
  if $f$ is a motivic weak equivalence, then $f^\fib$ is
  a schemewise weak equivalence. One concludes
  that $a_\Nis$ preserves all weak equivalences, thus is a
  left Quillen functor. Since the unit $A\to i\bigl(a_\Nis(A)\bigr)$
  is a motivic weak equivalence for every $A$, the functor
  $a_\Nis$ is a Quillen equivalence.
\end{proof}

\begin{note}
\label{note:Quillen-eq-MV}
  Note that a map $f$ in $\M_\bullet(S)$ is a motivic weak equivalence
  if and only if $a_\Nis(f)$ is a weak equivalence in the Morel-Voevodsky
  model structure on simplicial sheaves. Conversely, a map of simplicial
  sheaves is a weak equivalence if and only if it is a motivic weak
  equivalence when considered as a map of motivic spaces.
\end{note}

\begin{remark}\label{rem:injective}
  Starting with the injective
  model structure on simplicial pre\-sheaves
  mentioned in Example~\ref{ex:presheaves-model},
  there is a model structure $\M_\bullet^{\mathrm{im}}(S)$
  on the category of pointed motivic spaces with
  motivic weak equivalences as weak equivalences and
  monomorphisms as cofibrations. It has the advantage
  that every object is cofibrant, but the disadvantage that
  it does not behave well under base change \cite[Ex.~3.1.22]{MV}
  or geometric realization (to be defined below).
  The identity functor is a left Quillen equivalence
  $\Id\colon \M_\bullet^{\cm}(S)\to \M_\bullet^{\mathrm{im}}(S)$,
  since the homotopy categories coincide.
\end{remark}

Further, let $Spc_\bullet(S)$ be the category of pointed Nisnevich
sheaves on $\Sm/S$. Recall the cosimplicial smooth scheme over $S$
whose value at $n$ is
\begin{equation}\label{eq:7}
  \Delta_S^n =
 \Spec\bigl(\mathcal{O}_S[X_0,\dotsc,X_n]/(\sum_{i=0}^n X_i =1)\bigr)
\end{equation}
The functor $Spc_\bullet(S)\to \mathcal{M}_\bullet(S)$ sending $A$ to
the simplicial object $\Sing_S(A)_n = A(-\times \Delta_S^n)$
has a left adjoint $|-|_S\colon \mathcal{M}_\bullet(S)$. It maps
$B$ to the coend
\begin{equation*}
 |B|_S = \int_{n\in \Delta} B_n\times \Delta_S^n
\end{equation*}
in the category of pointed Nisnevich sheaves.
The following statement is proved in \cite{MV}.

\begin{theorem}[Morel-Voevodsy]\label{thm:comparison-v1}
  There is a model structure
  on the category $Spc_\bullet(S)$
  such that the pair $(|-|_S,\Sing_S)$ is a Quillen equivalence
  to the Morel-Voevodsky model structure.
  The functor $|-|_S$ is strict symmetric monoidal.
  In particular, the total left derived functor of $|-|_S$
  is a strict symmetric monoidal equivalence from
  Voevodsky's pointed homotopy category to the
  unstable pointed $\Aff^1$-homotopy category
\end{theorem}

\subsection{Topological realization}
\label{sec:topol-real}

In the case where the base scheme is the complex numbers, there is
a topological realization functor $\real_\CC\colon
\M_\bullet^{\cm}(\CC)\to \Top_\bullet$ which is a strict symmetric
monoidal left Quillen functor. It is defined as follows. If $X\in
\Sm_\CC$, the set $X(\CC)$ of complex points is a topological
space when equipped with the analytic topology. Call this
topological space $X^\an$. It is a smooth manifold, and in
particular a compactly generated topological space. One may view
$X\mapsto X^\an$ as a functor $\Sm_\CC\to \Top$. Note that if
$i\colon Z \rightarrowtail X$ is a closed embedding in $\Sm_\CC$, then the
resulting map $i^\an$ is the closed embedding of a smooth
submanifold, and in particular a cofibration of compactly
generated topological spaces. Every motivic space $A$ is a
canonical colimit
\begin{equation*}
\colim_{X\times \Delta^n \to A} X\times \Delta^n \xra{\cong} A
\end{equation*}
and one defines
\begin{equation*}
 \real_\CC(A)\defeq \colim_{X\times \Delta^n \to A} X^\an \times |\Delta^n| \in \Top.
\end{equation*}
Observe that if $A$ is pointed, then so is $\real_\CC(A)$.

\begin{theorem}\label{thm:realization}
  The functor $\real_\CC\colon\M_\bullet^{\cm}(\CC)\to \Top_\bullet$
  is a strict symmetric monoidal left Quillen functor.
\end{theorem}

\begin{proof}
  The right adjoint of $\real_\CC$ maps the compactly generated
  pointed topological space $Z$ to the pointed
  motivic space $\Sing_\CC(Z)$ which
  sends $X\in \Sm_\CC$ to the pointed
  simplicial set $\sSet_{\Top}(X^\an, Z)$.
  To conclude that $\real_\CC$ is strict symmetric monoidal, it
  suffices to observe that there is a canonical
  homeomorphism $(X\times Y)^\an \iso X^\an \times Y^\an$, and that
  geometric realization is strict symmetric monoidal.

  Suppose now that $i\colon \union_{j=1}^m Z_j\hra X$ is an
  acceptable monomorphism. One computes $\real_\CC(\union_{j=1}^m Z_j)$
  as the coequalizer
  \begin{equation*}
     \coprod_{j,j'}\bigl(Z_j\times_{X} Z_{j'}\bigr)^\an
       \rightrightarrows   \coprod_{j=1}^m Z^\an_j.
  \end{equation*}
  in $\Top$.
  Every map $\bigl(Z_j\times_{X} Z_{j'}\bigr)^\an \to Z_j^\an$
  is a closed embedding of smooth submanifolds of complex projective
  space. In particular, one may equip $Z_j^\an$ with a cell complex
  structure such that $\bigl(Z_j\times_{X} Z_{j'}\bigr)^\an$ is
  a subcomplex for every $j'$. Then $\real_\CC(\union_{j=1}^m Z_j)$ is the union
  of these subcomplexes, and in particular again a subcomplex.
  It follows that $\real_\CC(i)$ is a cofibration of topological
  spaces. Since $\real_\CC$ is compatible with pushout products,
  it maps the generating closed cofibrations to cofibrations
  of topological spaces.

  To conclude that $\real_\CC$ preserves trivial cofibrations
  as well, it suffices by Dugger's lemma \cite[Cor.~A2]{Dugger} to check
  that $\real_\CC$ maps every map in $J_\CC^m$ to a weak
  homotopy equivalence. In fact, since the domains and codomains
  of the maps $\partial \Delta^m \hra \Delta^m$ are cofibrant,
  it suffices to check the latter for the maps in
  diagram~(\ref{eq:6}). In the first case, one obtains the map
  $ X^\an \hra (\Aff^1_\CC\times X)^\an \cong \mathbb R^2 \times X^\an$, in the
  second case one obtains up to simplicial homotopy equivalence
  the canonical map $U^\an\cup_{p^{-1}(V)^\an}Y^\an \to X^\an$ for a Nisnevich
  square
  \begin{equation*} \xymatrix{ V \ar[r] \ar[d] & Y \ar[d]^-p \\
    U \ar[r] & X} \end{equation*}
  This is in fact a homeomorphism of topological spaces. The result
  follows.
\end{proof}

Suppose now that $R\hra \CC$ is a subring of the complex
numbers. Let $f\colon \Spec(\CC)\to \Spec(R)$ denote the
resulting morphism of base schemes. The realization with
respect to $R$ (or better $f$) is defined as the composition
\begin{equation}\label{eq:8}
  \real_R = \real_\CC\circ f^\ast \colon
  \M_\bullet(R) \to \M_\bullet(\CC) \to \Top_\bullet.
\end{equation}
It is a strict symmetric monoidal Quillen functor.
The most relevant case is $R=\mathbb{Z}$.

\begin{example}\label{example:real-bgl}
  The topological realization of the Grassmannian
  $\Gr(m,n)$ (over any base with a complex point)
  is the complex Grassmannian with the
  usual topology. Since $\real_\CC$ commutes with
  filtered colimits, $\real_\CC(\Gr)$ is the infinite
  complex Grassmannian, which in turn is
  the classifying space $\mathrm{B}U$ for
  the infinite unitary group. Because $\real_\CC$ is a left
  Quillen functor, the topological realization
  of any closed cofibrant motivic space
  weakly equivalent to $\Gr$ 
  is homotopy equivalent to $\mathrm{B} U$.
\end{example}

\subsection{Spectra}
\label{sec:spectra}

\begin{definition}\label{defn:spectrum}
  Let $\Pro^1_S$ denote the pointed projective line over $S$.
  The category
  $\MS(S)$ of $\Pro^1$-{\em spectra\/} over $S$ has the following objects.
  A $\Pro^1$-{\em spectrum\/} $E$ consists of a sequence $(E_0,E_1,E_2,\ldots)$
  of pointed motivic spaces over $S$, and structure maps
  $\sigma_n^E\colon E_n \wedge \Pro^1 \to E_{n+1}$ for every $n\geq 0$.
  A map of $\Pro^1$-spectra is a sequence of maps of pointed motivic
  spaces which is compatible with the structure maps.
\end{definition}

\begin{example}\label{ex:suspension-spectrum}
  Any pointed motivic space $B$ over $S$ gives rise to
  a $\Pro^1_S$-suspension spectrum
  \begin{equation*}
    \Sigma^\infty_{\Pro^1} B =
    (B, B\wedge \Pro^1, B\wedge \Pro^1\wedge \Pro^1, \ldots)
  \end{equation*}
  having identities as structure maps. More generally, let
  $\Fr_n B$ denote the $\Pro^1$-spectrum having values
  \begin{equation*}
    (\Fr_n B)_{n+m} = \begin{cases}
      B\wedge {\Pro^1}^{\wedge m}  & m\geq 0 \\
      \ast                 & m < 0
  \end{cases}
  \end{equation*}
  and identities as structure maps, except for $\sigma_{n-1}^{\Fr_n B}$.
  The functor $B\mapsto \Fr_n B$ is left adjoint to the functor
  sending the $\Pro^1$-spectrum $E$ to $E_n$.
  We often write $\Sigma^\infty_{\Pro^1} B(-n)$ for $\Fr_n B$.
  For a $\Pro^1$-spectrum $E$ let
  $u_n \colon \Sigma^\infty_{\Pro^1} E_n(-n) \to E$
  be a map of $\Pro^1$-spectra adjoint to the identity map
  $E_n \to E_n$.
\end{example}

\begin{remark}\label{r:prominent-A}
  In Definition~\ref{defn:spectrum},
  one may replace $\Pro^1$ by any
  pointed motivic space $A$,
  giving the category $\MS_A(S)$ of $A$-spectra over $S$.
  Essentially the only relevant example for us is when
  $A$ is weakly equivalent to the pointed projective line $\Pro^1$.
  The Thom space $T = \Aff^1/\Aff^1-\lbrace 0\rbrace $ of the trivial
  line bundle over $S$ admits motivic weak equivalences
  \begin{equation}\label{eq:11}
    \xymatrix{\Pro^1 \ar[r]^-\sim & \Pro^1/\Aff^1 & T \ar[l]_-\sim }
  \end{equation}
  The motivic space $\Pro^1$ itself is not always the ideal
  suspension coordinate. For example,
  the algebraic cobordism spectrum $\MGL$ naturally comes as
  a $T$-spectrum.
  In order to switch between $T$-spectra and $\Pro^1$-spectra,
  consider the following general construction.
  A map $\phi\colon A\to B$ induces a functor $\phi^\ast\colon \MS_B(S)
  \to \MS_A(S)$ sending the $B$-spectrum $(E_0,E_1,\dotsc,\sigma_n^E)$
  to the $A$-spectrum
  \begin{equation*}
    (E_0,E_1,\dotsc)\quad\mathrm{with\ structure\ maps}
    \quad \sigma_n^{\phi^\ast E}=
    \sigma_n^E\circ (E_n\wedge \phi)
  \end{equation*}
  Its left adjoint $\phi_\sharp$ maps the $A$-spectrum
  $(F_0,F_1,\dotsc,\sigma_n^F)$ to the $B$-spectrum
  \begin{equation}\label{eq:9}
    \Bigl( F_0,  B\wedge F_0 \cup_{A\wedge F_0} F_1,
     B\wedge\bigl( B\wedge F_0 \cup_{A\wedge F_0} F_1\bigr)\cup_{A\wedge F_1} F_2\bigr),
     \dotsc \Bigr)
  \end{equation}
  having the canonical maps as structure maps. Note that
  for the purpose of constructing a model structure on $A$-spectra
  over $S$, the pointed motivic space $A$ has to be cofibrant in
  the model structure under consideration.
\end{remark}

The next goal is to construct a model structure on $\MS(S)$ having the motivic
stable homotopy category as its homotopy category.

\begin{definition}\label{defn:omega-spectrum}
  Let $\Omega_{\Pro^1} = \intHom_{\M_\bullet(S)}(\Pro^1_S,-)$
  denote the right adjoint of $\Pro^1\wedge -$. For a
  $\Pro^1$-spectrum $E$ with structure maps
  $\sigma^E_n\colon E_n \wedge \Pro^1\to E_{n+1}$,
  let $\omega^E_n\colon E_n \to \Omega_{\Pro^1} E_{n+1}$
  denote the adjoint structure map. A
  $\Pro^1$-spectrum $E$ is {\em closed stably fibrant\/} if
\begin{itemize}
  \item $E_n$ is closed motivic fibrant for every $n\geq 0$, and
  \item $\omega^E_n\colon E_n \to \Omega_{\Pro^1} E_{n+1}$ is a
    motivic weak equivalence for every $n\geq 0$.
\end{itemize}
\end{definition}

Any $\Pro^1$-spectrum $E$
admits a closed stably fibrant replacement.
First replace $E$ by a levelwise fibrant $\Pro^1$-spectrum
$E^\ell$ as follows. Let $E^\ell_0 = E_0^\fib$ for
a fibrant replacement in $\M_\bullet^{\cm}(S)$. Given
$E_n \to E^\ell_n$, set
\begin{equation*}E^\ell_{n+1} \colon\!\!=
  \Bigl(E^\ell_{n}\wedge \Pro^1 \cup_{E_n\wedge \Pro^1} E_{n+1}\Bigr)^\fib\end{equation*}
which yields a levelwise motivic weak equivalence
$E\to E^\ell$ of $\Pro^1$-spectra. To continue, observe that
the adjoint structure maps of any
$\Pro^1$-spectrum $F$ may be viewed as a natural transformation
\begin{equation*}
  q \colon F \to Q (F) \end{equation*}
where $Q (F)$ is the $\Pro^1$-spectrum with terms
$\Omega_{\Pro^1} F_1,\Omega_{\Pro^1} F_2,\dotsc$ and structure maps
\begin{equation*}
 \Omega_{\Pro^1}(\omega_{n+1}^F) \colon \Omega_{\Pro^1} F_{n+1} \to
 \Omega_{\Pro^1}^2 F_{n+2}.
\end{equation*}
Define $Q^\infty(E)$ as the colimit of the sequence
\begin{equation*}
  \xymatrix{  E^\ell \ar[r]^{q} & Q(E^\ell) \ar[r]^{Q(q)} &
          Q^2(E^\ell) \ar[r] & \dotsm.}
\end{equation*}

\begin{definition}\label{defn:stable-equivalence}
  A map $f\colon E\to F$ of $\Pro^1$-spectra is a {\em stable
  equivalence\/} if the map $Q^\infty(f)_n$ is a
  weak equivalence for every $n\geq 0$. It is a {\em closed
  stable fibration\/} if $f_n\colon E_n\to F_n$ is
  a closed motivic fibration and the induced map
  $E_n\to F_n\times_{Q^\infty(E)_n}Q^\infty(F)_n$
  is a motivic weak equivalence for every $n\geq 0$. It
  is a {\em closed cofibration\/}
  if $f_n\colon E_n\to F_n$ and $F_n \wedge \Pro^1 \cup_{E_n \wedge \Pro^1} E_{n+1}
  \to F_{n+1}$ are closed cofibrations for every $n\geq 0$.
\end{definition}

\begin{theorem}\label{thm:stable-model}
  The classes from Definition~\ref{defn:stable-equivalence}
  are a model structure on the category of $\Pro^1$-spectra,
  denoted $\MS^\cm(S)$.
  The identity functor on $\Pro^1$-spectra
  from Jardine's stable model structure to $\MS^\cm(S)$
  is a left Quillen equivalence. In particular,
  the homotopy category $\Ho\bigl(\MS^\cm(S)\bigr)$
  is equivalent to the motivic stable homotopy category $\SH(S)$.
\end{theorem}

\begin{proof}
  Recall that $\Pro^1$ is closed cofibrant by Lemma~\ref{lem:closed-emb-cof}.
  The existence of the model structure follows as
  in \cite[Thm.~2.9]{J}. Moreover, the stable equivalences
  coincide with the ones in \cite{J}, because so do the stabilization
  constructions and the unstable weak equivalences. Since
  every closed cofibration of motivic spaces is in particular
  a monomorphism, $\Id_{\MS(S)}$ is a left Quillen equivalence.
  Note that the closed cofibrations are generated by the set
  \begin{equation}\label{eq:17}
    \lbrace \Fr_m(g)
      \rbrace_{m \geq 0,\, g\in I^c_S}
  \end{equation}
  with $I^c_S$ defined in~(\ref{eq:3}).
  One may also describe a set of generating acyclic cofibrations.
\end{proof}

\begin{remark}\label{rem:stable-model}
  We will identify
  the homotopy category $\Ho\bigl(\MS^\cm(S)\bigr)$
  with $\SH(S)$ via the equivalence from Theorem~\ref{thm:stable-model}.
  Note that one may form the smash product of a motivic space $A$
  and a $\Pro^1$-spectrum $E$ by setting $(A\wedge E)_n \defeq A \wedge E_n$
  and $\sigma_n^{A\wedge E} \defeq A\wedge \sigma_n^E $. If
  $A$ is closed cofibrant, $A\wedge -$ is a left Quillen functor by
  Theorem~\ref{thm:closed-motivic-model}. Because
  $\Pro^1\wedge -$ is a Quillen equivalence on Jardine's stable model
  structure by results in~\cite[Sect.~3.4]{J},
  Theorem~\ref{thm:stable-model} implies that
  $\Pro^1\wedge -\colon \MS^\cm(S)\to \MS^\cm(S)$ is a left Quillen equivalence.
  Since $\Pro^1 \simeq S^{2,1} = S^1 \wedge (\Aff^1-\lbrace 0\rbrace ,1)$, also
  $S^1\wedge -$ is a left Quillen equivalence. It induces the
  shift functor in the triangulated structure on
  $\SH(S)$.
  The triangles are those which are isomorphic to the image of
  \begin{equation*} \xymatrix{
    E \ar[r]^-f & F \ar[r] & F/E & E\wedge \Delta^1 \cup_E F \ar[l]_-\sim
    \ar[r] & S^1\wedge E} \end{equation*}
  in $\SH(S)$, where $f\colon E\to F$ is an inclusion
  of $\Pro^1$-spectra. As well,
  one has sphere spectra $S^{p,q} \in \SH(S)$ for
  all integers $p,q\in \mathbb{Z}$.
\end{remark}

\begin{example}\label{example:smash-product}
  Since $\SH(S)$ is an additive category, the canonical map
  $E\vee F\to E\times F$ is a stable equivalence. In
  the special case of $\Pro^1$-suspension spectra,
  the canonical map factors as
  \begin{equation*}\Sigma^\infty_{\Pro^1}A \vee \Sigma^\infty_{\Pro^1}B
  \iso \Sigma^\infty_{\Pro^1} (A\vee B) \to \Sigma^\infty_{\Pro^1} (A\times B)
   \to\Sigma^\infty_{\Pro^1} A  \times \Sigma^\infty_{\Pro^1}B \end{equation*}
  which shows that $\Sigma^\infty_{\Pro^1}(A\times B)$
  contains $\Sigma^\infty_{\Pro^1}(A\wedge B)$ as a retract
  in $\SH(S)$. Thus it is even a direct summand. The latter can be deduced
  as follows. The (reduced) {\em join\/} $(A,a_0) \ast (B,b_0)$ is defined as
  the pushout in the diagram
  \begin{equation*}
  \xymatrix{
    A\times B\times \partial \Delta^1 \cup \lbrace a_0\rbrace \times \lbrace b_0\rbrace
            \times \Delta^1
    \ar[r] \ar[d] & A \vee B \ar[d] \\
    A\times B\times \Delta^1 \ar[r] & A \ast B.   }
  \end{equation*}
  of pointed motivic spaces over $S$. Attaching $A\wedge (\Delta^1,0)$
  and $B\wedge (\Delta^1,0)$ to $A \ast B$ via $A\vee B$ produces a
  pointed motivic space $C$ which is equipped with a sectionwise weak
  equivalence $C\to (A\times B) \wedge S^{1,0}$. Collapsing
  $\lbrace a_0\rbrace \ast B$ and $A\ast \lbrace b_0\rbrace $ inside $C$
  yields a sectionwise
  weak equivalence $C \to (A\wedge B\wedge S^{1,0}) \vee (A\wedge S^{1,0}) \vee
  (B\wedge S^{1,0})$. Since $S^{1,0}$ is invertible in $\SH(S)$, one gets
  a splitting $\Sigma^\infty_{\Pro^1} (A\times B) \simeq
  \bigl( \Sigma^\infty_{\Pro^1} (A\wedge B)\bigr) \vee \bigl(\Sigma^\infty_{\Pro^1} A
  \bigr)\vee \bigl(\Sigma^\infty_{\Pro^1} B\bigr)$ in $\SH(S)$.
\end{example}

For a $\Pro^1$-spectrum $E$ let $\Tr_n E$ denote the $\Pro^1$-spectrum
with
\begin{equation*}
  \bigl(\Tr_n E\bigr)_{m} =
  \begin{cases}
    E_{m} & m\leq n \\
    E_n \wedge (\Pro^1)^{\wedge m-n} =
    \bigl(\Fr_n E_n\bigr)_{m} & m\geq n
  \end{cases}
\end{equation*}
and with the obvious structure maps. The structure maps of $E$
determine maps
$\Tr_n E \to \Tr_{n+1} E$ such that $E = \colim_n \Tr_{n} E$.
The canonical map $\Fr_n E_n \to \Tr_n E$ adjoint to the identity
$\id\colon E_n \to E_n$ is an identity in all levels $\geq n$, and
in particular a stable equivalence. The identity
$\id\colon E_n \wedge (\Pro^1)^{\wedge n} \to \bigl(\Fr_0 E_n)_n$ leads by
adjointness to the map
\begin{equation}\label{eq:14}
  \xymatrix{\Fr_n (E_n)\wedge (\Pro^1)^{\wedge n} \ar[r]^\iso &
   \Fr_n\bigl(E_n \wedge (\Pro^1)^{\wedge n}\bigr) \ar[r] & \Fr_0 E_n }
\end{equation}
and hence to a map $\Fr_n E_n \to \Omega_{\Pro^1}^{n} \bigl(\Fr_0 E_n \bigr)$.
Since the map~(\ref{eq:14}) is an isomorphism in all levels $\geq n$,
it is a stable equivalence. Because $\Omega_{\Pro^1}$ is a Quillen equivalence,
the map $\Fr_n E_n \to \Omega_{\Pro^1}^{n} \bigl((\Fr_0 E_n)^\fib\bigr)$
is a stable equivalence as well if $E_n$ is closed cofibrant. In fact,
the condition on $E_n$ can be removed since $\Fr_n$ preserves all weak
equivalences.
This leads to the following statement.

\begin{lemma}\label{lem:spectrum-colim}
  Any $\Pro^1$-spectrum $E$ is the colimit of a natural sequence
  \begin{equation}\label{eq:13}
    \xymatrix{ \Tr_0 E\ar[r] & \Tr_1 E\ar[r] & \Tr_2 E \ar[r] & \dotsm }
  \end{equation}
  of $\Pro^1$-spectra in which the $n$-th term is
  naturally stably equivalent to $\Fr_n E_n = \Sigma^\infty_{\Pro^1} E_n(-n)$
  from \ref{ex:suspension-spectrum}, and also
  to $\Omega_{\Pro^1}^n \bigl((\Sigma^\infty_{\Pro^1} E_n)^f\bigr)$.
\end{lemma}

One may use the description in Lemma~\ref{lem:spectrum-colim} for
computations as follows. Say that a $\Pro^1$-spectrum $F$ is {\em finite\/}
if it is stably equivalent to a $\Pro^1$-spectrum $F'$ such that
$\ast \to E'$ is obtained by attaching finitely many cells from
the set~(\ref{eq:17}).

\begin{lemma}\label{lem:filtered-colimit}
  Let $D(0)\to D(1)\to D(2) \to \dotsm $ be a sequence
  of maps of $\Pro^1$-spectra, with colimit $D(\infty)$.
  \begin{enumerate}
    \item\label{item:3} Suppose that $F$ is a finite $\Pro^1$-spectrum. The canonical
      map \begin{equation*}\colim_{i\geq 0} \Hom_{\SH(S)}\bigl(F,D(i)\bigr)
           \to \Hom_{\SH(S)}\bigl(F, D(\infty)\bigr) \end{equation*}
       is an isomorphism.
   \item\label{item:4} For any $\Pro^1$-spectrum $E$ there is a canonical
     short exact sequence
     \begin{equation}\label{eq:18} 0 \to {\varprojlim_{i\geq 0}}^{\!1}
       \bigl[S^{1,0}\wedge D(i),E\bigr] \to
    \bigl[D(\infty),E\bigr] \to
    \varprojlim_{i\geq 0}\bigl[D(i),E\bigr] \to 0
     \end{equation}
     of abelian groups, where $[-,-]$ denotes $\Hom_{\SH(S)}(-,-)$.
  \end{enumerate}
\end{lemma}

\begin{proof}
  Observe first that stable equivalences and closed stable fibrations are detected
  by the functor $Q^\infty$ which is defined in~\ref{defn:stable-equivalence} as a
  sequential colimit.
  Lemma~\ref{lem:fib-filtered}
   implies that stable equivalences and closed stable fibrations
  with closed stably fibrant codomain are closed under filtered colimits.
  Thus by Theorem~\ref{thm:quillen-hocat} one may compute
  \begin{equation*}
    \Hom_{\SH(S)}(F,\colim_{i\geq 0} D(i)) \iso
  \Hom_{\MS(S)} (F,\colim_{i\geq 0} D(i)^\fib)/\simeq \end{equation*}
  for any cofibrant $\Pro^1$-spectrum $F$, where $\simeq$ denotes the
  equivalence relation ``simplicial homotopy''.
  This implies statement~\ref{item:3}
  because $\Hom_{\MS(S)}(F,-)$ commutes with
  filtered colimits if $\ast \ra F$ is obtained by attaching finitely many cells.

  To prove the second statement, let $C$ be the coequalizer of the
  diagram
  \begin{equation*}
    \xymatrix{ \displaystyle\bigvee_{i\geq 0} D(i) \ar@<1.6ex>[r]^-f \ar@<0.1ex>[r]_-g &
      \displaystyle\bigvee_{i\geq 0}\Delta^1_+\wedge D(i)}
  \end{equation*}
  where $f$ resp.~$g$ is defined on the $i$-th summand $D(i)$
  as $D(i) = 1_+\wedge D(i) \hra \Delta^1_+\wedge D(i)$ resp.~$D(i)
  \ra D(i+1) =0_+\wedge D(i+1)\hra \Delta^1_+\wedge D(i+1)$.
  The canonical map $C \ra \colim_{i\geq 0} D(i)$ induced by the
  composition $\Delta^1_+\wedge D(i) \ra D(i) \ra \colim_{i\geq 0} D(i)$
  is a weak equivalence. In the stable homotopy category, which
  is additive, one may take the difference of $f$ and $g$, and thus
  describe $\colim_{i\geq 0} D(i)$ via the
  distinguished triangle
  \begin{equation}\label{eq:22}
      \xymatrix{ \displaystyle\bigvee_{i\geq 0} D(i) \ar@<1ex>[r]^-{f-g} &
      \displaystyle\bigvee_{i\geq 0}D(i) \ar@<1ex>[r] &
      \displaystyle\colim_{i\geq 0} D(i) \ar@<1ex>[r] &
      \displaystyle\bigvee_{i\geq 0} S^{1,0}\wedge D(i).}
  \end{equation}
  Applying $[-,E] \colon \!\! =\Hom_{\SH(S)}(-,E)$ to the triangle~(\ref{eq:22})
  produces a long exact sequence
  \begin{equation*}
    \xymatrix@C=0.9cm{ \dotsm  &
      \displaystyle\prod_{i\geq 0}\bigl[D(i),E\bigr] \ar@<-0.7ex>[l]_-{f^\ast-g^\ast\!\!} &
      \bigl[D(\infty),E\bigr] \ar@<-0.7ex>[l] &
      \displaystyle\prod_{i\geq 0} \bigl[S^{1,0}\wedge D(i),E\bigr] \ar@<-0.7ex>[l]\! &
      \dotsm \ar@<-0.7ex>[l]}
  \end{equation*}
  which may be split into the short exact sequence
  \begin{equation*}
    \xymatrix@C=0.7cm{ 0  & 
      \varprojlim_{i\geq 0} \bigl[D(i),E] \ar[l] &
      \bigl[D(\infty),E\bigr] \ar[l] &
      {\varprojlim^1_{i\geq 0}} \bigl[S^{1,0}\!\wedge D(i),E\bigr] \ar[l] &
      0. \ar[l]}
  \end{equation*}
\end{proof}

\subsection{Symmetric spectra}
\label{sec:symmetric-spectra}

There seems to be no reasonable (i.e.~symmetric monoidal) smash product for
$\Pro^1$-spectra inducing a decent symmetric monoidal smash product
on $\SH(S)$. This will be solved as in \cite{HSS} and \cite{J}.

\begin{definition}\label{defn:symm-spectrum}
A {\em symmetric $\Pro^1$-spectrum\/} $E$ over $S$ consists of
a sequence $(E_0,E_1,\ldots)$ of pointed motivic spaces over $S$,
together with group actions $(\Sigma_n)_+\wedge E_n \to E_{n}$
and structure maps $\sigma_n^E\colon
E_n \wedge \Pro^1\to E_{n+1}$ for all $n\geq 0$.
Iterations of these structure
maps are required to be as equivariant as they can,
using the permutation action
of $\Sigma_n$ on $(\Pro^1)^{\wedge n}$. A map of symmetric $\Pro^1$-spectra
is a sequence of maps of pointed motivic spaces which is compatible with
all the structure (group actions and structure maps).
Call the resulting category $\MSS(S)$.
\end{definition}

\begin{example}\label{ex:symm-spectra}
  Analogous to Example~\ref{ex:suspension-spectrum},
  the $n$-th shifted suspension spectrum $\Fr^\Sigma_n A$
  of a pointed motivic space $A$ has as values
  \begin{equation*}
    (\Fr^\Sigma_n B)_{m+n} =
    \begin{cases}
      \Sigma_{m+n}^+\wedge_{\Sigma m\times \lbrace{1\rbrace}}
      A\wedge (\Pro^1)^{\wedge m} & m\geq 0 \\
      \ast                 & m < 0
    \end{cases}
  \end{equation*}
  where the $m$-th fold smash product $(\Pro^1)^{\wedge m}$
  carries the natural permutation action.
\end{example}

Every symmetric $\Pro^1$-spectrum determines a $\Pro^1$-spectrum
by forgetting the symmetric group actions. Call the resulting functor
$u\colon \MSS(S)\to \MS(S)$. It has a left adjoint $v$, which
is characterized uniquely up to unique isomorphism by the fact that
\begin{equation}\label{eq:15}
  v\bigl(\Fr_n A\bigr) = \Fr_n^\Sigma A.
\end{equation}

The smash product $E\wedge F$
of two symmetric $\Pro^1$-spectra $E$ and $F$ is constructed as follows.
Set $(E\wedge F)_n$ as the coequalizer of the diagram
\begin{equation}\label{eq:12}
  \xymatrix@C=5em{
    {\displaystyle\coprod_{r+1+s=n} }
  \!\!\!\!\Sigma_n^+ {\wedge}_{\Sigma_r \times \Sigma_1 \times \Sigma_s} E_r\wedge\Pro^1\!\!\wedge F_s
   \ar@<-0.1ex>[r]_-{\sigma_r^E\wedge (F_s\circ \mathrm{twist})}
   \ar@<2ex>[r]^-{\sigma_r^E\wedge F_s} &
    \displaystyle\coprod_{r+s=n} \!\!\!\! \Sigma_n^+ \!\wedge_{\Sigma_r \times \Sigma_s} \!E_r\wedge F_s}
 \end{equation}
where the coequalizer is taken in the category of pointed $\Sigma_n$-motivic spaces.
The structure map $\sigma_n^{E\wedge F}$ is induced by the
structure maps $\sigma_0^F,\dotsc,\sigma_n^F$ of $F$.
One may provide natural coherence isomorphisms for associativity,
commutativity and unitality, where the unit is
$\one_S = (S_+,\Pro^1,\Pro^1\wedge \Pro^1,\dotsc,(\Pro^1)^{\wedge n},\dotsc)$
with the obvious permutation action and identities as structure maps.

We proceed with the homotopy theory of symmetric $\Pro^1$-spectra,
as in \cite{J}. As one deduces from \cite[Thm. 7.2]{Hovey:stable-model},
there exists
a cofibrant replacement functor $(-)^\cofrepl\to \Id_{\MSS(S)}$
for the model structure on symmetric $\Pro^1$-spectra with
levelwise weak equivalences and levelwise fibrations.

\begin{definition}
\label{defn:symm-stable-model}
A map $\phi\colon E\to F$ of symmetric
$\Pro^1$-spectra is a {\em levelwise acyclic fibration\/} if
$\phi_n\colon E_n\to F_n$ is an acyclic closed motivic fibration of pointed
motivic spaces over $S$ for all $n\geq 0$.
A map $\phi\colon E\to F$ of symmetric
$\Pro^1$-spectra is a {\em closed cofibration\/} if it has
the left lifting property with respect to all levelwise acyclic
fibrations.
A levelwise fibrant
symmetric $\Pro^1$-spectrum $E$ is {\em closed stably fibrant\/} if the
adjoint $E_n\to \M_X(\Pro^1,E_{n+1})$ of the structure map is
a weak equivalence for every $n\geq 0$.
A map $\phi\colon E\to F$ is a {\em stable equivalence\/} if
the map
\begin{equation*}\sSet_{\MSS_X}(\phi^\cofrepl,G)\colon \sSet_{\MSS_X}(F^\cofrepl,G)\to
 \sSet_{\MSS_X}(E^\cofrepl,G) \end{equation*} is a weak equivalence
of pointed motivic spaces for all closed
stably fibrant symmetric
$\Pro^1$-spectra $G$. The {\em closed stable fibrations\/} are then defined
by the right lifting property.
\end{definition}

\begin{theorem}[Jardine]
\label{thm:quillen-eq}
The classes of stable equivalences, closed cofibrations and
closed stable fibrations from Definition~\ref{defn:symm-stable-model}
constitute a symmetric monoidal model structure
on $\MSS(S)$.
The forgetful functor $u\colon \MSS^\cm(S)\to \MS^\cm(S)$ is a
right Quillen equivalence.
\end{theorem}

\begin{proof}
  The proof of the first statement follows as in \cite[Thm.~4.15, Prop.~4.19]{J}.
  Note that the stable equivalences in $\MSS^\cm(S)$
  and in Jardine's model structure $\MSS^\mathrm{Jar}(S)$
  coincide, since the unstable
  weak equivalences do so by Remark~\ref{rem:injective}.
  In particular, the identity $\Id\colon \MSS^\cm(S) \to \MSS^\mathrm{Jar}(S)$
  is a left Quillen equivalence. The closed cofibrations
  in $\MSS^\cm(S)$ are generated by the inclusions
  \begin{equation}\label{eq:16}
    \lbrace \Fr^\Sigma_m(g)
      \rbrace_{m \geq 0,\, g\in I^c_S}
  \end{equation}
  with $I^c_S$ defined in~(\ref{eq:3})
  By formula~(\ref{eq:15}), the left adjoint $v$ of $u$ sends the
  generating cofibrations to the generating cofibrations. It follows that
  $v$ is a left Quillen functor. Since
  $v\colon \MS^\mathrm{Jar}(S)\to \MSS^\mathrm{Jar}(S)$ is a Quillen equivalence
  by \cite[Thm.~4.31]{J}, so is the functor
  $v\colon \MS^\cm(S)\to \MSS^\cm(S)$.
\end{proof}

\begin{remark}\label{rem:monoidal-sh}
  Let $\trder  u\colon \SH^\Sigma(S) : = \Ho\bigl(\MSS^\cm(S)\bigr)\to \SH(S)$
  be the total right derived functor of $u$, having $\tlder v$ as a
  left adjoint (and left inverse). Since $\tlder u$ is
  an equivalence, the category $\Ho(\MS^\cm(S))$ inherits a closed
  symmetric monoidal product $\wedge$ by setting
  \begin{equation*}
    E\wedge F \defeq \trder u\bigl(\tlder v (E)\wedge \tlder v(F)\bigr)
  \end{equation*}
  In other words, if $E$ and $F$ are closed cofibrant
  $\Pro^1$-spectra, their smash product in $\SH(S)$ is given by
  the $\Pro^1$-spectrum $u\bigl( (v(E)\wedge v(F))^\fib\bigr)$.
  The unit is $\trder u \bigl(\tlder v (\one)\bigr)\iso \one$,
  the sphere $\Pro^1$-spectrum.
\end{remark}

\begin{notation}\label{not:bigrading}
  For $\Pro^1$-spectra $E$ and $F$ over $S$ define
  the $E$-cohomology and the $E$-homology of $F$ as
  \begin{align}
    E^{p,q}(F) & = \Hom_{\SH(S)}(F,S^{p,q}\wedge E) \label{eq:26}\\
    E_{p,q}(F) & = \Hom_{\SH(S)}(S^{p,q},F\wedge E) \label{eq:27}
  \end{align}
  for all $p,q\in \mathbb{Z}$.
  In the special case $F = \Sigma^\infty_{\Pro^1} A$, where $A$
  is a pointed motivic space over $S$ one writes $E^{p,q}(A)$
  and $E_{p,q}(A)$ instead. Note that
  there is an isomorphism $\Fr_i A \iso S^{-2i,-i} \wedge \Fr_0 A $
  in $\SH(S)$.
\end{notation}

\begin{remark}\label{rem:u-v-triangulated}
  Since $(v,u)$ is a Quillen adjoint pair of stable model categories
  the total derived pair respects in particular the triangulated
  structures. The functor $u$ preserves all colimits,
  thus both $\tlder v$ and $\trder u$ preserve arbitrary
  coproducts.
\end{remark}

\begin{lemma}\label{lem:hocolim-smash}
  Let $E$ and $F$ be $\Pro^1$-spectra. Then $E\wedge F\in \SH(S)$
  may be obtained as the sequential colimit of a sequence whose
  $n$-th term is stably equivalent to
  $\Omega^{2n}_{\Pro^1} (\Sigma^\infty_{\Pro^1} E_n \wedge F_n)^\fib$.
  Thus
  $E \wedge F \cong \textrm{hocolim} \Sigma^\infty_{\Pro^1} (E_n \wedge F_n)(-2n)$
  in $\SH(S)$. 
\end{lemma}

\begin{proof}
  Here $E\wedge F \in \SH(S)$ refers to the smash product of
  symmetric $\Pro^1$-spectra associated to closed
  cofibrant replacements $E^\cof \to E$ and $F^\cof \to F$. Because these
  two maps are levelwise weak equivalences,
  we may assume that both $E$ and $F$ are closed cofibrant.
  By Lemma~\ref{lem:spectrum-colim} $E$ and $F$ can be
  expressed as sequential colimits of their truncations. Since $v$
  preserves colimits, $v(E)$ is the sequential colimit of the
  diagram
  \begin{equation*}
    v(\Tr_0 E) = v(\Fr_0 E_0) = \Fr^\Sigma_0 E_0 \to
    v(\Tr_1 E) \to \dotsm
  \end{equation*}
  and similarly for $F$. The stable equivalence $\Fr_n E_n \to \Tr_n E$
  of cofibrant $\Pro^1$-spectra induces a stable equivalence
  $v(\Fr_n E_n) = \Fr^\Sigma_n E_n \to v (\Tr_n E)$. Since
  smashing with a symmetric $\Pro^1$-spectrum preserves colimits,
  one has
  \begin{equation*}
    v(\Tr_m E)\wedge \colim_n v(\Tr_n F) \iso \colim_n \bigl( v(\Tr_m E)
    \wedge v(\Tr_n F)\bigr)
  \end{equation*}
  for every $n$. It follows that $v(E)\wedge v(F)$ is the filtered colimit
  of the diagram sending $(m,n)$ to $v(\Tr_m E) \wedge v(\Tr_n F)$. Since
  the diagonal is a final subcategory in $\mathbb{N}\times\mathbb{N}$,
  there is a canonical isomorphism $\colim_n v(\Tr_n E)\wedge v(\Tr_n F) \iso
  v(E)\wedge v(F)$. Theorem~\ref{thm:quillen-eq} says that $\MSS^\cm(S)$
  is symmetric monoidal, thus the canonical map
  \begin{equation}\label{eq:19}
    \Fr^\Sigma_{2n} (E_n\wedge F_n) \iso \Fr^\Sigma_n E_n \wedge^\Sigma_n E_n
    \to v(\Tr_n E)\wedge v(\Tr_n F)
  \end{equation}
  is a stable equivalence. Let
  $\Fr^\Sigma_{2n} (E_n \wedge F_n) \to \Omega^{2n}_{\Pro^1} \Fr^\Sigma_0
  (E_n \wedge F_n)$ be the canonical map which is adjoint to the unit
  $E_n \wedge F_n \to \Omega^{2n}_{\Pro^1} (\Pro^1)^{\wedge 2n} \wedge
  (E_n \wedge F_n)$.
  As in the case of $\Pro^1$-spectra, the map
  \begin{equation*}\Fr^\Sigma_{2n} (E_n \wedge F_n) \to \Omega^{2n}_{\Pro^1} \Fr^\Sigma_0
  (E_n \wedge F_n) \to \Omega^{2n}_{\Pro^1} \bigl(\Fr^\Sigma_0
  (E_n \wedge F_n)\bigr)^\fib \end{equation*}
  is a stable equivalence. It follows that $v(E)\wedge v(F)$
  is the colimit of a sequence whose $n$-th term is stably equivalent
  to $\Omega^{2n}_{\Pro^1} \bigl(\Fr^\Sigma_0 (E_n \wedge F_n)\bigr)^\fib $.
  Hence it also follows that a fibrant replacement of
  $v(E)\wedge v(F)$ may be obtained as the colimit of a sequence of
  closed stably fibrant symmetric $\Pro^1$-spectra
  whose $n$-th term is stably equivalent
  to $\Omega^{2n}_{\Pro^1} \bigl(\Fr^\Sigma_0 (E_n \wedge F_n)\bigr)^\fib $.
  Since the forgetful functor $u$ preserves colimits, stable equivalences
  of closed stably fibrant symmetric $\Pro^1$-spectra and $\Omega_{\Pro^1}$,
  the $\Pro^1$-spectrum $u\bigl(v(E)\wedge v(F)\bigr)$ is the colimit
  of a sequence of $\Pro^1$-spectra whose $n$-th term is stably equivalent
  to $\Omega^{2n}_{\Pro^1} u \bigl(\bigl(\Fr^\Sigma_0
  (E_n \wedge F_n)\bigr)^\fib\bigr)$.
  The map \begin{equation*}\Fr_0 (E_n \wedge F_n) \to u
  \bigl(\bigl(v \Fr_0 (E_n \wedge F_n)\bigr)^\fib\bigr)\end{equation*}
  is a stable equivalence because $(v,u)$ is a Quillen equivalence~\ref{thm:quillen-eq},
  whence the result.
\end{proof}

As in the case of non-symmetric spectra, one may change the
suspension coordinate as in Remark~\ref{r:prominent-A}. If $A$ is a
pointed motivic space over $S$, let $\MSS_A(S)$ denote
the category of symmetric $A$-spectra over $S$.

\begin{lemma}\label{lem:change-susp-coord}
  A map $A\ra B$ in $\M_\bullet(S)$
  induces a strict symmetric monoidal
  functor $\MSS_A(S)\ra \MSS_B(S)$ having a
  right adjoint. If the map is a motivic weak
  equivalence of closed
  cofibrant pointed motivic spaces, this pair
  is a Quillen equivalence.
\end{lemma}

\begin{proof}
  This is quite formal. For a proof consider
  \cite[Thm.~9.4]{Hovey:stable-model}.
\end{proof}

Because the change of suspension coordinate functors are
lax symmetric monoidal, they preserve (commutative) monoid objects,
that is, (commutative) symmetric ring spectra. Recall
that a functor is lax symmetric monoidal if it is
equipped with natural structure maps as in~(\ref{eq:strict})
which are not necessarily isomorphisms.

\subsection{Stable topological realization}
\label{sec:stable-topol-real}

Let $\Sp = \Sp(\Top, \mathbb{C}\Pro^1)$ be the category of $\mathbb{C}\Pro^1$-spectra
(in $\Top$). An object in $\Sp$ is thus a sequence of pointed compactly
generated topological spaces $E_0,E_1,\dotsc$ with structure maps
$E_n \wedge \mathbb{C}\Pro^1 \to E_{n+1}$. The model structure on $\Sp$ is obtained
as follows: Cofibrations are generated by
\begin{equation*}
  \lbrace \Fr_m^\Top(|\partial \Delta^n \hra \Delta^n|_+)\rbrace_{m,n\geq 0}
\end{equation*}
so that every $\mathbb{C}\Pro^1$-spectrum $E$ has a cofibrant replacement
$E^\cofrepl \to E$ mapping to $E$ via a levelwise acyclic Serre fibration.
For any $\mathbb{C}\Pro^1$-spectrum $E$ and any $n\in \mathbb{Z}$ let
$\pi_nE$ be the colimit of the sequence
\begin{equation*}
\pi_{n+2m} E_m \to \pi_{n+2m+2}
E_m \wedge \mathbb{C}\Pro^1 \to \pi_{n+2(m+1)} E_{1+m}  \to \dotsm
\end{equation*}
where $m\geq 0$ and $n+2m\geq 0$. It is called the
$n$-th stable homotopy group of $E$. Note that homotopy groups
of non-degenerately based compactly generated topological spaces commute with
filtered colimits. If $E$ is cofibrant, $E_n$ is in particular
non-degenerately based for all $n\geq 0$.
A map $f\colon E\to F$ of $\mathbb{C}\Pro^1$-spectra is a {\em stable equivalence\/} if
the induced map $\pi_n f\colon \pi_n E^\cofrepl \to \pi_n F^\cofrepl$ is
an isomorphism for all $n\in \mathbb{Z}$. It is a {\em stable fibration\/}
if it has the right lifting property with respect to all stable acyclic
cofibrations.

Similarly, one may form the category $\Sp^\Sigma = \Sp^\Sigma(\Top,\mathbb{C}\Pro^1)$
of symmetric $\mathbb{C}\Pro^1$-spectra in $\Top$. Cofibrations are generated by
\begin{equation*}
  \lbrace \Fr_m^{\Top,\Sigma}
   (|\partial \Delta^n \hra \Delta^n|_+)\rbrace_{m,n\geq 0}
\end{equation*}
and a symmetric $\mathbb{C}\Pro^1$-spectrum is {\em stably fibrant\/} if
its underlying $\mathbb{C}\Pro^1$-spectrum is stably fibrant.
A map $f\colon E\to F$ of symmetric $\mathbb{C}\Pro^1$-spectra is a
{\em stable equivalence\/} if
the induced map $\sSet_{\Sp^\Sigma}(f^\cofrepl,G)$ of simplicial
sets of maps is
an isomorphism for all stably fibrant symmetric $\mathbb{C}\Pro^1$-spectra $G$.
It is a {\em stable fibration\/}
if it has the right lifting property with respect to all stable acyclic
cofibrations.

\begin{theorem}\label{thm:cpone-spectra}
  Stable equivalences, stable fibrations and cofibrations form
  (symmetric monoidal) model structures on the categories of (symmetric)
  $\mathbb{C}\Pro^1$-spectra in $\Top$.
  The functor forgetting the symmetric group actions is a right
  Quillen equivalence. There is a zig-zag of strict symmetric monoidal
  left Quillen functors connecting $\Sp^\Sigma(\Top,\mathbb{C}\Pro^1)$ and
  $\Sp^\Sigma(\Top,S^1)$. In particular,
  the homotopy category of (symmetric) $\mathbb{C}\Pro^1$-spectra is equivalent
  as a closed symmetric monoidal and triangulated category
  to the stable homotopy category.
\end{theorem}

\begin{proof}
  The statement about the model structures follows as in
  \cite{HSS} if one replaces $S^1$ by $\mathbb{C}\Pro^1$ and
  simplicial sets by compactly generated topological spaces.
  The same holds for the statement about the functor forgetting
  the symmetric group actions. To construct the zig-zag,
  consider the corresponding stable model structure on
  the category of symmetric $S^1$-spectra in the category
  $\Sp^\Sigma(\Top,\mathbb{C}\Pro^1)$, which is isomorphic
  as a symmetric monoidal model category to
  the category of symmetric $\mathbb{C}\Pro^1$-spectra in the
  category $\Sp^\Sigma(\Top,S^1)$ of topological symmetric
  $S^1$-spectra. The suspension spectrum functors
  give a zig-zag
  \begin{equation}\label{eq:21}
    \xymatrix@R=0.7cm{
      \Sp^\Sigma(\Top,\mathbb{C}\Pro^1) \ar[d] & \Sp^\Sigma(\Top,S^1) \ar[d] \\
    \Sp^\Sigma\bigl(\Sp^\Sigma(\Top,\mathbb{C}\Pro^1),S^1\bigr)
    \ar[r]^\iso & \Sp^\Sigma\bigl(\Sp^\Sigma(\Top,S^1)\mathbb{C}\Pro^1\bigr)}
  \end{equation}
  of strict symmetric monoidal left Quillen functors.
  Since $\mathbb{C}\Pro^1\wedge -$ is a left Quillen equivalence on
  the left hand side in the zig-zag~(\ref{eq:21})
  and $S^1\wedge S^1\iso \mathbb{C}\Pro^1$,
  $S^1\wedge -$ is a left Quillen equivalence on the left hand side
  as well. By \cite[Thm.~9.1]{Hovey:stable-model}, the arrow
  pointing to the right in the zig-zag~(\ref{eq:21}) is a Quillen
  equivalence. A similar argument works for the arrow on the right
  hand side, which completes the proof.
\end{proof}

Given a $\Pro^1$-spectrum $E$ over $\CC$, one gets a
$\mathbb{C}\Pro^1$-spectrum $\real_\CC(E) = (\real_\CC E_0,\real_\CC E_1 ,\dotsc)$ with
structure maps
$\real_\CC(E_n)\wedge  \Pro^1 \iso
\real_\CC(E_n \wedge \Pro^1)\to \real_\CC(E_{n+1})$.
The right adjoint for the resulting functor
$\real_\CC\colon \MS(\CC)\to \Sp$ is also obtained by a levelwise application
of $\Sing_\CC$. The same works for symmetric $\Pro^1$-spectra over $\CC$.

\begin{theorem}\label{thm:complex-real}
  The functors $\real_\CC\colon \MS(\CC)\to \Sp$ and
  $\real_\CC\colon \MSS(\CC)\to \Sp^\Sigma$ are left Quillen
  functors, the latter being strict symmetric monoidal.
\end{theorem}

\begin{proof}
  Since the diagrams
  \begin{equation*}\xymatrix{
    \MSS(\CC) \ar[r]^{u} \ar[d]_{\Sing_\CC} &
      \MS(\CC) \ar[r]^{E\mapsto E_n} \ar[d]_{\Sing_\CC} &
       \M_\bullet(\CC) \ar[d]^{\Sing_\CC} \\
       \Sp^\Sigma \ar[r]^{u} & \Sp \ar[r]^{E\mapsto E_n} & \Top_\bullet}\end{equation*}
  commute, $\real_\CC$ preserves the generating cofibrations
  by~\ref{thm:realization}.
  Then Dugger's Lemma~\cite[Cor.~A.2]{Dugger} implies that
  $\real_\CC$ is a left Quillen functor, because
  $\Sing_\CC$ preserves weak equivalences and fibrations between
  fibrant objects. The fact that $\real_\CC\colon \MSS(\CC)\to \Sp^\Sigma$
  is strict symmetric monoidal follows from the definition of
  the smash product~(\ref{eq:12}).
\end{proof}

\begin{example}\label{example:stable-real-bgl}
  Let $\mathrm{BGL}$ be the $\Pro^1$-spectrum over $\mathbb{C}$ constructed
  in~\ref{ConstructionOfBGL}. Its $n$-th term is a
  pointed motivic space $\mathcal{K}$ weakly equivalent to $\ZZ\times \Gr$.
  One may assume that $\mathcal{K}$ is closed cofibrant. Then
  by Theorem~\ref{thm:complex-real}
  the $n$-th term of $\real_\CC(\mathrm{BGL})$ is weakly equivalent
  to $\mathrm{B}U$. To show that the $\CC\Pro^1$-spectrum
  $\real_\CC(\mathrm{BGL})$ is the one representing complex $K$-theory,
  it suffices to check that the structure map
  $\mathcal{K}\wedge \Pro^1\ra \mathcal{K}$ realizes to the
  structure map $\mathrm{B}U \wedge \CC\Pro^1 \ra \mathrm{B}U$ of
  complex $K$-theory. Consider the diagram
  \begin{equation*}\xymatrix{
    \Hom_{\mathrm{H}_\bullet(\CC)}(\ZZ\times \Gr,\ZZ\times \Gr) \ar[r]^-\iso \ar[d]&
    K_0^\mathrm{\mathrm{alg}}(\ZZ\times \Gr)\ar[d] \\
        \Hom_{\Ho(\Top_\bullet)}\bigl(\real_\CC(\ZZ\times \Gr),\real_\CC(\ZZ\times \Gr)
     \bigr) \ar[r]^-\iso &
    K_0^\mathrm{top}\bigl(\real_\CC(\ZZ\times \Gr)\bigr)}\end{equation*}
  where the vertical map on the left hand side is
  induced by $\real_\CC$ and the vertical map on the right hand side
  is induced by the passage from algebraic to topological complex vector bundles.
  The upper horizontal isomorphism sends the identity to the
  class $\xi_\infty$ described in Remark~\ref{Useful} via
  tautological vector bundles over Grassmannians. The right vertical map
  sends $\xi_\infty$ to the class $\zeta_\infty$
  obtained via the corresponding
  tautological bundles, viewed as complex topological vector bundles.
  Since $\zeta_\infty$ is the image of
  $\id_{\real_\CC(\ZZ\times \Gr)}$ under the lower horizontal
  isomorphism, the diagram commutes at the identity.
  By naturality, it follows that the diagram
  \begin{equation*}
    \xymatrix{
    \Hom_{\mathrm{H}_\bullet(\CC)}(A,\ZZ\times \Gr) \ar[r]^-\iso \ar[d]&
    K_0^\mathrm{\mathrm{alg}}(A)\ar[d] \\
        \Hom_{\Ho(\Top_\bullet)}\bigl(\real_\CC(A),\real_\CC(\ZZ\times \Gr)
     \bigr) \ar[r]^-\iso &
    K_0^\mathrm{top}\bigl(\real_\CC(A)\bigr)}
  \end{equation*}
  commutes for every pointed motivic space $A$ over $\mathbb{C}$. In particular,
  the structure map of $\mathrm{BGL}$ which corresponds to
  $(\xi_\infty) \otimes ([\mathcal{O}(-1)]-[\mathcal{O}])$
  maps to the structure map of the complex $K$-theory spectrum, since it
  corresponds to the ``same'' class, viewed as a difference of
  complex topological vector bundles.
\end{example}

\begin{proposition}\label{prop:base-change-spectra}
  A morphism $f\colon S\rightarrow S^\prime$ of base schemes induces a strict
  symmetric monoidal left Quillen functor
  \begin{equation*}
    f^\ast\colon \MSS(S^\prime)\rightarrow \MSS(S)
  \end{equation*}
  such that $\bigl(f^\ast(E)\bigr)_n = f^\ast(E_n)$.
\end{proposition}

\begin{proof}
  The structure maps of $f^\ast(E)$ are defined via the canonical map
  \begin{equation*}\xymatrix@C=3em{f^\ast(E)_n\wedge \Pro^1_{S^\prime}
      \iso f^\ast(E_n \wedge \Pro^1_S)\ar[r]^-{f^\ast(\sigma_n^E)} & f^\ast(E_{n+1})
      =f^\ast(E)_{n+1}.}
  \end{equation*}
  It follows that $f^\ast$ has the functor
  $f_\ast$ as right adjoint, where $\bigl(f_\ast(E))_n = f_\ast(E_n)$ and
  \begin{equation*}
    \xymatrix{\sigma_n^{f_\ast E} = f_\ast E_n \wedge \Pro^1_{S^\prime} \ar[r] &
      f_\ast E_n \wedge f_\ast f^\ast \Pro^1_{S^\prime} \ar[r]^-\iso &
      f_\ast E_n \wedge f_\ast \Pro^1_S  \ar[d] \\
       & f_\ast E_{n+1} & f_\ast(E_n \wedge \Pro^1_S) \ar[l]_-{f_\ast \sigma^E_n}}
  \end{equation*}
  Theorem~\ref{thm:closed-motivic-model} and
  Dugger's lemma imply that $f^\ast$ preserves cofibrations and $f_\ast$
  preserves fibrations. Since $f^\ast\colon \M_\bullet(S^\prime)\ra \M_\bullet(S)$
  is strict symmetric monoidal and preserves all colimits, then so
  is $f^\ast \colon \MSS(S^\prime)\ra \MSS(S)$ by the definition of
  the smash product~(\ref{eq:12}).
\end{proof}

In particular, any complex point
$f\colon \Spec(\mathbb{C})\ra S$ of a base scheme $S$
induces a strict symmetric monoidal
left Quillen functor
\begin{equation}\label{eq:20}
  \MSS(S) \ra \MSS(\mathbb{C})\ra \Sp^\Sigma(\Top,\mathbb{C}\Pro^1)
\end{equation}
to the category of topological $\mathbb{C}\Pro^1$-spectra.

\section{Some results on K-theory}
\label{sec:products-k-theory}

\subsection{Cellular schemes}
\label{sec:cellular-schemes}

Suppose that
$S$ is a regular base scheme.
Recall that an $S$-cellular scheme is an $S$-scheme
$X$ equipped with a filtration
$X_0 \subset X_1 \subset \dots \subset X_n = X$
by closed subsets such that
for every integer
$i \geq 0$
the $S$-scheme
$X_i \smallsetminus X_{i-1}$
is a disjoint union of several copies of the
affine space
$\Aff^i_S$.
We do not assume that
$X$ is connected.
{\it A pointed\/} $S$-{\it cellular scheme\/}
is an $S$-cellular scheme equipped with
a closed $S$-point
$x\colon S \hra X$
such that
$x(S)$ is contained in one of the open cells
(a cell which is an open subscheme of $X$).
The examples we are interested in are
Grassmannians, projective lines and their products.

\begin{lemma}
\label{KofXsmashY}
Let $(X,x)$ and $(Y,y)$ be pointed motivic spaces.
Then the sequence
\begin{equation*}
0 \to K^{}_i(X \wedge Y) \to K^{}_i(X \times Y) \to K^{}_i(X \vee Y) \to 0
\end{equation*}
is short exact and the natural map
\begin{equation*}
K^{}_i(X) \oplus K^{}_i(Y) \to  K^{}_i(X \vee Y)
\end{equation*}
is an isomorphism.
\end{lemma}

\begin{proof}
  The exactness of the sequence follows from the isomorphism
  $K_i \iso \mathrm{BGL}^{-i,0}$
  and Example~\ref{example:smash-product}.
  The isomorphism is formal, given the isomorphism
  $K_i \iso \mathrm{BGL}^{-i,0}$.
\end{proof}

\begin{corollary}
\label{AtensorB}
Let
$(X,x)$ and $(Y,y)$
be pointed smooth $S$-schemes.
Let
$a \in K^{}_0(X)$
and
$b \in K^{}_0(Y)$
be such that
$x^\ast(a)=0=y^\ast(b)$
in
$K^{}_0(S)$.
Then the element
$a \otimes b \in K^{}_0(X \times Y)$
belongs to the subgroup
$K^{}_0(X \wedge Y)$.
\end{corollary}

\begin{proof}
  Since
  $a \otimes b$
  vanishes on
  $x(S) \times Y$
  and on
  $X \times x(S)$
  it follows that
  $a \otimes b$
  vanishes on
  $X \vee Y$.
  Whence
  $a \otimes b \in K^{}_0(X \wedge Y)$
  by Lemma~\ref{KofXsmashY}.
\end{proof}

We list further useful statements.

\begin{lemma}
\label{KofCellularVar}
Let $X$ be a smooth $S$-cellular scheme. Then the map
$$
K^{}_r(S) \otimes_{K^{}_0(S)} K^{}_0(X) \to K^{}_r(X)
$$
is an isomorphism and
$K^{}_0(X)$
is a free
$K^{}_0(S)$-module of rank equal to the number of cells.
\end{lemma}

The Lemma easily follows from a slightly different claim which we consider
as a well-known one.
\begin{claim}
\label{KofCellularScheme}
Under the assumption of the Lemma the map of Quillen's $K$-groups
$$
K_r(S) \otimes_{K_0(S)} K^{\prime}_0(X_j) \to K^{\prime}_r(X_j)
$$
is an isomorphism and
$K^{\prime}_0(X_j)$
is a free
$K_0(S)$-module of the expected rank.
\end{claim}

\begin{lemma}
\label{KofWedge}
Let
$(X,x)$ and $(Y,y)$
be pointed smooth $S$-cellular schemes.
Then the map
$$
K^{}_i(S) \otimes_{K^{}_0(S)} K^{}_0(X \vee Y) \to K^{}_i(X \vee Y)
$$
is an isomorphism and
$K^{}_0(X \vee Y)$
is a projective
$K^{}_0(S)$-module.
\end{lemma}

\begin{lemma}
\label{KofSmash}
Let
$(X,x)$ and $(Y,y)$
be pointed smooth $S$-cellular schemes.
Then the map
$$
K^{}_i(S) \otimes_{K^{}_0(S)} K^{}_0(X \wedge Y) \to K^{}_i(X \wedge Y)
$$
is an isomorphism and
$K^{}_0(X \wedge Y)$
is a projective
$K^{}_0(S)$-module.
\end{lemma}

\begin{proof}
Consider the commutative diagram
\begin{equation*}
\xymatrix{
K_i(X \wedge Y) \ar[r]^-{\alpha}  &
K_i(X \times Y) \ar[r]^-{\beta}  &
K_i(X \vee Y)  \\
K_i(S) \otimes K_0(X \wedge Y) \ar[r]^-{\gamma} \ar[u]^-{\epsilon}&
K_i(S) \otimes K_0(X \times Y) \ar[r]^-{\delta}\ar[u]_-{\rho} &
K_i(S) \otimes K_0(X \vee Y)\ar[u]_-{\theta}}
\end{equation*}
in which
$K_i$ is written for
$K^{}_i$
and the tensor product is taken over
$K^{}_0(S)$.

The sequence
\begin{equation*}
0 \to K_i(X \wedge Y) \xra{\alpha} K_i(X \times Y) \xra{\beta} K_i(X \vee Y) \to 0
\end{equation*}
is short exact by Lemma~\ref{KofXsmashY}.
In particular it is short exact for $i=0$.
Now the sequence
\begin{equation*}
0 \to
K_i(S) \otimes K_0(X \wedge Y)
\xra{\gamma} K_i(S) \otimes K_0(X \times Y)
\xra{\delta}
K_i(S) \otimes K_0(X \vee Y)
\to 0
\end{equation*}
is short exact since
$K_0(X \vee Y)$
is a projective
$K_0(S)$-module.

The arrows
$\rho$ and $\theta$
are isomorphisms by Lemmas~\ref{KofCellularVar}
and~\ref{KofWedge}
respectively.
Whence
$\epsilon$
is an isomorphism as well.
Finally
$K_0(X \wedge Y)$
is a projective
$K_0(S)$-module
since the sequence
$0 \to K_0(X \wedge Y) \to K_0(X \times Y) \to K_0(X \vee Y) \to 0$
is short exact and
$K_0(X \times Y)$
and
$K_0(X \vee Y)$
are projective
$K_0(S)$-modules.
\end{proof}

As well we need to know that certain
${\varprojlim}^1$-groups vanish.
Given a set $M$ and a smooth $S$-scheme $X$, we write
$M \times X$ for the disjoint union $\bigsqcup_M X$ of $M$
copies of $X$ in the category of motivic spaces over $S$.
Recall that $[-n,n]$
is the set of integers with absolute value $\leq n$.

\begin{lemma}
\label{LimOneVanish}
\begin{equation*}
{\varprojlim}^1K^{}_i(\mathrm{Gr}(n,2n))=0
\end{equation*}
\begin{equation*}
{\varprojlim}^1K^{}_i([-n,n] \times \mathrm{Gr}(n,2n))=0
\end{equation*}
\begin{equation*}
{\varprojlim}^1K^{}_i([-n,n] \times \mathrm{Gr}(n,2n) \times [-n,n] \times \mathrm{Gr}(n,2n))=0
\end{equation*}
\end{lemma}

\begin{proof}
  This holds since all the bondings map in the towers are surjective.
\end{proof}

\begin{lemma}
\label{KTTofGr}
The canonical maps
\begin{equation*}
K^{}_i(\Gr) \to {\varprojlim}K^{}_i\bigl(\mathrm{Gr}(n,2n)\bigr)
\end{equation*}
\begin{equation*}
K^{}_i(\mathrm{Gr} \times \mathrm{Gr}) \to {\varprojlim}K^{}_i\bigl(\mathrm{Gr}(n,2n) \times \mathrm{Gr}(n,2n)\bigr)
\end{equation*}
are isomorphisms. A similar statement holds for the pointed motivic spaces
$\mathbb{Z} \times \mathrm{Gr}$,
$(\mathbb{Z} \times \mathrm{Gr}) \times \Pro^1$,
$(\mathbb{Z} \times \mathrm{Gr}) \times (\mathbb{Z} \times \mathrm{Gr})$,
$(\mathbb{Z} \times \mathrm{Gr}) \times \Pro^1 \times (\mathbb{Z} \times \mathrm{Gr}) \times \Pro^1$.
\end{lemma}

\begin{proof}
  This follows from Lemma~\ref{LimOneVanish}.
\end{proof}

\begin{lemma}
\label{KofGrsmashGr}
The canonical maps
\begin{equation*}
K^{}_i(\mathrm{Gr} \wedge \mathrm{Gr}) \to {\varprojlim}K^{}_i\bigl(\mathrm{Gr}(n,2n) \wedge \mathrm{Gr}(n,2n)\bigr)
\end{equation*}
\begin{equation*}
K^{}_i(\mathrm{Gr} \wedge \mathrm{Gr} \wedge \Pro^1) \to {\varprojlim}K^{}_i\bigl(\mathrm{Gr}(n,2n) \wedge \mathrm{Gr}(n,2n) \wedge \Pro^1 \bigr)
\end{equation*}
are isomorphisms.
A similar statement holds for the pointed motivic spaces
$(\mathbb{Z} \times \mathrm{Gr}) \wedge (\mathbb{Z} \times \mathrm{Gr})$,
$(\mathbb{Z} \times \mathrm{Gr}) \wedge (\mathbb{Z} \times \mathrm{Gr}) \wedge \Pro^1$
and
$(\mathbb{Z} \times \mathrm{Gr}) \wedge \Pro^1 \wedge (\mathbb{Z} \times \mathrm{Gr}) \wedge \Pro^1$.
\end{lemma}

\begin{proof}
It follows immediately from Lemma
\ref{KTTofGr}
and Lemma
\ref{KofXsmashY}.
\end{proof}


\begin{thebibliography}{MMMM}
\label{References}

\bibitem[D]{Dugger}
\emph{D.~Dugger}.
Replacing model categories with simplicial ones.
Trans.~Amer. Math.~Soc. 353  (2001),  no.~12, 5003--5027 (electronic).

\bibitem[DR{\O}]{DRO:motivic}
\emph{B.~I.~Dundas, O.~R{\"o}ndigs, P.~A.~{\O}stv{\ae}r.}
Motivic functors. Doc.~Math. 8 (2003), 489--525.

\bibitem[DR{\O}2]{DRO:enriched}
\emph{B.~I.~Dundas, O.~R{\"o}ndigs, P.~A.~{\O}stv{\ae}r.}
Motivic functors. Doc.~Math. 8 (2003), 409--488.

\bibitem[FS]{FS}
\emph{E.~Friedlander, A.~Suslin.}
The spectral sequence relating algebraic ${K}$-theory to motivic cohomology.
Ann.~Sci.~\'Ecole Norm.~Sup.~(4)  35  (2002),  no.~6, 773--875.

\bibitem[GS]{Gepner-Snaith}
\emph{D.~Gepner, V.~Snaith.}
On the motivic spectra representing algebraic cobordism and algebraic
$K$-theory. Preprint, 20 pages. Available via
{\tt arXiv:0712.2817v1 [math.AG]}.

\bibitem[He]{Heller}
\emph{A.~Heller}.
Homotopy theories.  Mem.~Amer.~Math.~Soc.  71  (1988),  no.~383, vi+78 pp.

\bibitem[Hi]{Hirschhorn}
\emph{P.~S.~Hirschhorn}.
Model categories and their localizations. Mathematical Surveys
and Monographs, 99.
American Mathematical Society, Providence, RI, (2003). xvi+457 pp.

\bibitem[Ho]{Hovey:book}
\emph{M.~Hovey}.
Model categories. Mathematical Surveys and Monographs, 63.
American Mathematical Society, Providence, RI, (1999). xii+209 pp.

\bibitem[Ho2]{Hovey:stable-model}
\emph{M.~Hovey}.
Spectra and symmetric spectra in general model categories.
J.~Pure Appl.~Algebra  165  (2001), no. 1, 63--127.

\bibitem[HSS]{HSS}
\emph{M.~Hovey, B.~Shipley, J.~Smith}.
Symmetric Spectra. Journal of the AMS
Volume 13, Number 1 (1999), 149--208.

\bibitem[I]{Isaksen:flasque}
\emph{D.~Isaksen}.
Flasque model structures. $K$-Theory 36 (2005) 371--395.

\bibitem[J]{J}
\emph{J.~F.~Jardine.}
Motivic symmetric spectra.
Doc.~Math. 5 (2000), 445--553.

\bibitem[MS]{Milnor-Stasheff}
\emph{J.~Milnor, J.~Stasheff.}
Characteristic classes.
Ann.~of Math.~Studies. No.~76. Princeton University Press (1974).

\bibitem[MV]{MV}
\emph{F.~Morel, V.~Voevodsky.} $\Aff^1$-homotopy theory of schemes.
Publ.~IHES \ 90, (1999), 45--143 (2001).

\bibitem[PPR]{PPR:conner-floyd}
\emph{I.~Panin, K.~Pimenov, O.~R{\"o}ndigs.}
On the relation of Voevodsky's algebraic cobordism to
Quillen's $K$-theory. To appear in Inventiones Mathematicae (2008). 
Preprint version available via
{\tt arXiv:0709.4124v1 [math.AG]}.

\bibitem[PS]{PSorcoh}
\emph{I.~Panin.}  ({\it After I.~Panin and A.~Smirnov})
Oriented Cohomology Theories on Algebraic Varieties.
Special issue in honor of H.~Bass on his seventieth birthday. Part III.
$K$-Theory, 30 (2003), no.~3, 265--314.

\bibitem[PY]{PY}
\emph{I.~Panin, S.~Yagunov.} Rigidity for orientable functors.
J.~Pure and Appl.~Algebra, 172, (2002), 49--77.

\bibitem[Qu]{Qu3}
\emph{D.~Quillen.} Higher algebraic $K$-theory: I. Lect.~Notes Math.
341 (1973), 85--147.

\bibitem[R]{Riou}
\emph{J.~Riou.} Op\'erations sur la K-th\'eorie alg\'ebrique et
r\'egulateurs via la th\'eorie homotopique des sch\'emas.
Comptes Rendus Math\'ematique.
Acad\'emie des Sciences. Paris 344 (2007), 27--32.

\bibitem[Sw]{Sw}
\emph{R.~M.~Switzer.} Algebraic topology - homotopy and homology.
Springer-Verlag, 1975.

\bibitem[TT]{TT}
\emph{R.~Thomason, T.~Trobaugh.} \ Higher algebraic K-theory of schemes
and of derived categories.
The Grothendieck Festschift, 3, 247-436, Birkh\"auser, Boston, 1990.

\bibitem[V]{V1}
\emph{V.~Voevodsky.} $\mathbb{A}^1$-Homotopy theory.
Doc.~Math., Extra Volume~ICM 1998(I), 579--604.

\bibitem[W]{W}
\emph{F.~Waldhausen.} Algebraic ${K}$-theory of spaces.
Springer Lecture Notes in Mathematics 1126 (1985), 318--419.

\bibitem[We]{We}
\emph{Ch.~Weibel.} Homotopy ${K}$-theory. In
Algebraic K-theory and algebraic number theory,
Volume 83 of Contemp. Math., AMS, Providence (1989), 461--488.

\end{thebibliography}
\end{document}